\documentclass[11pt]{amsart}
\allowdisplaybreaks
\usepackage[english]{babel}
\usepackage[pdftex,textwidth=400pt,marginratio=1:1]{geometry}
\usepackage{amsfonts}
\usepackage[dvips]{graphics}
\usepackage[colorinlistoftodos]{todonotes}
\usepackage{amsmath}
\usepackage{amsthm}
\usepackage{amssymb}
\usepackage{bbm}
\usepackage{cancel}
\usepackage{color}
\usepackage[all,cmtip]{xy}
\usepackage{graphicx}
\usepackage{mathabx}
\usepackage{tikz-cd}
\usepackage{epigraph} 
\usepackage{hyperref}
\usepackage{stmaryrd}
\usepackage{enumitem}
\usepackage[mathscr]{eucal}
\usepackage{soul}
\usepackage{xfrac}

\newtheorem{theorem}{Theorem}[section]
\newtheorem{corollary}[theorem]{Corollary}
\newtheorem{proposition}[theorem]{Proposition}
\newtheorem{lemma}[theorem]{Lemma}
\newtheorem{conjecture}[theorem]{Conjecture}

\theoremstyle{definition}
\newtheorem{definition}[theorem]{Definition}
\newtheorem{example}[theorem]{Example}
\newtheorem{remark}[theorem]{Remark}

\DeclareFontFamily{OT1}{rsfs}{}
\DeclareFontShape{OT1}{rsfs}{n}{it}{<-> rsfs10}{}
\DeclareMathAlphabet{\curly}{OT1}{rsfs}{n}{it}
\newcommand{\intprod}{\mathbin{\raisebox{\depth}{\scalebox{1}[-1]{$\lnot$}}}}

\newcommand{\bsix}{\underline{\mathbf{6}}}
\newcommand{\bfour}{\underline{\mathbf{4}}}
\newcommand{\bthree}{\underline{\mathbf{3}}}
\newcommand{\blambda}{\underline{\boldsymbol{\lambda}}}

\newcommand\LL{\mathbb L}

\renewcommand\O{\mathcal O}
\newcommand\PP{\mathbb P}

\newcommand\cV{\mathcal V}

\newcommand\cW{\mathcal W}

\newcommand\cE{\mathcal E}
\newcommand\EE{\mathbb E}
\newcommand\cA{\mathcal A}

\newcommand\cK{\mathcal K}

\newcommand\C{\mathbb C}

\newcommand\FF{\mathbb F}

\newcommand\sfZ{\mathsf Z}

\newcommand\Q{\mathbb Q}
\newcommand\cQ{\mathcal Q}

\newcommand\Z{\mathbb Z}

\newcommand\Coh{\mathrm{Coh}}

\newcommand\fr{\mathrm{fr}}
\newcommand\main{\mathrm{main}}

\newcommand\bfthree{\mathbf{3}}

\newcommand\Quot{\mathrm{Quot}}

\newcommand\SU{\mathrm{SU}}
\newcommand\Exp{\mathrm{Exp}}

\newcommand\vd{\mathrm{vd}}
\newcommand\pt{\mathrm{pt}}
\newcommand\vir{\mathrm{vir}}
\newcommand{\red}{\mathrm{red}}

\newcommand\loc{\mathrm{loc}}

\newcommand\td{\mathrm{td}}
\newcommand\rk{\operatorname{rk}}

\newcommand\tr{\operatorname{tr}}

\newcommand\im{\operatorname{im}}
\newcommand\ch{\operatorname{ch}}

\newcommand\id{\operatorname{id}}

\newcommand\Hom{\operatorname{Hom}}
\renewcommand\hom{\mathcal{H}{\it{om}}}
\newcommand\Ext{\operatorname{Ext}}
\newcommand\ext{\curly Ext}

\newcommand\Hilb{\operatorname{Hilb}}

\newcommand\Sym{\operatorname{Sym}}

\newcommand\mdot{{\scriptscriptstyle\bullet}}

\newcommand\INTO{\ar@{^{(}->}[r]}

\DeclareRobustCommand{\SkipTocEntry}[4]{}

\def\Supp{\mathrm{Supp}}

\def\udot{^{\mdot}}

\def\cM{\mathcal{M}}

\def\Quot{\mathrm{Quot}}

\def\GL{\mathrm{GL}}

\def\TT{\mathbb{T}}

\def\cH{\mathcal{H}}

\def\cU{\mathcal{U}}

\def\and{\quad\mathrm{and}\quad}

\def\Map{\mathsf{Map}}

\makeatletter
\def\l@section{\@tocline{1}{0pt}{1.5em}{1.5em}{}}
\def\l@subsection{\@tocline{2}{0pt}{3.2em}{3.2em}{}}
\makeatother

\begin{document}

\author[N.~Arbesfeld]{Noah Arbesfeld}
\address{University of Vienna, Faculty of Mathematics}
\email{noah.arbesfeld@univie.ac.at}

\author[M.~Kool]{Martijn Kool}
\address{Utrecht University, Mathematical Institute}
\email{m.kool1@uu.nl}

\author[W.~Lim]{Woonam Lim}
\address{Yonsei University, Department of Mathematics}
\email{woonamlim@yonsei.ac.kr}

\title[The geometry of Nekrasov's gauge origami theory]{The geometry of Nekrasov's gauge origami theory}

\begin{abstract}
Nekrasov’s gauge origami theory provides a (complex) 4-dimensional generalization of the ADHM quiver and its moduli spaces of representations.  We describe the origami moduli space as the zero locus of an isotropic section of a quadratic vector bundle on a smooth space. This allows us to give an algebro-geometric definition of the origami partition function in terms of Oh--Thomas virtual cycles. The key input is the computation of a sign associated to each torus fixed point of the moduli space. Furthermore, we establish an integrality result and dimensional reduction formulae, and discuss an application to non-perturbative Dyson--Schwinger equations following Nekrasov's work. Finally, we conjecture a description of the origami moduli space in terms of certain 2-dimensional framed sheaves on $\mathbb{P}^1 \times \mathbb{P}^1 \times \mathbb{P}^1 \times \mathbb{P}^1$, which we verify at the level of torus fixed points. 
\end{abstract}

\maketitle

\tableofcontents

\section{Introduction} \label{sec:intro}

The (2D) ADHM quiver $Q_{2}$, as given in Figure \ref{fig1} below, and its representations play a central role in geometry, representation theory and physics. The relation
\begin{equation} \label{eqn:ADHM}
[B_1, B_2] + I J = 0
\end{equation}
is sometimes referred to as the ADHM equation, which first appeared in the work of Atiyah--Drinfeld--Hitchin--Manin on constructions of $\SU(r)$ instantons on $S^4$ \cite{ADHM}.
Fixing dimension vector $(r,n)$, a representation $(B_1,B_2,I,J)$ of $Q_2$ is said to be stable if
$$\C \langle B_1,B_2 \rangle \cdot I(W) = V,$$
where $V:=\C^n$ and $W:=\C^r$. The moduli space $M_{Q_2}(r,n)$ of stable representations of $Q_2$ is a smooth quasi-projective variety of dimension $2rn$. The variety $M_{Q_2}(r,n)$ has an action by an algebraic torus $\mathbb{T} = (\C^*)^2 \times (\C^*)^r$, where $(\C^*)^2$ scales the loops and the action of $(\C^*)^r$ is induced from the diagonal action on $W$. Although $M_{Q_2}(r,n)$ is non-compact, the fixed locus $M_{Q_2}(r,n)^{\mathbb{T}}$ is 0-dimensional and reduced. Integration of  $\mathbb{T}$-equivariant cohomology classes against the $\TT$-equivariant fundamental class is therefore well-defined \cite{AB}. Certain generating series known as Nekrasov partition functions can be defined mathematically by fixing the form of the integrand and summing over all $n$ \cite{Nek1}.

\begin{figure}[ht]
\centering
\begin{tikzpicture}[>=stealth,->,shorten >=2pt,looseness=.5,auto]
  \matrix [matrix of math nodes,
           column sep={2cm,between origins},
           row sep={2cm,between origins}]
{ 
\node[draw, minimum size=7.5mm, shape=rectangle] (A) {r}; & & 
\node[draw, minimum size=7.5mm, shape=circle] (B) {n}; \\
};
\path[->] (B) edge [in=45,out=90,looseness=20]  node {$B_1$} (B);
\path[->] (B) edge [in=315,out=270,looseness=20]  node {} (B);
\node [anchor=west,right] at (-0.4,0.8) {$I$};  
\node [anchor=west,right] at (-0.4,-0.8) {$J$};  
\node [anchor=west,right] at (2.8,-2) {$B_2$}; 
\draw (A) to [bend left=15,looseness=1] (B) node [midway] {};
\draw (B) to [bend left=15,looseness=1] (A) node [midway,below] {};
\end{tikzpicture}
\vspace{-0.5cm}
\caption{The quiver $Q_2$.}\label{fig1}
\end{figure}
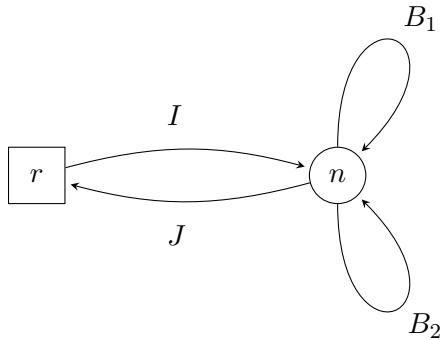

The space $M_{Q_{\mathrm{2}}}(r,n)$ admits a sheaf-theoretic description essentially due to Barth \cite{Bar}, see also \cite{Don, Nak, OSS}. Consider $\mathbb{P}^2$ with the line at infinity denoted by $\ell_{\infty}$. A rank $r$ framed sheaf on $\mathbb{P}^2$ with instanton number $n$ is a pair $(E,\phi)$, where $E$ is a rank $r$ torsion free sheaf on $\mathbb{P}^2$ satisfying $c_2(E) = n$ and $\phi \colon E|_{\ell_{\infty}} \rightarrow \O_{\ell_{\infty}}^{\oplus r}$ is an isomorphism. The moduli space of framed sheaves, denoted by $M_{\mathbb{P}^2}(r,n)$, admits a natural $\mathbb{T}$-action and there exists a $\mathbb{T}$-equivariant isomorphism
\begin{equation} \label{eqn:Barth}
M_{Q_2}(r,n) \cong M_{\PP^2}(r,n).
\end{equation}

In this paper we study the 4D ADHM quiver $Q_{\mathrm{4}}$ introduced by Nekrasov, which generalizes the 2D ADHM quiver \cite{Nek2, Nek3, Nek4, Nek5, Nek6}. The relevant moduli space of stable quiver representations is denoted by $M_{Q_4}(\vec{r},n)$ and referred to as the {\it origami moduli space}. Nekrasov associates to these moduli spaces a certain generating function of invariants known as the \emph{gauge origami partition function}. As we explain in Example \ref{ex:2dreductioncont}, the gauge origami partition function can be considered as a generalization of the local Vafa--Witten partition function to the setting of singular surfaces. Unlike the 2D ADHM theory, the origami moduli spaces are usually singular, making it difficult to understand the geometric meaning of the partition function and give a sheaf-theoretic description of it. Therefore, we ask:
\begin{enumerate}\item[(1)] Does there exist a virtual cycle $[M_{Q_{4}}(\vec{r},n)]^{\vir}$ that gives rise to an algebro-geometric definition of the gauge origami partition function?\item[(2)]Does the origami moduli space $M_{Q_4}(\vec{r},n)$ admit a sheaf-theoretic description?
\end{enumerate} 

We solve (1) by realizing $M_{Q_4}(\vec{r},n)$ as the zero locus of an isotropic section of a quadratic vector bundle on the moduli space of stable representations of the 4D ADHM quiver without relations. This description provides us with a (globally defined) Oh--Thomas virtual cycle \cite{OT}, which we use to give an algebro-geometric definition of the gauge origami partition function. By carefully analyzing the orientations (signs) on the torus fixed points, we show that our algebro-geometric definition matches with Nekrasov's original definition, given in terms of Boltzmann weights. This allows us to reprove the so-called {\it non-perturbative Dyson--Schwinger equations}, following Nekrasov's work. We also give a conjectural answer to (2) in terms of certain 2-dimensional framed sheaves on $\PP^1 \times \PP^1 \times \PP^1 \times \PP^1$, which we verify at the level of fixed points.

For an extensive exploration of the (quantum) algebraic structures and the mathematical physics of gauge origami theory, we refer to the series of papers \cite{KN1,KN2,KN3,KN4} by Kimura--Noshita. A representation-theoretic construction of Nekrasov's $qq$-characters can be found in work of Liu \cite{Liu}. Instead, our focus is on the geometry of gauge origami theory.

\subsection{Quiver description}

Denote by $\widehat{Q}_4$ the (framed) quiver without relations from Figure \ref{fig2}. Following Nekrasov \cite{Nek4}, we introduce the following notation. We let $\bfour := \{1,2,3,4\}$ and denote by $\bsix$ the collection of subsets $A \subset \bfour$ such that $|A| = 2$. We also define $\bthree:=\{\{1,2\}, \{1,3\}, \{2,3\}\} \subset \bsix$. Furthermore, we write $\overline{A}$ for the unique element of $\bsix$ satisfying $A \cap \overline{A} = \varnothing$ and we define $\phi(A) = \min \overline{A}$. Then we consider the following ADHM equations 
\begin{align} 
\begin{split} \label{eqn:4dADHM}
&\mu_{A} := [B_a, B_b] + I_A J_A = 0, \quad \nu_A := J_{\overline{A}} I_A = 0, \\ 
&\mu_{A,\overline{a}} := B_{\overline{a}} I_A = 0, \quad \nu_{A,\overline{a}} := J_A B_{\overline{a}} = 0,
\end{split}
\end{align}
for all $A = \{a<b\} \in \bsix$ and $\overline{a} \in \overline{A}$. The 4D ADHM quiver $Q_4$ is obtained from $\widehat{Q}_4$ by imposing the relations \eqref{eqn:4dADHM}. We fix a dimension vector
\begin{equation} \label{eqn:dimvec}
(\vec{r}, n) = ( (r_A), n),
\end{equation}
where $(r_A) = (r_A)_{A \in \bsix}$. We define $V:= \C^n$, $W_A := \C^{r_A}$ and set 
\[
r := \sum_{A} r_A.
\] 
A representation $(B_a,I_A,J_A) = (B_a,I_A,J_A)_{a \in \bfour, A \in \bsix}$ of $Q_4$ is called stable when
\begin{equation} \label{eqn:ADHMstab}
\sum_{A} \C \langle B_1, B_2, B_3, B_4 \rangle \cdot I_A(W_A) = V.
\end{equation}

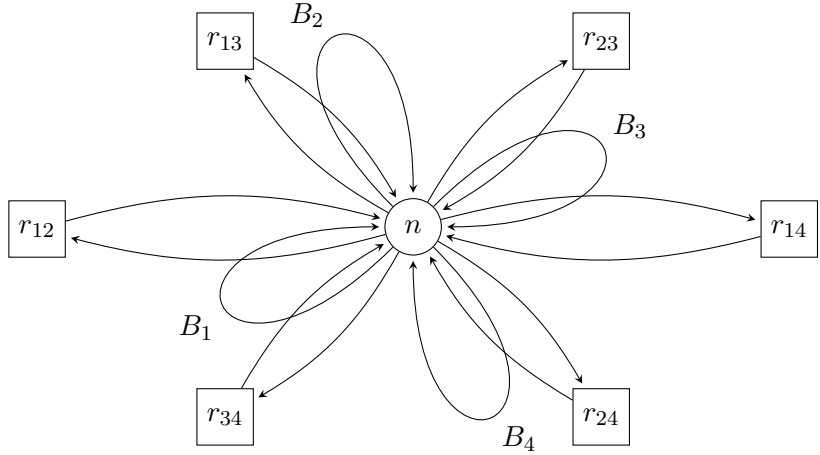
\begin{figure}[ht]
\centering
\begin{tikzpicture}[>=stealth,->,shorten >=2pt,looseness=.5,auto]
  \matrix [matrix of math nodes,
           column sep={2.5cm,between origins},
           row sep={2.5cm,between origins}]
{ 
 & \node[draw, shape=rectangle, minimum size=7.5mm] (D) {r_{13}}; &  & 
   \node[draw, shape=rectangle, minimum size=7.5mm] (E) {r_{23}}; &  \\
\node[draw, shape=rectangle, minimum size=7.5mm] (A) {r_{12}}; & & 
\node[draw, shape=circle, minimum size=7.5mm] (B) {n}; & & 
\node[draw, shape=rectangle, minimum size=7.5mm] (C) {r_{14}}; \\
 & \node[draw, shape=rectangle, minimum size=7.5mm] (F) {r_{34}}; &  & 
   \node[draw, shape=rectangle, minimum size=7.5mm] (G) {r_{24}}; &  \\
};

\path[->] (B) edge [in=0,out=45,looseness=30]  node {$B_3$} (B);
\path[->] (B) edge [in=90,out=135,looseness=30]  node {$B_2$} (B);
\path[->] (B) edge [in=180,out=225,looseness=30]  node {$B_1$} (B);
\path[->] (B) edge [in=270,out=315,looseness=30]  node {$B_4$} (B);

\draw (A) to [bend left=15,looseness=1] (B) node [midway] {};
\draw (B) to [bend left=15,looseness=1] (A) node [midway,below] {};
\draw (D) to [bend left=15,looseness=1] (B) node [midway] {};
\draw (B) to [bend left=15,looseness=1] (D) node [midway,below] {};
\draw (E) to [bend left=15,looseness=1] (B) node [midway] {};
\draw (B) to [bend left=15,looseness=1] (E) node [midway,below] {};
\draw (F) to [bend left=15,looseness=1] (B) node [midway] {};
\draw (B) to [bend left=15,looseness=1] (F) node [midway,below] {};
\draw (G) to [bend left=15,looseness=1] (B) node [midway] {};
\draw (B) to [bend left=15,looseness=1] (G) node [midway,below] {};
\draw (C) to [bend left=15,looseness=1] (B) node [midway] {};
\draw (B) to [bend left=15,looseness=1] (C) node [midway,below] {};
\end{tikzpicture}
\vspace{-0.5cm}
\caption{The quiver $Q_4$. The arrows from $r_A$ to $n$ are labelled by $I_A$ and the arrows from $n$ to $r_A$ are labelled by $J_A$.
}\label{fig2}
\end{figure}

There exists a quasi-projective scheme $\cM := M_{Q_4}(\vec{r},n)$ parametrizing stable representations of $Q_4$ with dimension vector $(\vec{r},n)$ (Proposition \ref{prop:King}). This scheme  is in general singular (Examples \ref{ex:butterfly n=1} and \ref{ex:book n=1}). It is a closed subscheme of the \emph{smooth} variety $\cA := A_{\widehat{Q}_4}(\vec{r},n)$ parametrizing stable representations of $\widehat{Q}_4$ with dimension vector $(\vec{r}, n)$. Nekrasov refers to elements of $\cM$ as \emph{spiked instantons}. Roughly speaking, elements of $\cM$ can be regarded as instantons on the union of the 6 coordinate 2-planes in $\C^4$.

For $r_{12} = r$ and $r_A = 0$ otherwise, the stable representations of $Q_4$ and $Q_2$ coincide and $M_{Q_4}(\vec{r},n) \cong M_{Q_2}(r,n)$ (Example \ref{ex:2dreduction}).

We construct an even rank quadratic vector bundle $(E,q) \to \mathcal{A}$, i.e., a vector bundle with a fibre-wise non-degenerate bilinear form $q \colon E \otimes E \to \O_{\cA}$. The bundle $(E,q)$ has the following properties. There exists an isotropic section $s \in H^0(\mathcal{A},E)$, that is, a section $s$ satisfying $q(s,s)=0$. Over a representation  $(B_a,I_A,J_A)$, the section $s$ is defined by $(\mu_A,\nu_A,\mu_{A,\overline{a}}, \nu_{A,\overline{a}})$, where $A$ runs over all elements of $\bsix$ and $\overline{a}$ over all elements of $\overline{A}$. The section $s$ cuts out $\mathcal{M}$, i.e., $\mathcal{M} = Z(s)$. This construction gives rise to a 3-term symmetric obstruction theory 
\[
\phi \colon E\udot_{Q_4} \to \tau^{\geq -1}\LL_{\cM}, \quad (E\udot_{Q_4})^\vee[2] \stackrel{\theta}{\cong} E\udot_{Q_4}, \quad \theta^\vee[2] = \theta
\]
by a standard construction which we review in Section \ref{sec:quivervirtualstruc}, and where $\tau^{\geq -1}\LL_{\cM}$ denotes the truncated cotangent complex of $\cM$. We furthermore construct a maximal isotropic subbundle $\Lambda \subset E$, which gives rise to an orientation on $\cM$, that is, an isomorphism
\begin{equation*} 
o \colon \det(E^\mdot_{Q_4}) \cong \O_{\cM}
\end{equation*}
with the property that $o \otimes o = (-1)^{\rk(\rk-1)/2} \det(\theta)$, where $\rk:=\rk(E)$. The precise choice of orientation is described in \eqref{eqn:choiceofori}. 

We consider the following algebraic tori
\begin{equation} \label{eqn:defT}
T:= Z(t_1t_2t_3t_4-1) \leq (\C^*)^4, \quad \mathbb{T} := T \times (\C^*)^r.
\end{equation}
We refer to $(\C^*)^r$ as the framing torus. As discussed in Section \ref{sec:quivermoduli}, there is a natural action of $\mathbb{T}$ on $\cA$, $E$, and $\Lambda$ with respect to which $s$ is $\TT$-equivariant and $q$ is $\TT$-invariant. As a consequence, the morphisms $\phi$ and $\theta$ are $\TT$-equivariant.  Moreover, the fixed point locus $\cM^{\TT}$ is 0-dimensional and reduced. By Oh--Thomas theory \cite{OT}\footnote{As we mostly work equivariantly and K-theoretically, we follow the approach of Oh--Thomas instead of those of Borisov--Joyce \cite{BJ} or Cao--Leung \cite{CL}.}, these data give rise to a virtual cycle and virtual structure sheaf
\[
[\cM]^{\vir} \in A_{\vd}^{\TT}(\cM,\Q), \quad \widehat{\O}_{\cM}^{\vir} \in K_0^{\TT}(\cM)_{\loc}, \quad \vd := - \sum_{A \in \bthree} r_A r_{\overline{A}},
\]
where $A_{*}^{\TT}(\cM,\Q)$ denote the $\TT$-equivariant (rational) Chow groups of $\cM$ and $K_0^\TT(\cM)_{\loc}$ denotes the (suitably localized) Grothendieck group of $\TT$-equivariant coherent sheaves on $\cM$.\footnote{As we work with \emph{$\TT$-equivariant} Chow groups, there are nontrivial classes in negative degrees.} These classes depend on the choice of orientation \eqref{eqn:choiceofori}. Note that the virtual dimension is independent of the instanton number $n$.

We define the \emph{cohomological} and \emph{K-theoretic gauge origami partition functions}
\begin{equation} \label{eqn:origamipartintro}
\sfZ_{\vec{r}}(q) := \frac{1}{\mathsf{C}_{\vec{r}}} \sum_{n=0}^\infty q^n \int_{[M_{Q_4}(\vec{r},n)]^{\vir}} 1, \quad 
\sfZ_{\vec{r}}^K(q) := \frac{1}{\mathsf{C}^K_{\vec{r}}}\sum_{n=0}^\infty q^n \chi(M_{Q_4}(\vec{r},n), \widehat{\O}^{\vir}),
\end{equation}
where 
\[
\mathsf{C}_{\vec{r}} := \int_{[M_{Q_{4}}(\vec{r},0)]^{\vir}} 1, \quad \mathsf{C}^K_{\vec{r}} := \chi(M_{Q_{4}}(\vec{r},0), \widehat{\O}^{\vir}),
\]
are normalization factors.\footnote{
    The cohomological normalization factor matches, up to sign, with the ``cross part'' of the perturbative prefactor in \cite[Eqn.~(111), (119)]{Nek4}.
}
The coefficients of the cohomological gauge origami partition function are rational functions in the $\TT$-equivariant parameters $s_1$,$s_2$,$s_3$,$s_4$ (modulo $s_1+s_2+s_3+s_4=0)$ and $(v_{A,\alpha})$, where $A$ runs over all elements of $\bsix$ and $\alpha$ over all elements of $\{1, \ldots, r_A\}$. The coefficients of the K-theoretic gauge origami partition function are rational functions in $t_1^{\frac{1}{2}}$,$t_2^{\frac{1}{2}}$,$t_3^{\frac{1}{2}}$,$t_4^{\frac{1}{2}}$ (modulo $t_1t_2t_3t_4=1$) and $(w_{A,\alpha}^{\frac{1}{2}})$, where $s_i = c_1^{\TT}(t_i)$ and $v_{A,\alpha} = c_1^{\TT}(w_{A,\alpha})$. In fact, the cohomological gauge origami partition function $\sfZ_{\vec{r}}(q)$  can be obtained from its K-theoretic analogue $\sfZ_{\vec{r}}^K(q)$  as a limit, see Lemma \ref{lem:lim 1}.
Consider  the $\TT$-equivariant K-group of a point 
\[
K_0^{\TT}(\pt) \cong \Z[t_1^{\pm 1},\ldots,t_4^{\pm 1}, (w_{A,\alpha}^{\pm 1})] / (t_1t_2t_3t_4 - 1).
\]
It contains elements
\[
t_A := t_a t_b, \quad P_a := 1-t_a, \quad P_A := (1-t_a)(1-t_b), \quad P_{1234} := \prod_{a=1}^4 (1-t_a),
\]
for any $A = \{a<b\} \in \bsix$. For any virtual character $\chi \in K_0^{\TT}(\pt)$, we write $\chi^*$ for its dual, obtained by replacing $t_a \mapsto 1/t_a$, $w_{A,\alpha} \mapsto 1/w_{A,\alpha}$. Furthermore, assuming $\chi$ has no $\TT$-fixed part with negative coefficient, we consider its Euler class $e(\chi)$ and symmetrized K-theoretic Euler class
\begin{equation} \label{eqn:bracket}
[\chi] := \widehat{\Lambda}_{-1} \chi^* := \frac{\Lambda_{-1} \chi^*}{\det(\chi^*)^{\frac{1}{2}}},
\end{equation}
where $\Lambda_{-1}(-) = \sum_i (-1)^i \Lambda^i (-)$. 
The fixed locus $M_{Q_4}(\vec{r},n)^{\TT}$ is indexed by collections of integer partitions (finite 2D partitions) of total size $n$, that is
\[
\blambda = (\lambda_{A,\alpha} ), \quad |\blambda| := \sum_{A, \alpha} |\lambda_{A,\alpha}| = n,
\]
where $|-|$ denotes the size. We view an integer partition $\lambda$ as a finite subset of $\Z_{\geq 0}^2$ satisfying the property that if $(i,j) \in \lambda$, $0 \leq i' \leq i$, and $0 \leq j' \leq j$ then $(i',j') \in \lambda$. To such a fixed point, we associate the following elements of $K_0^{\TT}(\pt)$ introduced by Nekrasov \cite[(90)]{Nek4}\footnote{Beware of a minor typo in loc.~cit.: $\phi(\overline{A})$ should be $\phi(A)$.}
\begin{align}
\begin{split} \label{eqn:Neknot}
\mathsf{v}_{\blambda} &:= \sum_{A} (P_{\phi(A)} T_A + P_{\overline{A}} N_A \sum_{B \neq A} K_B^*) - P_{1234} \sum_{A<B}K_A K_B^*, \\
T_A &:= N_A K_A^* + t_A N_A^* K_A - P_A K_A K_A^*, \\
K_{A,\alpha} &:=\sum_{(i,j) \in \lambda_{A,\alpha}} t_a^i t_b^j w_{A,\alpha}, \quad K_A := \sum_{\alpha} K_{A,\alpha}, \quad  \quad N_A := \sum_{\alpha} w_{A,\alpha},
\end{split}
\end{align}
where $A<B$ is defined by the lexicographical order (e.g.~$\{1,2\}$ is the smallest element). Note that $T_A$ is precisely the character of the tangent bundle of the classical instanton moduli  multiplied by a factor $t_A$. This implies that $\mathsf{v}_{\blambda}$ has no $\TT$-fixed terms.\footnote{Only the diagonal contributions of $T_A$ potentially have $\TT$-fixed terms. These are copies of the tangent space to fixed points of the Hilbert scheme of points of $\C^2$, up to a factor $t_A$. It is well-known that these tangent spaces contain no eigenspaces of the form $\C \cdot (t_1 t_2)^m$, from which the claim follows. See, e.g.,~\cite[Lem.~3.6]{CK} for a similar result for $\Hilb^n(\C^4)$.} 

Our first main result is that the invariants \eqref{eqn:origamipartintro}, globally defined in terms of Oh--Thomas virtual cycles, indeed produce Nekrasov's ``Boltzmann weights'' \cite[Eqn.~(89), (90)]{Nek4} with the correct sign choice.
\begin{theorem} \label{thm:main}
Fix $\vec{r} = (r_A).$ For any $n$, let $o$ be the orientation given by \eqref{eqn:choiceofori}. Then 
\[
\sfZ_{\vec{r}}(q) = \sum_{\blambda} q^{|\blambda|} e(-\mathsf{v}_{\blambda}), \quad \sfZ_{\vec{r}}^K(q) = \sum_{\blambda} q^{|\blambda|} [-\mathsf{v}_{\blambda}].
\]
\end{theorem}

The proof of this formula uses the Oh--Thomas virtual torus localization formula \cite{OT}. The main difficulty is to determine the orientations (signs) on the fixed points induced by our (global) choice of orientation \eqref{eqn:choiceofori}. In work by the second-named author and Rennemo \cite{KR}, the Hilbert schemes $\Hilb^n(\C^4)$ are also realized as zero loci of isotropic sections of quadratic vector bundles on a smooth ambient space. Although $Q_{4}$ is more complicated than the framed 4-loop quiver relevant to $\Hilb^n(\C^4)$, we are able to reduce the orientation calculation for $Q_{4}$ to the one for the framed 4-loop quiver. Surprisingly, after rewriting our results in terms of Nekrasov's character $\mathsf{v}_{\blambda}$ we find that the signs are all $+1$. 

Theorem \ref{thm:main} immediately implies the dimensional reduction formulae as follows. For $r_{12} = r$ and $r_A = 0$ otherwise, the origami partition function reduces to the generating series of $\widehat{\chi}_{-t_3}(M_{Q_2}(r,n))$ associated to the 2D ADHM quiver $Q_2$, where $\widehat{\chi}_{-y}(-)$ denotes the symmetrized Hirzebruch $\chi_{-y}$-genus (Example \ref{ex:2dreductioncont}). Moreover, for arbitrary $\vec{r}$ with $r_{14} = r_{24}= r_{34} = 0$, the origami partition function reduces to the generating series of $\chi(M_{Q_3}(\vec{r},n),\widehat{\mathcal{O}}^{\vir})$ associated to the 3D ADHM quiver $Q_3$ (Example \ref{ex:3dreductioncont}). Here $M_{Q_3}(\vec{r},n)$ admits a Behrend--Fantechi perfect obstruction theory induced by the presentation of $M_{Q_3}(\vec{r},n)$ as the critical locus of a regular function on a smooth ambient variety and $\widehat{\mathcal{O}}^{\vir}$ denotes its (Nekrasov--Okounkov twisted) Behrend--Fantechi virtual structure sheaf. The fact that $M_{Q_3}(\vec{r},n)$ has a description as a critical locus is due to Rap\v{c}\'ak--Soibelman--Yang--Zhao \cite{RSYZ}, who construct an action of the cohomological Hall algebra of $\C^3$ on the critical cohomology of these moduli spaces. We therefore see that the origami partition function contains invariants of holomorphic symplectic as well as $(-1)$-shifted symplectic moduli spaces as special cases. We also characterize a ``genuinely 4-dimensional'' partition function in terms of the plethystic exponential (defined in \eqref{def:pleth}).\footnote{Nekrasov refers to elements of $M_{Q_4}(\vec{r},n)$ with $r_{12}, r_{34} \neq 0$ and $r_A = 0$ otherwise as crossed instantons.}
\begin{theorem} \label{thm:crossinst}
For $r_{12} = r_{34} = 1$ and $r_A = 0$ otherwise, we have
\[
\sfZ_{\vec{r}}^K(q) = \mathrm{Exp} \Bigg( \frac{[t_1t_3][t_2t_3]}{[t_1][t_2]} \frac{q}{1-q} \Bigg) \mathrm{Exp} \Bigg( \frac{[t_1t_3][t_1t_4]}{[t_3][t_4]} \frac{q}{1-q} \Bigg).
\]
\end{theorem}
The proof of this theorem involves several steps. First, for arbitrary $M_{Q_4}(\vec{r},n)$, we establish that the Oh--Thomas virtual structure sheaf $\widehat{\O}^{\vir}$, which a priori is only a \emph{localized} K-theory class, lifts to an actual coherent sheaf and therefore has an \emph{integral} K-theory class (up to some equivariant parameters, Proposition \ref{prop:sheaflift}). Second, for $r_{12} = r_{34} = 1$ and $r_A=0$ otherwise, we use the fact that $M_{Q_4}(\vec{r},n)$ is isomorphic to a certain Quot scheme on $\C^4$ (Proposition \ref{prop: Quot scheme}), which allows us to push-forward to $\Sym^n(\C^4)$.  As a consequence, one can show that the origami partition function is independent of the equivariant framing parameters. Third, taking certain limits in the equivariant framing parameters then provides the desired splitting. For a recent work on lifts of the Oh--Thomas virtual structure sheaf in the context of spin structures, we refer to work of Kuhn \cite{Kuh}.

We introduce a co-stable analogue of the origami partition function (Section \ref{sec:stabcostab}) and we conjecture that the stable/co-stable wall-crossing formula is trivial. This implies a nontrivial symmetry of the origami partition function, i.e.,
$$
\mathsf{Z}_{\vec{r}}^K(q) \Big|_{(w_{A,\alpha}) \mapsto (t_A w_{A,\alpha}^{-1})} = \mathsf{Z}_{\vec{r}}^K(q),
$$
for which we give some computational evidence. This generalizes a theorem proved in \cite{AKL} in the case $r_{12} = r$ and $r_A = 0$ otherwise.

Closely following Nekrasov \cite{Nek2, Nek3}, another application of Theorem \ref{thm:main} is to the so-called non-perturbative Dyson--Schwinger equations. Suppose  $r_{12}, r_{34} > 0$ and $r_A = 0$ otherwise. Nekrasov's compactness theorem \cite[Sect.~8]{Nek3} asserts that for certain subtori $\mathbb{T}' \leq \TT$ given by the kernel of some $\TT$-characters, the fixed locus $M_{Q_4}(\vec{r},n)^{\TT'}$ is \emph{proper}. This implies that the (non-normalized, cohomological) origami partition function $\mathsf{C}_{\vec{r}}\cdot\mathsf{Z}_{\vec{r}}(q)$ has no poles in the equivariant parameters corresponding to these $\TT$-characters. The vanishing of these poles implies nontrivial equations for the ``classical'' Nekrasov partition function associated to the original 2D ADHM quiver $Q_2$. We discuss one of Nekrasov's examples in Section \ref{sec:DS}. We emphasize that this application relies on having a \emph{global} definition of the origami partition function, as provided in this paper.

Finally, we discuss connectedness of the origami moduli spaces. Connectedness of moduli spaces of quiver representations is a subtle problem in the presence of relations. In Theorem \ref{thm: connectedness of origami moduli space}, we prove that the moduli space $M_{Q_4}(\vec{r},n)$ is always connected, thus generalizing the connectedness of the usual 2D ADHM moduli spaces.

\subsection{Framed sheaves description}

Let $X$ be a smooth projective 4-fold. Let $D$ be an ample effective divisor on $X$, not necessarily smooth, such that $2D \in |-K_X|$. Let $F$ be a (perfect) 1-dimensional sheaf on $D$, called a framing sheaf, which is $\mu_D$-semistable. We consider framed sheaves $(E,\phi)$ consisting of a pure 2-dimensional sheaf $E$ on $X$ together with an isomorphism $\phi \colon E|_D \to F$. For a fixed topological type $v = (0,0,\gamma,*,*)\in H^{\mathrm{even}}(X,\Q)$, let $\underline{M}_{(X,D,F)}^{\fr}(v)$ be the moduli functor parametrizing families of framed sheaves $(E,\phi)$ on $X$ with $\ch(E) = v$. The following is (essentially) Theorem \ref{cor: CY4 obstruction theory} in Section \ref{sec:derstrucsheaves}. 

\begin{theorem} \label{thm:fr}
The moduli functor $\underline{M}_{(X,D,F)}^{\fr}(v)$ is represented by a quasi-projective scheme $\cM := M_{(X,D,F)}^{\fr}(v)$. Also, $\cM$ admits a 3-term symmetric obstruction theory 
$$
\phi_{\fr} \colon E\udot_{\fr} \to \tau^{\geq -1} \LL_{\cM}, \quad (E\udot_{\fr})^\vee[2] \stackrel{\theta_{\fr}}{\cong} E\udot_{\fr}, \quad \vd := \frac{1}{2} \rk(E_{\fr}\udot) =  -\frac{1}{2} \gamma^2,
$$ 
such that for every point $P = [(E,\phi)] \in \cM$, we have
\[
(E\udot_{\fr})^\vee|_P \cong R\Hom_X(E,E(-D))[1].
\]
Furthermore, $\cM$ has a derived enhancement with a $(-2)$-shifted symplectic structure.
\end{theorem} 

The moduli space is realized as an open subset of a (torsion version of) Huybrechts--Lehn's moduli space of $\delta$-stable framed sheaves \cite{HL}. This approach is inspired by a similar result of Bruzzo--Markushevich for framed torsion free sheaves on surfaces \cite{BM}. To show the existence of the $(-2)$-shifted symplectic structure
    in the sense of Pantev--To\"en--Vaqui\'e--Vezzosi \cite{PTVV}, we rely on work of Spaide \cite{Spa}. We expect that the moduli space $M_{(X,D,F)}^{\fr}(v)$ can be equipped with orientations by adapting the constructions in \cite{CGJ, Boj, JU}. In enumerative geometry, moduli of framed sheaves have also appeared in the context of local surfaces in \cite{Opr} and in the recent proof of the K-theoretic DT/PT vertex correspondence of \cite{KLT}.

The moduli space $M_{(X,D,F)}^{\fr}(v)$ is non-compact, so to define invariants we require the action of an algebraic torus $\TT$ with compact fixed locus. Our main example is 
$$
X=\PP^1\times\PP^1\times\PP^1\times\PP^1,\quad D=\{(z_1,z_2,z_3,z_4)\,|\,z_1z_2z_3z_4=\infty\}. 
$$
For each $A\in \bsix$, consider the surface 
$$S_A:=\{(z_1,z_2,z_3,z_4)\in X\,|\,z_{\bar{a}}=z_{\bar{b}}=0\},\quad \overline{A}=\{\bar{a}<\bar{b}\}
$$
and denote its intersection with $D$ by $\ell_A:=S_A\cap D$. Then $S_A\simeq \PP^1\times \PP^1$ and $\ell_A$ is a chain of two $\PP^1$'s. Our choice of the framing sheaf on $D$ is
\[
F_{\vec{r}} := \bigoplus_{A} \O_{\ell_A}^{\oplus r_A}
\]
with $r_A \in \Z_{\geq 0}$. Choosing the Chern character according to the dimension vector $(\vec{r}, n)$ of the 4D ADHM quiver, we set
$$v_{\vec{r},n}:=\sum_{A}r_A\cdot\ch(\O_{S_A})-n[\pt],
$$
and set $M_{(\PP^1)^4}(\vec{r},n) := M_{(X,D,F)}^{\fr}(v_{\vec{r},n})$. We observe that the virtual dimension of $M_{(\PP^1)^4}(\vec{r},n)$ arising from the $3$-term symmetric obstruction theory is $- \sum_{A \in \bthree} r_A r_{\overline{A}}$, which matches the virtual dimension of the corresponding origami moduli space.

The torus $\mathbb{T} \cong (\C^*)^3 \times (\C^*)^r$ defined in \eqref{eqn:defT} acts naturally on $M_{(\PP^1)^4}(\vec{r},n)$, where the action of the first factor is induced from the standard action on $(\PP^1)^4$ and the second factor acts on the framing. Then $\phi_{\fr}$, $\theta_{\fr}$ from Theorem \ref{thm:fr} are $\TT$-equivariant.

There are several indications that the moduli spaces $M_{Q_4}(\vec{r},n)$ and $M_{(\mathbb{P}^1)^4}(\vec{r},n)$ are closely related. Both depend on the same discrete parameters $(\vec{r}, n)$ and admit $\mathbb{T}$-equivariant derived enhancements with a $(-2)$-shifted symplectic form of the same virtual dimension. We make the following conjecture, which can be seen as a 4D analogue of Barth's result.
\begin{conjecture}\label{conj:main}
    For any $(\vec{r},n)$, there exists a $\mathbb{T}$-equivariant isomorphism 
    $$M_{Q_4}(\vec{r},n) \cong M_{(\mathbb{P}^1)^4}(\vec{r},n)$$ 
    intertwining the $\mathbb{T}$-equivariant 3-term symmetric obstruction theories on both sides. 
\end{conjecture}

Our main evidence for this conjecture is that the $(\C^*)^r$-fixed loci of both moduli spaces are isomorphic and that the virtual tangent representations at the $\TT$-fixed points coincide (Corollary \ref{cor:isofixloc} and Proposition \ref{prop:Tvirmatch} in Section \ref{sec:origamiviasheaves}).\footnote{In further support of this conjecture, Rennemo has notified us that he can produce a morphism $M_{Q_4}(\vec{r},n) \to M_{(\PP^1)^4}(\vec{r},n)$, which agrees with Barth's isomorphism when $r_A = 0$ for all $A \neq \{1,2\}$.}
\begin{theorem}\label{prop:comparefx}
There exists a $\TT$-equivariant isomorphism $$M_{Q_4}(\vec{r},n)^{(\C^*)^r} \cong M_{(\PP^1)^4}(\vec{r},n)^{(\C^*)^r}.$$ Under this isomorphism, for any $P \in M_{Q_4}(\vec{r},n)^{\TT} \cong M_{(\PP^1)^4}(\vec{r},n)^{\TT}$, we have
\[
(E_{Q_4}^{\mdot})^{\vee}|_P = (E_{\mathrm{fr}}^{\mdot})^{\vee}|_P \in K_0^{\TT}(\pt).
\]
\end{theorem}
This result implies that the virtual invariants defined by the quiver and framed sheaves models coincide, provided we use the sign choice coming from the quiver model. We provide some further evidence for Conjecture \ref{conj:main} in Section \ref{sec:mainconj}. In particular, we establish a $\TT$-equivariant isomorphism between the $(J_A = 0)$ locus in $M_{Q_4}(\vec{r},n)$ and a corresponding locus in $M_{(\PP^1)^4}(\vec{r},n)$ --- these loci contain the $(\C^*)^r$-fixed loci of Theorem \ref{prop:comparefx}. \\

\noindent \textbf{Acknowledgments.} This paper would not exist without N.~Nekrasov's beautiful discovery and exploration of gauge origami theory. We thank Y.~Bae, Y.~Cao, H.~Iritani, D.~Joyce, T.~Kimura, T.~Kinjo, H.~Liu, S.~Monavari, A.~Okounkov, H.~Park, B.~Szendr\H{o}i and R.~P.~Thomas for useful discussions. The second-named author would like to specifically thank J.~Rennemo; the sign analysis in \cite{KR} provides a key input for this paper. The third-named author would like to specially thank D.~Butson, L.~Hennecart and G. Noshita for discussions related to Quot scheme loci, connectedness of origami moduli spaces, and stable/co-stable wall-crossing respectively. NA was supported by the EPSRC through grant EP/R013349/1 and the World Premier International Research Center Initiative (WPI), MEXT, Japan. MK is supported by NWO Grant VI.Vidi.192.012 and ERC Consolidator Grant FourSurf 101087365. WL is supported by the Yonsei University Research Fund of 2024-22-0502, the POSCO Science Fellowship of POSCO TJ Park Foundation, and NRF grant funded by the Korean government (MSIT) (RS-2025-00514643).

\section{Quiver description} \label{sec:quiver}

\subsection{Moduli of quiver representations} \label{sec:quivermoduli}

In this section, we introduce the moduli spaces of quiver representations studied in this paper. In particular, we realize the stability \eqref{eqn:ADHMstab} from the introduction as a King stability condition.

Consider the (framed) quiver $\widehat{Q}_4$ without relations in Figure \ref{fig2}. Using the notation of the introduction, we fix the dimension vector $(\vec{r},n)$ from \eqref{eqn:dimvec} and set $V := \C^n$ and $W_A := \C^{r_A}$. We consider the subset 
\[
\mathcal{U} \subset \mathcal{R} := \bigoplus_{a} \mathrm{End}(V) \oplus \bigoplus_{A} \Big( \Hom(W_A,V) \oplus \Hom(V,W_A) \Big)
\]
consisting of representations $$P:= (B_a,I_A,J_A)$$ satisfying the stability condition \eqref{eqn:ADHMstab}. 

As there exists a King stability condition on $\widehat{Q}_4$ whose locus of stable objects is $\mathcal{U}$ (by the proof of Proposition \ref{prop:King} below), it follows that $\mathcal{U}$ is open. We consider the  action of $G := \GL(n,\C)$ on $\mathcal{R}$  given by 
\begin{equation} \label{eqn:GLaction}
g \cdot P = (g B_a g^{-1},gI_A,J_Ag^{-1}),
\end{equation}
for $g \in G$. This action restricts to $\mathcal{U}$, where it is free. We consider the quotient stack
\[
A_{\widehat{Q}_4}(\vec{r},n) := [\mathcal{U} / G].
\]
By the following proposition, this quotient  is a smooth quasi-projective variety. Recall that $r:=\sum_{A} r_{A}$.

\begin{proposition} \label{prop:King}
The variety $A_{\widehat{Q}_4}(\vec{r},n)$ is  smooth of dimension $3n^2+2rn$. 
\end{proposition}
\begin{proof}
We use the Crawley-Boevey trick: construct an auxiliary quiver $\widehat{Q}_4'$ with an isomorphic space of  representations and realize the stable framed representations of $\widehat{Q}_4$ as $\Theta$-stable representations of $\widehat{Q}_4'$ for an appropriate choice of King stability. We use the notation of \cite{Rei}.

Let $\widehat{Q}_4'$ be the quiver with two vertices denoted $0$ and $\infty$, edges consisting of 4 loops at $0$, $r$ arrows from $\infty$ to $0$ and $r$ arrows from $0$ to $\infty$. Then the representations of $\widehat{Q}_4$ with dimension vector $(\vec{r},n)$ and those of $\widehat{Q}_4'$ with dimension vector $(d_\infty,d_0) = (1,n)$ are identified by 
\begin{align*}
(B_a,I_A,J_A) \mapsto (B_a,i_{A,\alpha}, j_{A,\beta}),
\end{align*}
where $i_{A,\alpha}$ denote the column vectors of $I_{A}$ and $j_{A,\beta}$ denote the row vectors of $J_{A}$. Consider the group
$$
G' := (\mathrm{GL}(n,\C) \times \C^*) / \C^*,
$$
where the quotient is by diagonal matrices in both factors. This group acts by conjugation on the representations of $\widehat{Q}_4'$. Moreover, under the isomorphism 
\[
G' \cong \mathrm{GL}(n,\C) \times \{1\} \cong G
\]
the above identification of representations intertwines the group actions. 

We choose $\Theta = (\Theta_\infty = n, \Theta_0 = -1)$. Denote the induced $G'$-linearization by $\chi_{\Theta}$ \cite{Rei}. Then it is easy to see that there are no strictly $\Theta$-semistable representations of $\widehat{Q}_4'$ with dimension vector $(1,n)$. Moreover, the $\Theta$-stable representations of $\widehat{Q}_4'$ with dimension vector $(1,n)$ precisely correspond to the stable representations of $\widehat{Q}_{4}$. The result follows from \cite[Sect.~3.5]{Rei}.
\end{proof}

In the proof of the previous proposition, $A_{\widehat{Q}_4}(\vec{r},n)$ is realized as the moduli space of stable representations of some auxiliary quiver $\widehat{Q}_4'$ with respect to a choice of King stability condition. From this point of view, we consider the moduli space of semisimple representations of $\widehat{Q}_4'$ with dimension vector $(1,n)$ which (by abuse of notation) we denote by $A^{\mathrm{ssimp}}_{\widehat{Q}_4}(\vec{r},n)$. Then $A^{\mathrm{ssimp}}_{\widehat{Q}_4}(\vec{r},n)$ is an affine variety and there is a projective morphism \cite[Prop.~3.5]{Rei}
\[
A_{\widehat{Q}_4}(\vec{r},n) \to A^{\mathrm{ssimp}}_{\widehat{Q}_4}(\vec{r},n),
\]
called the semisimplification map. 

We impose the relations \eqref{eqn:4dADHM} via a section of a vector bundle on $A_{\widehat{Q}_4}(\vec{r},n)$. Consider the trivial vector bundle $\widetilde{E} \to \mathcal{U}$ with fibre
\[
\bigoplus_{A} \mathrm{End}(V) \oplus \bigoplus_{A} \Hom(W_A, W_{\overline{A}}) \oplus \bigoplus_{A, \overline{a}} (\Hom(W_A,V) \oplus \Hom(V,W_A)).
\]
The equations \eqref{eqn:4dADHM} define a section $\widetilde{s} \in H^0(\mathcal{U},\widetilde{E})$ given at some $P:= (B_a,I_A,J_A)$ by
\begin{equation} \label{eqn:defsec}
\widetilde{s}_P := (\mu_A,\nu_A,\mu_{A,\overline{a}}, \nu_{A,\overline{a}}).
\end{equation}
Moreover, $G$ acts on this vector bundle by
\[
g \cdot (P,(P_A,Q_A,R_{A,\overline{a}}, S_{A,\overline{a}})) = (g \cdot P,(g P_A g^{-1},Q_A,g R_{A,\overline{a}}, S_{A,\overline{a}} g^{-1}))
\]
and $\widetilde{s}$ is a $G$-equivariant section. Therefore $\widetilde{E}, \widetilde{s}$ descend to a vector bundle and section $E,s$ on $A_{\widehat{Q}_4}(\vec{r},n)$. We define the closed subscheme
\[
M_{Q_4}(\vec{r},n) := Z(s) \subset A_{\widehat{Q}_4}(\vec{r},n).
\]
The closed points of $M_{Q_4}(\vec{r},n)$ correspond to isomorphism classes of stable representations of the 4D ADHM quiver $Q_4$, which is $\widehat{Q}_4$ together with the relations \eqref{eqn:4dADHM}. Following Nekrasov, we refer to $M_{Q_4}(\vec{r},n)$ as an \emph{origami moduli space}.

As before, there exists a moduli space of semisimple representations of $Q'_4$, denoted by $M^{\mathrm{ssimp}}_{Q_4}(\vec{r},n)$, together with a projective morphism
$$M_{Q_4}(\vec{r},n)\rightarrow M^{\mathrm{ssimp}}_{Q_4}(\vec{r},n),
$$
again referred to as the semisimplification map.

The following two examples establish that $Q_4$ and its representations form a generalization of the traditional 2D ADHM quiver and its more recent 3D analogue. 

\begin{example} \label{ex:2dreduction}
Consider $Q_4$ and fix $r_{12} = r$ and $r_A = 0$ otherwise. Let $P = (B_a,I_A,J_A)$ be a stable representation of $Q_4$ with dimension vector $(\vec{r},n)$. We claim that $B_3=B_4=0$. By the stability condition \eqref{eqn:ADHMstab}, every vector in $V$ is of the form $v=f(B_1,B_2,B_3,B_4)I_{12}(w)$ for some polynomial $f$ (in non-commuting variables) and $w\in W_{12}$. The 4D ADHM equations \eqref{eqn:4dADHM} imply that $B_3$ commutes with all $B_a$'s and $B_3I_{12}=0$, so $B_3(v)=0$. Similarly, we have $B_4=0$. Therefore $(B_1,B_2,I_{12}, J_{12})$ is a representation of the 2D ADHM quiver $Q_2$ (Figure \ref{fig1}) with dimension vector $(r,n)$. It is easy to see that this provides an isomorphism $M_{Q_4}(\vec{r},n) \cong M_{Q_2}(r,n)$. 
\end{example}

\begin{example} \label{ex:3dreduction}
We define the 3D ADHM quiver $Q_3$ as follows. Remove from Figure \ref{fig2} the loop labelled by $B_4$ and the arrows labelled by $I_{a4}$, $J_{a4}$. We impose the relations 
\begin{equation} \label{eqn:3dADHM} 
[B_a, B_b] + I_A J_A = 0, \quad B_{\overline{a}} I_A = 0, \quad J_A B_{\overline{a}} = 0,
\end{equation}
for all $A = \{a<b\} \in \bthree := \{\{1,2\}, \{1,3\}, \{2,3\}\}$ and $\overline{a}\in \{1,2,3\}\backslash A$. 
    These relations can be  obtained by taking cyclic derivatives of the framed potential 
    $$W^{\fr}:=B_3([B_1, B_2]+I_{12}J_{12})+B_2([B_1, B_3]+I_{13}J_{13})+B_1([B_2, B_3]+I_{23}J_{23})
    $$
    introduced in \cite{RSYZ}.
The stability condition is the same as before, namely 
\begin{equation*}
\sum_{A \in \bthree} \C \langle B_1, B_2, B_3 \rangle \cdot I_A(W_A) = V.
\end{equation*} 
As in the 4D case, there is a smooth variety $A_{\widehat{Q}_3}(\vec{r},n)$ parametrizing stable $(\vec{r},n)$-dimensional representations of the corresponding quiver $\widehat{Q}_3$ without relations. Then the relations \eqref{eqn:3dADHM} cut out $M_{Q_3}(\vec{r},n) \subset A_{\widehat{Q}_3}(\vec{r},n)$.

Now consider $Q_4$ with $r_{a4} = 0$ for all $a$. Then any stable representation $P = (B_a,I_A,J_A)$ of $Q_4$ with dimension vector $(\vec{r},n)$ satisfies $B_4 = 0$ by the argument of Example \ref{ex:2dreduction}. Thus there is an isomorphism $M_{Q_4}(\vec{r},n) \cong M_{Q_3}(\vec{r},n)$.
\end{example}

\subsection{Torus action and fixed loci}\label{subsec:torus}

In this section, we define a torus action on the origami moduli space and analyze its torus fixed points. We define the algebraic torus $\mathbb{T}' := (\C^*)^4 \times (\C^*)^r$, which has $\TT \leq \TT'$ (defined in \eqref{eqn:defT}) as a subtorus. 

Recall the definition of the origami moduli space from the previous section. Then $\TT$ acts on points $P=(B_a,I_A,J_A)$ of $\mathcal{R}$ by\footnote{We act on the \emph{points} by $B_a\mapsto t_a^{-1} B_a$ so that  we act on the corresponding coordinate functions by $b_a \mapsto t_a b_a$.}
\begin{equation} \label{eqn:Taction}
\tau \cdot P = (t_a^{-1} B_a, I_A w_A^{-1}, t_A^{-1} w_A J_A), 
\end{equation}
for $\tau = (t_a,w_A) \in \TT$, where $w_A$ is regarded as a diagonal matrix with entries $w_{A,1}, \ldots, w_{A,r_A}$. As the $\TT$-action preserves the open subset $\mathcal{U}\subseteq \mathcal{R}$ and the actions of $G$ and $\TT$ commute, there is an induced $\TT$-action on $A_{\widehat{Q}_4}(\vec{r},n)$. 

The action of $\TT$ lifts to $\widetilde{E}$ by 
\begin{align}\label{eqn: Taction on bundle}
&\tau \cdot (P,(P_A,Q_A,R_{A,\overline{a}}, S_{A,\overline{a}})) = \nonumber\\
&(\tau \cdot P,(t_A^{-1} P_A, t_{\overline{A}}^{-1} w_{\overline{A}} Q_A w_A^{-1},t_{\overline{a}}^{-1} R_{A,\overline{a}} w_A^{-1}, t_A^{-1} t_{\overline{a}}^{-1} w_A S_{A,\overline{a}})).
\end{align}
The $G$ and $\TT$ actions commute, so we obtain an action of $\TT$ on $E$. As $s$ is $\TT$-equivariant, there is an induced $\TT$-action on $M_{Q_4}(\vec{r},n)$. In conclusion, we have the diagram
\begin{displaymath} 
\xymatrix
{
& E \ar[d] \\
M_{Q_4}(\vec{r},n) := Z(s) \ar@^{(->}[r] & \ar@/_1pc/[u]_{s} A_{\widehat{Q}_4}(\vec{r},n),
}
\end{displaymath}
where $\TT$ acts on $E$ and $s$ is $\TT$-equivariant. 

We first analyze the fixed loci with respect to the framing torus action. The following proposition resembles a result of Bifet for Quot schemes \cite{Bif}.
\begin{proposition} \label{prop:fixlocquiver}
There exists a $\TT$-equivariant isomorphism of schemes 
\[
M_{Q_4}(\vec{r},n)^{(\C^*)^r} \cong  \bigsqcup_{\sum_{A,\alpha} n_{A,\alpha} = n} \prod_{A,\alpha} M_{Q_2}(1,n_{A,\alpha}),
\]
where the disjoint union is over all decompositions $\sum_{A,\alpha} n_{A,\alpha} = n$ and $M_{Q_2}(1,n_{A,\alpha})$ denotes the moduli space of representations of the 2D ADHM quiver with dimension vector $(1,n_{A,\alpha})$. Moreover, any representation $(B_a, I_A, J_A)$ in this fixed locus satisfies $J_A = 0$ for all $A$.
\end{proposition}

\begin{proof}
We first construct a bijection at the level of closed points. The map from right to left follows by writing
\[
\C^n = \bigoplus_{A,\alpha} \C^{n_{A,\alpha}}
\]
and taking direct sums. For the map from left to right, we start with a $(\C^*)^r$-fixed representation $(B_a,I_A,J_A)$. Then for each $w = (w_{A,\alpha}) \in (\C^*)^r$, there is a unique $g_w \in G = \GL(n,\C)$ such that
\[
g_w B_a g_w^{-1} = B_a, \quad g_w I_A w_A^{-1} = I_A, \quad w_A J_A g_w^{-1} = J_A,
\]
where $w_A$ denotes the diagonal matrix with entries $w_{A,\alpha}$. This assignment yields a representation
\[
(\C^*)^{r} \to \GL(V), \quad w \mapsto g_w,
\]
where $V:=\C^n$. Denote by $\chi_{A,\alpha} \colon (\C^*)^r \to \C^*$ the character defined by $\chi_{A,\alpha}(w) = w_{A,\alpha}$. Then we obtain a decomposition 
\[
V = \bigoplus_{A,\alpha} V_{A,\alpha} \otimes \chi_{A,\alpha}
\]
and we define $n_{A,\alpha} := \dim V_{A,\alpha}$. Denote the standard basis of $W_A = \C^{r_A}$ by $e_{A,\alpha}$. Clearly $B_a V_{A,\alpha} \subset V_{A,\alpha}$ and $I_A e_{A,\alpha} \in V_{A,\alpha}$. We claim that the following hold.
\begin{itemize}
\item $J_{A}V_{A,\alpha} \subset \C e_{A,\alpha}$ and $J_A V_{B,\beta} = 0$ for all $A \neq B$ and any $\beta = 1, \ldots, r_B$;
\item $B_{\overline{a}} V_{A,\alpha} = 0$ for all $A \in \bsix$, $\alpha=1, \ldots, r_A$, $\overline{a} \in \overline{A}$. 
\end{itemize}
Both properties follow from the 4D ADHM equations \eqref{eqn:4dADHM} and stability \eqref{eqn:ADHMstab}. The resulting assignment is a well-defined bijection. For the final statement, any representation $(B_1,B_2,I,J)$ in $M_{Q_2}(1,n)$ satisfies $J=0$ \cite[Prop.~2.8]{Nak}.

The argument also works relative to a base $\C$-scheme $S$ of finite type, $w \in (\C^*)^r$ is replaced by $w \in \Hom(S,(\C^*)^r)$ and $(B_a,I_A,J_A)$ is replaced by a family of stable representations over $S$. Namely one considers $B_a \in \Hom(\cV, \cV)$, $I_A \in \Hom(W_A \otimes \O_S, \cV)$, and $J_A \in \Hom(\cV,W_A \otimes \O_S)$, with $\cV$ a rank $n$ locally free sheaf on $S$ satisfying the equations in \eqref{eqn:4dADHM} over $S$ and stability on the fibres over all points $s \in S$. Thus, the bijection is an isomorphism of schemes by \cite[Thm.~2.3]{Fog}.
\end{proof}

Using Proposition \eqref{prop:fixlocquiver}, we can describe the $\mathbb{T}$-fixed loci. 

\begin{corollary}\label{cor: T-fixed loci quiver}
    The $\mathbb{T}$-fixed locus $M_{Q_4}(\vec{r},n)^{\mathbb{T}}$ consists of reduced points labeled by 
\[
\blambda = (\lambda_{A,\alpha})_{A\in\bsix,\,1\leq\alpha\leq r_A}, \quad |\blambda| = \sum_{A,\alpha} |\lambda_{A,\alpha}| = n,
\]
where each $\lambda_{A,\alpha}$ is an integer partition. 
\end{corollary}
\begin{proof}
Under the isomorphism of Proposition \ref{prop:fixlocquiver}, the action of $\TT$ on the factor $M_{Q_2}(1,n_{A,\alpha})$ restricts to the action of $(\C^*)^2 \times (\C^*)^r$, where $(\C^*)^2$ acts by the standard action (with coordinates $(t_a)_{a \in A}$) and $(\C^*)^r$ acts trivially. The fixed locus is $M_{Q_2}(1,n_{A,\alpha})^{(\C^*)^2}$ is 0-dimensional and reduced \cite{Nak}, so we deduce that $M_{Q_4}(\vec{r},n)^{\TT}$ is also 0-dimensional and reduced. Note that $M_{Q_2}(1,n) \cong \Hilb^n(\C^2)$, the Hilbert scheme of $n$ points on $\C^2$ \cite{Nak}. The corollary follows from the standard description of the $(\C^*)^2$-fixed point locus of $\Hilb^n(\C^2)$.
\end{proof}

We set some notation for integer partitions. An integer partition $\lambda$ is a sequence
\[
\lambda = (\lambda_i \in \Z_{\geq 0})_{i \geq 1}
\]
satisfying $\lambda_{i} \geq \lambda_{i+1}$ for all $i$ and of total finite \emph{size} $|\lambda| := \sum_i \lambda_i$.  We can view $\lambda$ as a subset of $\Z_{\geq 0}^2$ by defining 

\[
(i,j) \in \lambda \Leftrightarrow j+1 \leq \lambda_{i+1}.
\]
We denote the \emph{length} of a partition by $\ell(\lambda)$; it is the largest $\ell$ such that $\lambda_{\ell} \neq 0$. For any square $(i,j) \in \lambda$, we define the \emph{hook length} at $(i,j)$ by
$$
h_\lambda(i,j):= |\{(i,j') \in \lambda \, | \, j' \geq j\} \cup \{(i',j) \in \lambda \, | \, i' \geq i\}|,
$$
and we set $h_{\lambda}(i,j) = 0$ when $(i,j) \notin \lambda$. For a sequence $\blambda = (\lambda_{A,\alpha})$, it is convenient to view each $\lambda_{A,\alpha}$ as a subset of $\Z_{\geq 0}^4$ as follows. For $A = \{a<b\}$, we view $\lambda_{A,\alpha} \subset \Z_{\geq 0}^2$ as above, and embed $\Z_{\geq 0}^2 \subset \Z_{\geq 0}^4$ via the obvious inclusion according to $A=\{a<b\} \subset \bfour$ (e.g., ~for $A=\{2,4\}$, we map $(x,y) \mapsto (0,x,0,y)$). 

\medskip

We now introduce the rank $n$ tautological vector bundle $\mathcal{V}$ on $M_{Q_4}(\vec{r},n)$. Consider the trivial vector bundle $\widetilde{\mathcal{V}} \to \mathcal{U}$ with fibre $V := \C^n$. The group $G$ acts on $\widetilde{\mathcal{V}}$ by
\[
g \cdot (P,v) = (g \cdot P, gv),
\]
for all $g \in G$. Furthermore, the trivial action of $\TT$ on $\widetilde{\mathcal{V}}$ commutes with the $G$-action. Therefore, on the quotient $A_{\widehat{Q}_4}(\vec{r},n)$, we obtain a rank $n$ $\TT$-equivariant vector bundle $\mathcal{V}$ and we denote its restriction to $M_{Q_4}(\vec{r},n)$ by the same symbol. 

Denote the usual tautological bundle on $M(1,n_{A,\alpha})$ by $\mathcal{V}_{A,\alpha}$. In terms of Hilbert schemes, $\mathcal{V}_{A,\alpha}$ is the  vector bundle $\O^{[n]}$ with fibre $H^0(Z,\O_Z)$ over $Z \in \Hilb^n(\C^2)$. Under the isomorphism of Proposition \ref{prop:fixlocquiver}, we have
\[
\mathcal{V}|_{\prod_{A,\alpha} M_{Q_2}(1,n_{A,\alpha})} \cong \bigboxplus_{A,\alpha}  \mathcal{V}_{A,\alpha}.
\]
Therefore, for any fixed point $P \in M_{Q_4}(\vec{r},n)^{\TT}$ corresponding to integer partitions $\blambda$, the $\TT$-character of the fibre $\mathcal{V}|_P \cong V$ is given by 
\begin{equation} \label{eqn:charV}
\sum_{A} K_A, \ K_A:=\sum_{\alpha} K_{A,\alpha}, \ K_{A,\alpha} := \sum_{(i,j) \in \lambda_{A,\alpha}} t_a^i t_b^j w_{A,\alpha} \ \textrm{for all } A = \{a<b\} \in \bsix.
\end{equation}

\subsection{Examples and \texorpdfstring{$J_A=0$}{JA} loci}

In this section, we give some examples of the origami moduli spaces $M_{Q_4}(\vec{r},n)$ and we discuss a relation to Quot schemes.

In the following examples, we explicitly describe the moduli space $M_{Q_4}(\vec{r},1)$ for two choices of rank vectors. Such descriptions are possible as setting $n=1$ leads to several simplifications: the stability \eqref{eqn:ADHMstab} is equivalent to having at least one nonzero $I_A$; the commutator $[B_a,B_b]$ automatically vanishes for any $A=\{a,b\} \in \bsix$ and the group $\GL(1,\C)$ acts on $B_a$, $I_A$, $J_A$ with weights $0$, $1$, $-1$, respectively. 

\begin{example}\label{ex:butterfly n=1}
Let $n=1$ and let $\vec{r}$ be the rank vector such that $r_{12}=r_{34}=1$ with $r_A=0$ otherwise. By Proposition \ref{prop:Jiszero} below, all stable solutions of the generalized ADHM equations have $J_A=0$. There are 3 types of solutions up to $\C^*$-action (all non-specified variables are zero):
    \begin{enumerate}
        \item[(i)] $I_{12}=1$, $B_1, B_2\in \C$, 
        \item[(ii)] $I_{34}=1$, $B_3,B_4\in\C$,
        \item[(iii)] $[I_{12}: I_{34}]\in \PP^1$.
    \end{enumerate}
These solutions are not mutually disjoint. The type (iii) solution with $[I_{12}: I_{34}]=[1:0]$ (resp.~$[I_{12}: I_{34}]=[0:1]$) is equal to the type (i) (resp.~type (ii)) solution with $B_a=0$ for all $a\in \bfour$. Therefore, $M_{Q_4}(\vec{r},1)$ has three irreducible components $\C^2$, $\C^2$ and $\PP^1$ glued together as in Figure \ref{fig3}.
\begin{figure}[ht]
\begin{center}
\includegraphics[scale=.09]{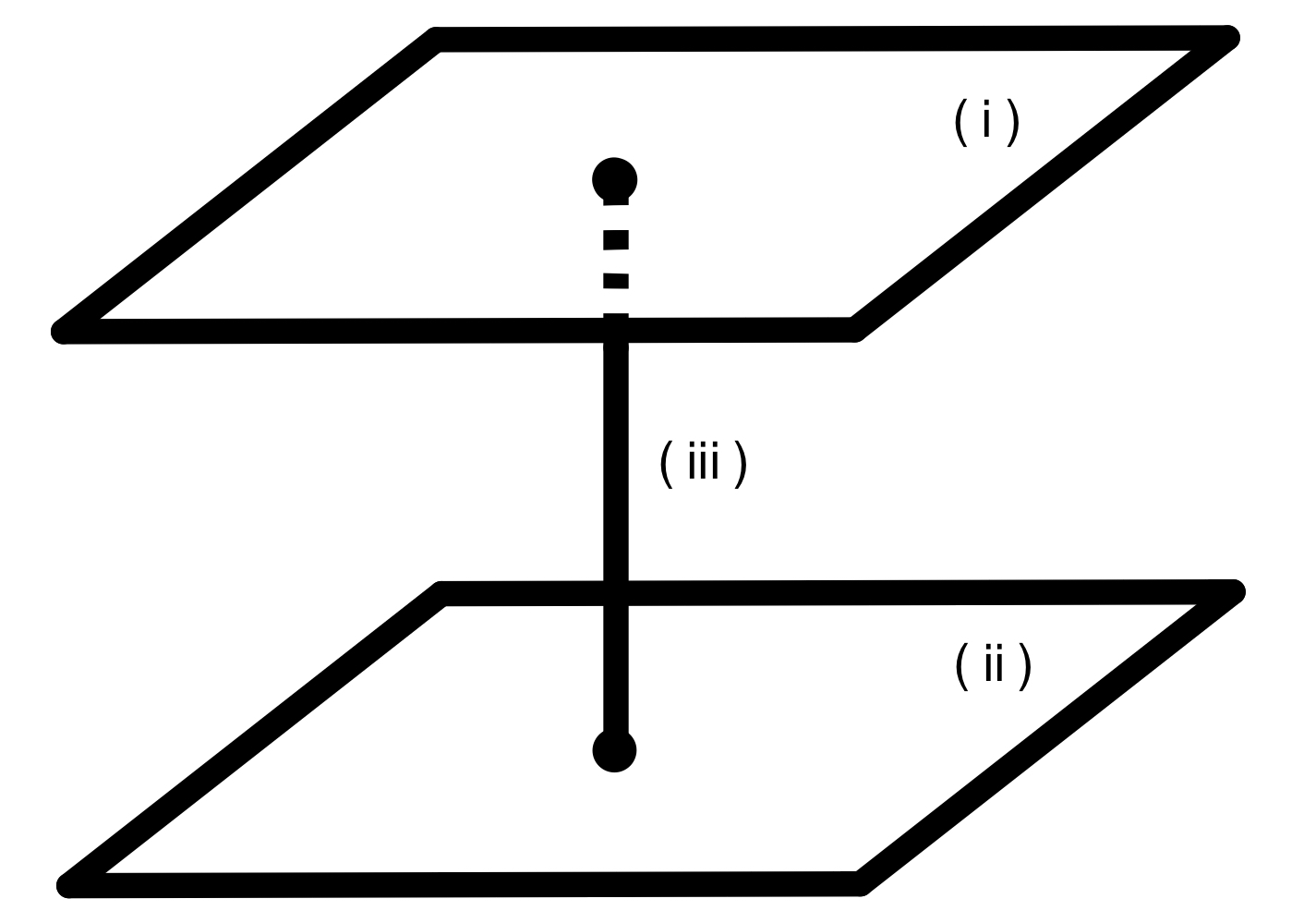}
\end{center}
\caption{The moduli space of Example \ref{ex:butterfly n=1}.}\label{fig3}
\end{figure}
Additionally, one can check from the local equations that the moduli space is reduced. This example is also worked out in \cite[Sect.~6.2]{Nek3} for more general $r_{12}, r_{34}>0$ with $r_A=0$ otherwise.
\end{example}

\begin{example}\label{ex:book n=1}
Let $n=1$ and let $\vec{r}$ be the rank vector such that $r_{12}=r_{23}=1$ with $r_A=0$ otherwise. As explained in Example \ref{ex:3dreduction}, all the stable solutions of the generalized ADHM equations have $B_4=0$. There are five types of solutions up to $\C^*$-action (all non-specified variables are zero):
    \begin{enumerate}
        \item[(i)] $I_{12}=1$, $B_1, B_2\in \C$,
        \item[(ii)] $I_{23}=1$, $B_2, B_3\in \C$,
        \item[(iii)] $[I_{12}:I_{23}]\in \PP^1$, $B_2\in \C$,
        \item[(iv)] $I_{12}=1$, $J_{23}\in \C$, $B_2\in \C$,
        \item[(v)] $I_{23}=1$, $J_{12}\in \C$, $B_2\in \C$.
    \end{enumerate}    
Therefore, the moduli space has five components $\C^2, \C^2, \C\times \PP^1, \C^2, \C^2$ glued together as in Figure \ref{fig4}. \begin{figure}[ht]
\begin{center}
\includegraphics[scale=.09]{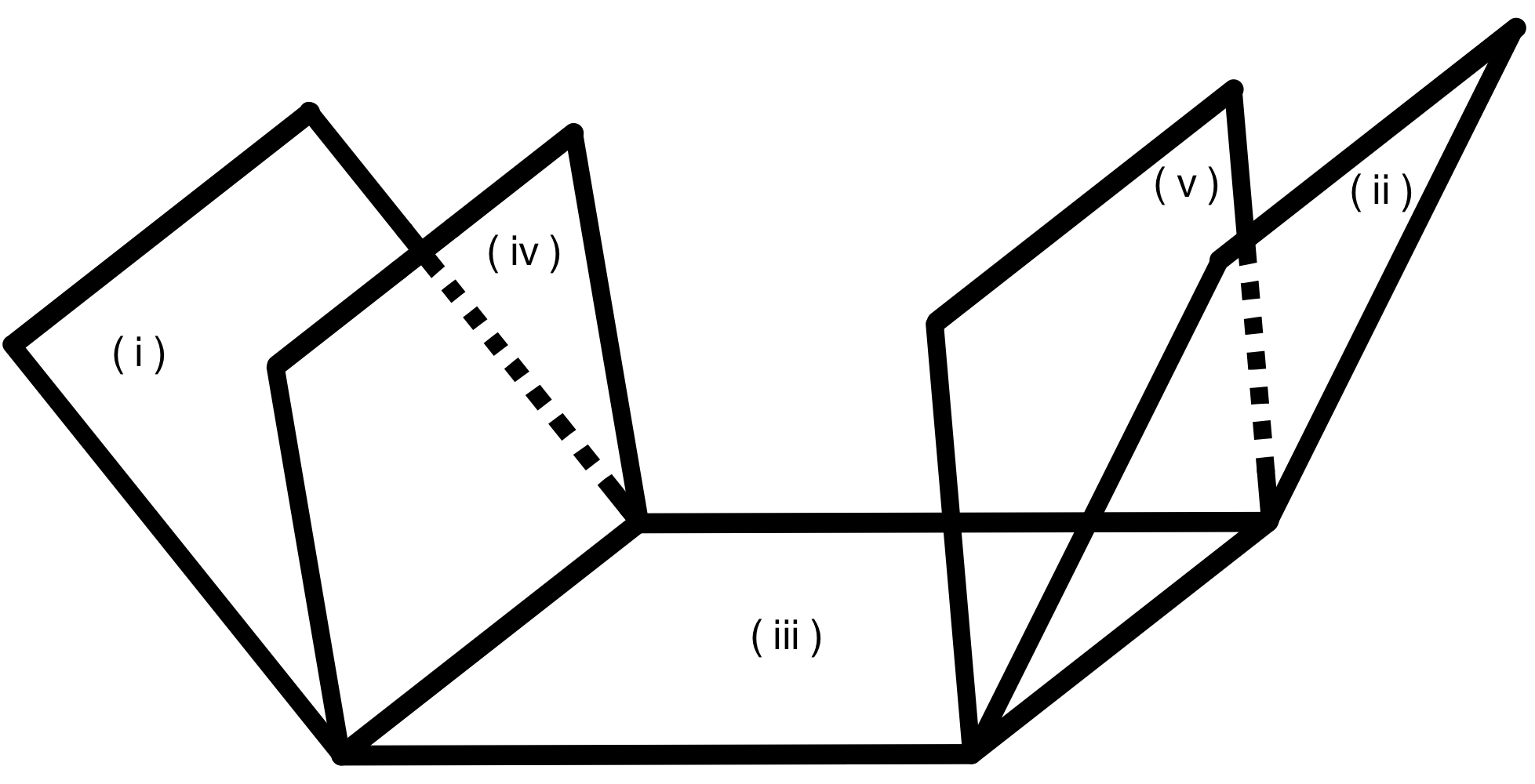}
\end{center}
\caption{The moduli space of Example \ref{ex:book n=1}.}\label{fig4}
\end{figure}
Again, one can check that the moduli space is reduced. The components (iv) and (v) have $J_A\neq 0$ in the interior, so they do not belong to $\widetilde{M}_{Q_4}(\vec{r},1)$ as in Definition \ref{def: J=0 loci} below.

\end{example}
\begin{remark}
In contrast to Example \ref{ex:butterfly n=1}, the moduli space we describe in Example \ref{ex:book n=1} is different from that of  \cite[Sect.~6.4]{Nek3}. In loc.~cit.~there are only solutions of type (i), (ii), (iii). This difference is due to the modification of the moduli spaces introduced in \cite[Sect.~3.3]{Nek3}.\footnote{\label{footnote: Nekrasov's different moduli space}Following the notation of \cite[Sect.~3.3]{Nek3}, the moduli space we consider in this paper is denoted by $\mathfrak{M}_n^0(\vec{r})$. The moduli space described in \cite[Sect.~6.4]{Nek3} is the smaller one $\mathfrak{M}_n^\infty(\vec{r})$ obtained by imposing extra equations. 
    These two moduli spaces coincide if the rank vector is such that $r_{12}, r_{34} \geq 0$ and $r_A=0$ otherwise.} This modification makes a significant difference with regard to the Donaldson--Thomas theory of Calabi--Yau 3-folds. While the moduli space we describe in Example \ref{ex:book n=1} is globally a critical locus (as explained in Section \ref{sec:quivervirtualstruc}), the one in \cite[Sect.~6.4]{Nek3} \emph{cannot} be written as a critical locus around the intersection points of the irreducible components. This is because Nekrasov's moduli space, which is a union of three components (i), (ii), (iii), has reduced lci singularities. By \cite[Cor.~1.21]{BF2}, critical loci do not have such singularities. 
\end{remark}

We end this section with a relation to Quot schemes on the singular surface $\bigcup_{A \in \bsix} \C_A^2 \subset \C^4$, where $\C_A^2 = Z(x_{\overline{a}},x_{\overline{b}})$ with $\overline{A} = \{\overline{a}, \overline{b}\}$. We refer to \cite{FM} for an analysis of a similar Quot scheme on $\bigcup_{a \in \bfour} \C_a^3 \subset \C^4$ and \cite{Mon} on $\bigcup_{a \in \{1,2\}} \C_a^1 \subset \C^2$ (in the obvious notation).

\begin{definition}\label{def: J=0 loci}
    We denote by $\widetilde M_{Q_4}(\vec r, n)$ the closed subscheme of the moduli space $M_{Q_4}(\vec r, n)$ parametrizing $(B_a,I_A, J_A)$ with $J_A=0$ for all $A\in \bsix$. 
\end{definition}

By Proposition \ref{prop:fixlocquiver}, we have
\[
M_{Q_4}(\vec{r},n)^{(\C^*)^r} \subset \widetilde{M}_{Q_4}(\vec{r},n).
\]

\begin{proposition}\label{prop: Quot scheme}
    There is an isomorphism 
    $$\widetilde M_{Q_4}(\vec r, n)\simeq \Quot_{\C^4}\Big(\bigoplus_{A}\O_{\C^2_A}^{\oplus r_A}, n\Big),
    $$
    where the right hand side is the Quot scheme parametrizing length $n$ quotients of the sheaf $\bigoplus_{A} \O_{\C^2_A}^{\oplus r_A}$.
\end{proposition}

\begin{proof}
As all $J_A$ are set to zero, $\widetilde M_{Q_4}(\vec r, n)$ parametrizes the linear maps $(B_a,I_A)$ satisfying the equations
$$[B_a, B_b]=0,\quad B_{\overline{a}} I_A = 0, 
$$
for all $A = \{a<b\} \in \bsix$ and $\overline{a} \in \overline{A}$, and the stability condition
$$\sum_{A} \C \langle B_1, B_2, B_3, B_4 \rangle \cdot I_A(W_A) = V,$$
up to $\GL(n,\C)$-action. Using the relations, we rewrite the stability condition as 
$$\sum_{A=\{a<b\} \in \bsix} \C [B_a, B_b]\, I_A(W_A) = V.
$$
As the polynomial ring $\C[B_1,B_2, B_3, B_4]$ acts on the $n$-dimensional vector space $V$, we can regard $V$ as a 0-dimensional sheaf on the affine 4-space $\C^4$ of length $n$. Furthermore, the stability condition implies that the linear maps $I_A \colon W_A\rightarrow V$ induce a surjective morphism 
$$\bigoplus_{A} \Big(\O_{\C^2_A}\otimes W_A\Big)\twoheadrightarrow V
$$
of coherent sheaves on $\C^4$. Two equivalent collections of linear maps under $\GL(n,\C)$-action define the same point in the Quot scheme. This process is reversible, yielding a bijection between $\C$-points of $M_{Q_4}(\vec r, n)$ and the Quot scheme. 

\end{proof}

\subsection{Connectedness}

In general, the question of connectedness (in the complex analytic topology) of moduli spaces of representations of quivers with relations is a difficult problem. The connectedness Nakajima quiver varieties is proved in \cite{CB}. In this section, we prove connectedness of the moduli space $M_{Q_4}(\vec{r},n)$. 

The following lemma provides a sufficient criterion for connectedness when a scheme is equipped with a torus action and an equivariant proper morphism. 

\begin{lemma}\label{lem: connectedness criterion}
    Let $X$, $Y$ be $\mathbb{C}$-schemes of finite type equipped with the action of an algebraic torus $T$. Let $\pi \colon X\rightarrow Y$ be a $T$-equivariant proper morphism. Assume that $Y^T=\{y_0\}$ and for every $y\in Y$ there exists a $1$-parameter subgroup $\lambda \colon \C^*\rightarrow T$ such that $\lim_{t\rightarrow 0}\lambda(t)\cdot y=y_0$. If $X^T$ lies in a single connected component of $X$, then $X$ is connected.
\end{lemma}
\begin{proof}
    We write $x\sim x'$ if $x$ and $x'$ lie in the same connected component of $X$. Let $x\in X$ be any point and set $y=\pi(x)$. By assumption, there exists a $1$-parameter subgroup $\lambda \colon \C^*\rightarrow T$ such that $\lim_{t\rightarrow 0}\lambda(t)\cdot y=y_0$. In other words, there exists a morphism
    $$f \colon \C\rightarrow Y,\quad t\mapsto\begin{cases}
        \lambda(t)\cdot y &\textnormal{if }t\neq 0,\\
        y_0&\textnormal{if }t=0.
    \end{cases}
    $$
    This morphism fits into a commutative diagram
    \begin{equation*}
        \begin{tikzcd}
\C^* \arrow[r] \arrow[d, phantom, "\subseteq" sloped] & X \arrow[d, "\pi"] \\
\C \arrow[r,"f"]            & Y                
\end{tikzcd}
    \end{equation*}
    where the top arrow sends $t\in \C^*$ to $\lambda(t)\cdot x\in X$. As $\pi$ is proper, the morphism $f$ uniquely lifts to a morphism $\widetilde{f} \colon \C\rightarrow X$. Set $x_0:=\widetilde{f}(0)\in X$. Then $x\sim x_0$, as they are connected by $\widetilde{f}$. Note that $X_{y_0}:=\pi^{-1}(y_0)$ is a proper $T$-scheme containing both $X^T$ and $x_0$. Choose a generic $1$-parameter subgroup $\mu \colon\C^*\rightarrow T$ with $X^T=X^\mu$. As $X_{y_0}$ is proper, there exists a limit $\lim_{t\rightarrow 0}\mu(t)\cdot x_0=:x_1\in X_{y_0}$. As limit points are necessarily $\mu$-fixed, we have $x_1\in X^{\mu}=X^T$. Thus $x\sim x_0\sim x_1\in X^T$. As $X^T$ lies in a single connected component, this shows that $X$ is connected.   
\end{proof}

\begin{theorem}\label{thm: connectedness of origami moduli space}
    For every $(\vec{r},n)$, the moduli space $M_{Q_4}(\vec{r},n)$ is connected. 
\end{theorem}
\begin{proof}
    Apply Lemma \ref{lem: connectedness criterion} to the following setting:
    \begin{enumerate}
        \item [i)] $X=M_{Q_4}(\vec{r},n),\ Y=M^{\mathrm{ssimp}}_{Q_4}(\vec{r},n)$ equipped with their natural $\mathbb{T}$-actions,
        \item [ii)] the semisimplification map $\pi \colon M_{Q_4}(\vec{r},n)\rightarrow M^{\mathrm{ssimp}}_{Q_4}(\vec{r},n)$, which is proper and $\mathbb{T}$-equivariant,
        \item [iii)] $M^{\mathrm{ssimp}}_{Q_4}(\vec{r},n)^{\mathbb{T}}=\{y_0\}$ corresponding to $B_a=I_A=J_A=0$ for all $a\in\bfour$, $A\in\bsix$.
    \end{enumerate}
    We verify the assumptions of the lemma. First, consider $\lambda \colon \C^*\rightarrow T$ defined as 
    $$\lambda(t):=(t_a=t^{-1}, w_A=t^{-1}\cdot\textnormal{id}).$$
    As $\lambda(t)\cdot (B_a,I_A,J_A)=(tB_a, t I_A, tJ_A)$, this 1-parameter subgroup contracts the space $M^{\mathrm{ssimp}}_{Q_4}(\vec{r},n)$ to $y_0$ as $t$ goes to $0$.

    It remains to show that $M_{Q_4}(\vec{r},n)^{\mathbb{T}}$ lies in a single connected component. This is the  most delicate part of the proof. We prove the slightly stronger statement that $M_{Q_4}(\vec{r},n)^{(\C^*)^r}$ lies in a single connected component. Roughly speaking, this follows from the fact that components (i) and (ii) in Example \ref{ex:butterfly n=1} and \ref{ex:book n=1} are connected via component (iii). Recall from Proposition \ref{prop:fixlocquiver} that $(\C^*)^r$-fixed locus of $M_{Q_4}(\vec r, n)$ is
\begin{equation} \label{eqn:disjointunionHilb}
\bigsqcup_{\sum_{A,\alpha}n_{A,\alpha}=n}\prod_{A, \alpha}\Hilb^{n_{A,\alpha}}(\C^2_{A}).
\end{equation}
    Note that each set of the disjoint union, indexed by $(n_{A,\alpha})$, is connected.
    
    For each $B \in \bsix$ and $\beta = 1, \ldots, r_B$, we denote by $(\delta_{B,\beta})$ the sequence which has value 1 for $A = B$ and $\alpha = \beta$, and zero otherwise. Then any two numerical types $(n_{A,\alpha})$ indexing the connected components in \eqref{eqn:disjointunionHilb} can be related by a sequence of simple modifications, namely, by changing $(m_{A,\alpha})+(\delta_{B,\beta})$ to $(m_{A,\alpha})+(\delta_{B',\beta'})$ where $\sum_{A,\alpha}m_{A,\alpha}=n-1$. Therefore, it suffices to construct a connected family of representations intersecting both of the connected components in \eqref{eqn:disjointunionHilb} corresponding to $(m_{A,\alpha})+(\delta_{B,\beta})$ and $(m_{A,\alpha})+(\delta_{B',\beta'})$ with $(B,\beta) \neq (B',\beta')$. Fix 
    $$(Z_{A,\alpha})\in \prod_{A, \alpha}\Hilb^{m_{A,\alpha}}(\C^2_{A})$$
    such that the support of $Z_{A,\alpha}$ is disjoint from the origin $o\in \C^2_A$ for all $A,\alpha$. We denote the ADHM data corresponding to $Z_{A,\alpha}$ by
    \[
    B_a^{(A,\alpha)} \in \mathrm{End}(V_{A,\alpha}), \quad I_{A,\alpha} \in \Hom(\C_{A,\alpha}, V_{A,\alpha}), \quad J_{A,\alpha} = 0 \in \Hom(V_{A,\alpha}, \C_{A,\alpha}),
    \]
    where $V_{A,\alpha}$ is a vector space of dimension $m_{A,\alpha}$, $\C_{A,\alpha}$ is a 1-dimensional vector space and $J_{A,\alpha} = 0$ by Proposition \ref{prop:fixlocquiver}. We define a $\PP^1_{[s:t]}$-family of representations with rank vector $(\vec r,n)$. Namely for each $[s:t] \in \PP^1$, define $(B_a, I_A, J_A)$ as follows:
    \begin{enumerate}
        \item [a)] the central vertex is $V=\Big(\bigoplus_{A,\alpha} V_{A,\alpha} \Big)\oplus \C$,\vspace{5pt}
        \item [b)] the endomorphism $B_a$ acts on $V_{A,\alpha}$ by $B_{a}^{(A,\alpha)}$ and on $\C$ by $0$,\vspace{5pt}
        \item [c)] denote by $\iota_{A,\alpha} \colon V_{A,\alpha} \hookrightarrow \bigoplus_{A',\alpha'} V_{A',\alpha'}$ the inclusion, then the linear map $I_A$ is determined by
        $$I_A|_{\C_{A,\alpha}}(1) =\begin{cases}
            (\iota_{A,\alpha}(I_{A,\alpha}(1)),0) &\textnormal{if}\quad (A,\alpha) \neq (B,\beta), (B',\beta' )\\
            (\iota_{A,\alpha}(I_{A,\alpha}(1)),s) &\textnormal{if}\quad (A,\alpha) = (B,\beta)\\
            (\iota_{A,\alpha}(I_{A,\alpha}(1)),t) &\textnormal{if}\quad (A,\alpha) = (B',\beta' ),
        \end{cases}
        $$
        \item [d)] $J_A=0$ for all $A$.
    \end{enumerate}
    This choice of $(B_a, I_A, J_A)$ satisfies the generalized ADHM equations \eqref{eqn:4dADHM}. Stability follows from the fact that the support of $Z_{A,\alpha}$ is away from $o\in \C^2_A$ and $(s,t)\neq (0,0)$. This property can be seen by applying the 2-dimensional analogue of \cite[Lem.~3.7]{KR} combined with \cite[Lem.~3.6]{KR}. Specializing to the case $[s:t]=[1:0]$, this construction yields a $(\C^*)^r$-fixed point with numerical data $(m_{A,\alpha})+(\delta_{B,\beta})$. More precisely, it corresponds to $(Z_{A,\alpha}')$ such that $Z_{B,\beta}'=Z_{B,\beta}\sqcup o$ and $Z_{A,\alpha}'=Z_{A,\alpha}$ if $(A,\alpha)\neq (B,\beta)$. The case $[s:t]=[0:1]$ is similar. In words: the above pencil of representations moves a point supported at $o$ from copy $(B,\beta)$ to copy $(B',\beta')$. This completes the proof.
\end{proof}

Recall from Definition \ref{def: J=0 loci} that $\widetilde{M}_{Q_4}(\vec{r},n)\hookrightarrow M_{Q_4}(\vec{r},n)$ denotes the closed locus with extra equations $J_A=0$ for all $A\in\bsix$. The argument of the proof of the previous theorem also applies to this locus. 

\begin{proposition}\label{prop: connectedness of Quot scheme}
    For every $(\vec{r},n)$, the moduli space $\widetilde{M}_{Q_4}(\vec{r},n)$ is connected. 
\end{proposition}
\begin{proof}
Note that $\widetilde{M}_{Q_4}(\vec{r},n)^{(\C^*)^r}=M_{Q_4}(\vec{r},n)^{(\C^*)^r}$. As the explicit $\PP^1_{[s:t]}$-families of representations in the proof of Theorem \ref{thm: connectedness of origami moduli space} lie in the $J_A=0$ locus, the same proof applies to $\widetilde{M}_{Q_4}(\vec{r},n)\rightarrow \widetilde{M}^{\mathrm{ssimp}}_{Q_4}(\vec{r},n)$.

\end{proof}

\subsection{Virtual structures} \label{sec:quivervirtualstruc}

In this section, we endow $M_{Q_4}(\vec{r},n)$ with a symmetric 3-term obstruction theory. We also discuss find some values of $\vec{r}$ for which the origami moduli space is smooth or quasi-smooth. Along the way, we discuss the relevant obstruction theories, virtual cycles and virtual structure sheaves in our setting.

\subsubsection{Smooth case} Consider the moduli space $\mathcal{M}:=M_{Q_4}(\vec{r},n)$. When $r_{12} = r$ and $r_A = 0$ otherwise, $M_{Q_4}(\vec{r},n) \cong M_{Q_2}(r,n)$ (Example \ref{ex:2dreduction}) is a smooth variety of dimension $2rn$ \cite{Nak}. Moreover, it is holomorphic symplectic, that is \cite{Nak} 
\begin{equation} \label{eqn:holsymp}
T_{\mathcal{M}} \stackrel{\theta}{\cong} \Omega_{\mathcal{M}} \otimes (t_1t_2)^{-1}, \quad \theta^{\vee} = \theta \otimes (t_1t_2).
\end{equation}

For general $\vec{r}$, the space $\mathcal{M}$ is no longer smooth, for example as in Examples \ref{ex:butterfly n=1} and \ref{ex:book n=1}.

\subsubsection{Quasi-smooth case}  Suppose $r_{a4} = 0$ for all $a$. We explain how to obtain a quasi-smooth structure on the origami moduli spaces. By Example \ref{ex:3dreduction}, one has $M_{Q_4}(\vec{r},n) \cong M_{Q_3}(\vec{r},n)$. As in \cite{RSYZ}, consider the regular function 
\[
f \colon \mathcal{A} := A_{\widehat{Q}_3}(\vec{r},n) \to \C, \quad (B_1,B_2,B_3,(I_A,J_A)_{A \in \bthree}) \mapsto \sum_{A = \{a<b\} \in \bthree} \tr(B_{\overline{a}} \mu_A).
\]
Recall that $\bthree = \{\{1,2\}, \{1,3\}, \{2,3\}\}$, $\overline{a}$ is the unique element of $\{1,2,3\} \setminus A$, and $\mu_A$ is as defined in \eqref{eqn:4dADHM}. By \cite[Lem.~3.1.1]{RSYZ}, one has $M_{Q_3}(\vec{r},n) = Z(df)$. Thus there is a Behrend--Fantechi 2-term obstruction theory of virtual dimension zero \cite{BF}
\begin{displaymath}
\xymatrix
{
T_{\mathcal{A}}|_{\mathcal{M}} \otimes (t_1t_2t_3) \ar[r] \ar^{(df)^*}[d] & \Omega_{\mathcal{A}}|_{\mathcal{M}} \ar@{=}[d] & E_{Q_3}\udot \ar[d] \\
I/I^2|_{\mathcal{M}} \ar^{d}[r] & \Omega_{\mathcal{A}}|_{\mathcal{M}} & \tau^{\geq -1} \LL_{\mathcal{M}},
}
\end{displaymath}
where $I := I_{\cM/\cA} \subset \O_{\cA}$ denotes the ideal sheaf of $\cM \subset \cA$, the bottom map is the K\"ahler differential and the top map $d \circ (df)^*$ is the Hessian of $f$. Note that
\[
(E_{Q_3}\udot)^{\vee}[1] \stackrel{\theta}{\cong} E_{Q_3}\udot \otimes (t_1t_2t_3)^{-1}, \quad \theta^\vee[1] = \theta \otimes (t_1t_2t_3).
\]
The moduli spaces $\mathcal{A}, \mathcal{M}$ are generally non-compact but they have an action of 
\[
\TT = (\C^*)^3 \times (\C^*)^r
\]
defined as in \eqref{eqn:Taction}. As  $f$ is scaled by $\TT$ with weight $t_1t_2t_3$, the above 2-term obstruction theory is $\TT$-equivariant. The $\TT$-fixed locus $\mathcal{M}^{\TT}$ is 0-dimensional and reduced (Proposition \ref{prop:fixlocquiver}). So there is a virtual cycle and virtual structure sheaf \cite{BF}
\[
[\cM]^{\vir} \in A_{0}^{\TT}(\cM), \quad \O_{\cM}^{\vir} \in K_0^{\TT}(\cM).
\]
As in Nekrasov--Okounkov \cite{NO}, it is more natural to work with the \emph{twisted} virtual structure sheaf  
\begin{equation} \label{eqn:NOtwist}
\widehat{\O}_{\cM}^{\vir} := \O_{\mathcal{M}}^{\vir} \otimes (\det E_{Q_3}\udot)^{\frac{1}{2}} \cong \O_{\mathcal{M}}^{\vir} \otimes K_{\cA}|_{\cM} \otimes (t_1t_2t_3)^{-\frac{\dim(\cA)}{2}},
\end{equation}
which lies in $K_0^{\widetilde{\TT}}(\mathcal{M})$, where $\widetilde{\TT} \to \TT$ is a double cover of $\TT$. By the Atiyah--Bott localization formula for equivariant Chow groups \cite{EG} and the Thomason localization formula for equivariant K-theory \cite{Tho}, we have
\begin{align*}
\iota_* \colon A^{\TT}_*(\mathcal{M}^{\TT}) \otimes_{\Z[\sigma]} \Q(\sigma) &\stackrel{\cong}{\to} A^{\TT}_*(\mathcal{M}) \otimes_{\Z[\sigma]} \Q(\sigma) \\
\iota_* \colon K_0^{\TT}(\mathcal{M}^{\TT}) \otimes_{\Z[\tau,\tau^{-1}]} \Q(\tau) &\stackrel{\cong}{\to} K_0^{\TT}(\mathcal{M}) \otimes_{\Z[\tau,\tau^{-1}]} \Q(\tau),
\end{align*}
where $\Z[\sigma] \cong A_*^{\TT}(\pt)$, $\Z[\tau,\tau^{-1}] = K_0^{\TT}(\pt)$, $\tau = (t_1,t_2,t_3, w_{A,\alpha})$ denotes the collection of all primitive characters of $\TT$ and $\sigma = (s_1,s_2,s_3,v_{A,\alpha})$ are the corresponding equivariant parameters, i.e., 
\[
s_a = c_1^{\TT}(t_a), \quad v_{A,\alpha} = c_1^{\TT}(w_{A,\alpha}).
\]
In order to work with the twisted virtual structure sheaf, we have to tensor the above isomorphism by $- \otimes_{\Z[\tau,\tau^{-1}]} \Q(\tau^{\frac{1}{2}})$. Therefore, one can define $\int_{[\cM]^{\vir}} 1$ by the equivariant push-forward of $(\iota_*)^{-1}[\cM]^{\vir}$ to a point. Similarly, the refined invariant $\chi(\cM, \widehat{\O}_{\cM}^{\vir})$ is defined as the K-theoretic equivariant push-forward of $(\iota_*)^{-1} \widehat{\O}_{\cM}^{\vir}$ to a point. Concretely, these expressions are given by the Graber--Pandharipande localization formula \cite{GP}, and its K-theoretic analogue due to Qu \cite{Qu}
\begin{align*}
\int_{[\cM]^{\vir}} 1 &= \int_{[\cM^{\TT}]^{\vir}} \frac{1}{e(N^{\vir})} \in \Q(\sigma), \\ 
\chi(\cM,\widehat{\O}_{\cM}^{\vir}) &= \chi\Big(\cM^{\TT}, \frac{\O_{\cM^{\TT}}^{\vir} \otimes \det(E_{Q_3}\udot)^{\frac{1}{2}}|_{\cM^{\TT}}}{ \Lambda_{-1} (N^{\vir})^{\vee}}\Big) \in \Q(\tau^{\frac{1}{2}}),
\end{align*}
where $[\cM^{\TT}]^{\vir}$, $\O_{\cM^{\TT}}^{\vir}$ are the induced virtual cycle and virtual structure sheaf of $\cM^{\TT}$, $N^{\vir}$ is the virtual normal bundle, $e(-)$ is the $\TT$-equivariant Euler class, and $\Lambda_{-1}(-)^\vee$ is the K-theoretic $\TT$-equivariant Euler class. Consider the Calabi--Yau torus
\[
\TT_0 := Z(t_1t_2t_3-1) \leq \TT.
\]
Then $E_{Q_3}\udot$ is a $\TT_0$-equivariant \emph{symmetric} 2-term obstruction theory in the sense of \cite{Beh}. In this case, the above-mentioned invariants coincide, are integers and will be calculated in Proposition \ref{prop:CY3spec} below. 

\subsubsection{General case} We now discuss the general origami moduli space $\cM := M_{Q_4}(\vec{r},n)$. Recall that $\cM \subset \cA := A_{\widehat{Q}_4}(\vec{r},n)$ is cut out by the section $s$ of the vector bundle $E \to \cA$ as described in Section \ref{sec:quivermoduli}. However, there is more structure. In order to discuss this structure, we index the standard basis of $\C^4$ and $\Lambda^2 \C^4$ as follows
\[
\C^4: \quad (e_a)_{a \in \bfour}, \quad \Lambda^2 \C^4: \quad (e_A := e_a \wedge e_b)_{A = \{a<b\} \in \bsix}.
\]
Define a quadratic form $q$ on $E$ as follows: over any point $P$, let $q \colon E|_P \to \C$ be
\begin{align}\label{eqn: quadratic form}
q(P_A,Q_A,R_{A,\overline{a}}, S_{A,\overline{a}}) &= \sum_{A \in \bthree} (\tr(P_A P_{\overline{A}}) - \tr(Q_A Q_{\overline{A}})) \cdot e_A \wedge e_{\overline{A}}\\
&\quad +\sum_{A \in \bsix} \tr(R_{A,\overline{a}} S_{A,\overline{b}} - R_{A,\overline{b}} S_{A,\overline{a}}) \cdot e_A \wedge e_{\overline{A}}, 
\notag\end{align}
where in the second sum we write $\overline{A} = \{\overline{a} < \overline{b}\}$. Furthermore, we use $e_A \wedge e_{\overline{A}}$ in order to keep track of signs by choosing an isomorphism $\C \cdot (e_1 \wedge \cdots \wedge e_4) \cong \C$ determined by $e_1 \wedge \cdots \wedge e_4 \mapsto 1$. The quadratic form $q$ is non-degenerate. This result holds as for any two finite-dimensional complex vector spaces $U,V$, the quadratic form $\Hom(U,V) \times \Hom(V,U) \mapsto \C$, $(A,B) \mapsto \tr(BA)$ is non-degenerate, and direct sums of non-degenerate quadratic forms are non-degenerate.
\begin{proposition}\label{prop:qszero}
    The section $s\in H^0(\mathcal{A},E)$ defined by the 4D ADHM equations \eqref{eqn:4dADHM} is isotropic with respect to the quadratic form $q$ in \eqref{eqn: quadratic form}.
\end{proposition}
\begin{proof}
    Recall that $\tr(XY)=\tr(YX)$ for any two matrices $X$ and $Y$ of size $m \times n$ and $n \times m$, respectively. In particular, the trace is invariant under cyclic permutation of four matrices. By definition of the quadratic form, we have
\begin{align*}
    q(s)=&\sum_{A=\{a<b\}\in \bthree}\tr\left(([B_a,B_b]+I_AJ_A)([B_{\overline{a}},B_{\overline{b}}]+I_{\overline{A}}J_{\overline{A}})\right)e_A\wedge e_{\overline{A}}\\
    &-\sum_{A=\{a<b\}\in \bthree}\tr\left(J_{\overline{A}}I_AJ_AI_{\overline{A}}\right)e_A\wedge e_{\overline{A}} \\
    &+\sum_{A\in \bsix}\tr\left(B_{\overline{a}}I_AJ_AB_{\overline{b}}-B_{\overline{b}}I_AJ_AB_{\overline{a}}\right)e_A\wedge e_{\overline{A}}.
\end{align*}
By symmetry and $4$-cyclic invariance of the trace, this sum simplifies to  
\begin{multline*}
    \sum_{A=\{a<b\}\in \bthree}\tr\left(
    [B_a,B_b][B_{\overline{a}},B_{\overline{b}}]+I_AJ_A[B_{\overline{a}},B_{\overline{b}}]+I_{\overline{A}}J_{\overline{A}}[B_a,B_b]\right)e_A\wedge e_{\overline{A}}\\
    -\sum_{A\in \bsix}\tr\left(I_AJ_A[B_{\overline{a}},B_{\overline{b}}]\right)e_A\wedge e_{\overline{A}}.
\end{multline*}
As $\{A,\overline{A}\}_{A\in \bthree}=\bsix$ and $e_A\wedge e_{\overline{A}}=e_{\overline{A}}\wedge e_A$, every term vanishes with the possible exception of
$$\sum_{A=\{a<b\}\in \bthree}\tr\left(
    [B_a,B_b][B_{\overline{a}},B_{\overline{b}}]\right)e_A\wedge e_{\overline{A}}.
$$
This expression is shown to vanish in \cite[Lem.~3.1]{KR}, but we recall its proof for the reader's convenience. Expanding the commutators and using cyclicity of the trace, the above expression equals 
$$\frac{1}{2}\sum_{\sigma\in \mathfrak{S}_4}\textnormal{sgn}(\sigma)\cdot \tr\left(B_{\sigma(1)}B_{\sigma(2)}B_{\sigma(3)}B_{\sigma(4)}\right)
$$
where $\textnormal{sgn}(\sigma)$ denotes the sign of $\sigma$. This sum vanishes.
\end{proof}

As $\cM$ is cut out by an isotropic section of a quadratic vector bundle on a smooth ambient space, we have a symmetric 3-term obstruction theory
\begin{displaymath}
\xymatrix
{
T_{\cA}|_{\cM} \ar^<<<<{ds}[r] & E|_{\cM} \stackrel{q}{\cong} E^*|_{\cM} \ar^{s}[d] \ar^>>>>{(ds)^*}[r] & \Omega _{\cA}|_{\cM} \ar@{=}[d] & E_{Q_4}\udot \ar[d] \\
& I/I^2|_{\cM} \ar^{d}[r] & \Omega_{\cA}|_{\cM} & \tau^{\geq -1} \LL_{\cM},
}
\end{displaymath}
where $s^* \colon E^* \to \O_{\cA}$ denotes the dual section, $ds$ is the differential of the section, and $(ds)^*$ is its dual. The top row is a complex by isotropy Proposition \ref{prop:qszero}. This obstruction theory is $\TT$-equivariant with respect to the $\TT$-action defined in \eqref{eqn:Taction}, and is  $\TT$-equivariantly symmetric
\[
(E_{Q_4}\udot)^{\vee}[2] \stackrel{\theta}{\cong} E_{Q_4}\udot, \quad \theta^\vee[2] = \theta.
\]
As the obstruction theory is 3-term, there is no corresponding Behrend--Fantechi virtual cycle. However, there is an Oh--Thomas virtual cycle \cite{OT}. Its construction requires a choice of orientation. In our setup, an orientation is determined by a map $o : \O_{\cA} \to \det E$ such that
the composition
\[
\O_{\cA} \overset{o \otimes o}{\to} \det(E) \otimes \det(E) \overset{\id \otimes \Lambda^{\rk} q}\to \det(E) \otimes \det(E^*) \to \O_{\cA}
\]
equals $(-1)^{\rk(\rk-1)/2}\id_{\O_A}$, where $\rk:=\rk(E)$ which is even.

We produce an orientation by specifying a ($\TT$-invariant) maximal isotropic subbundle 
\[
0 \to \Lambda \to E \to \Lambda^* \to 0.
\]
In fact, we will choose $\Lambda$ such that this short exact sequence is split. Any such $\Lambda$ produces a canonical isomorphism 
\[
p_\Lambda \colon \det E \cong \det \Lambda \otimes \det \Lambda^* \cong \O_{\cA},
\]
and, by \cite[Def.~2.2]{OT}, the orientation induced by $\Lambda$ is
\[
o_\Lambda = (-\sqrt{-1})^{\rk/2}p_{\Lambda}^{-1}.
\]
The following is immediate from the definitions of $E$ and $q$.
\begin{lemma}
Let $\Lambda \subset E$ be the $\TT$-invariant subbundle with fibre
\[
\bigoplus_{A \in \bthree} \Big( \mathrm{End}(V) \cdot t_A^{-1} \oplus \Hom(W_A, W_{\overline{A}}) \cdot t_{\overline{A}}^{-1} \Big) \oplus \bigoplus_{A \in \bsix, \overline{a} \in \overline{A}} \Hom(V,W_A) \cdot t_{\overline{a}}^{-1} t_A^{-1}
\]
over all $P \in \cA$. Then $\Lambda$ is maximal isotropic and thus induces an orientation $o_{\Lambda}$.
\end{lemma}

Throughout the rest of this paper, we fix the following orientation\footnote{The global sign $(-1)^{(r-1)n}$ is inserted in order to later get a precise match with Nekrasov's original definition of the origami partition function via Boltzmann weights (Theorem \ref{thm:main}).} for $M_{Q_4}(\vec{r},n)$:
\begin{equation} \label{eqn:choiceofori}
o := (-1)^{(r-1)n} \cdot o_\Lambda.
\end{equation}
Then we obtain an Oh--Thomas virtual cycle and virtual structure sheaf
\[
[\cM]^{\vir} \in A_{\vd}^{\TT}(\cM,\Z[\tfrac{1}{2}]), \quad \widehat{\O}_{\cM}^{\vir} \in K_0^{\TT}(\cM)_{\mathrm{loc}},
\]
where 
\[
\vd := \dim \cA - \frac{1}{2} \rk(E) = -\sum_{A \in \bthree} r_A r_{\overline{A}}.
\]
Here $A_{*}^{\TT}(\cM,\Z[\tfrac{1}{2}])$ denotes the $\TT$-equivariant Chow group with coefficients in $\Z[\tfrac{1}{2}]$ and 
\[
K_0^{\TT}(\cM)_{\mathrm{loc}} := K_0^{\TT}(\cM) \otimes_{ \Z[\tau,\tau^{-1}]} \Q(\tau^{\frac{1}{2}}). 
\]
As before, one can use the $\TT$-localization theorems in Chow or K-theory to define $\int_{[\cM]^{\vir}} 1$ and $\chi(\cM,\widehat{\O}_{\cM}^{\vir})$. 
\begin{definition} \label{def:origampartfun}
We define the cohomological, resp.~K-theoretic, origami partition function with respect to the choice of orientation \eqref{eqn:choiceofori} by
\begin{align*}
\sfZ_{\vec{r}}(q) &:= \frac{1}{\mathsf{C}_{\vec{r}}} \sum_{n=0}^\infty q^n \int_{[M_{Q_4}(\vec{r},n)]^{\vir}} 1, \quad  \sfZ_{\vec{r}}^K(q) := \frac{1}{\mathsf{C}^K_{\vec{r}}}\sum_{n=0}^\infty q^n \chi(M_{Q_4}(\vec{r},n), \widehat{\O}^{\vir}), \\
\mathsf{C}_{\vec{r}} &:= \int_{[M_{Q_{4}}(\vec{r},0)]^{\vir}} 1, \quad \hspace{61pt}\mathsf{C}^K_{\vec{r}} := \chi(M_{Q_{4}}(\vec{r},0), \widehat{\O}^{\vir}).
\end{align*}
In the next section we establish that the ``leading terms'' $\mathsf{C}_{\vec{r}}$, $\mathsf{C}^K_{\vec{r}}$ are non-zero, so that the above quotients are well-defined. 
\end{definition}

In order to make these expressions explicit, we require the Oh--Thomas localization formula \cite[Sect.~7]{OT}. As $\cM^{\TT}$ is 0-dimensional and reduced (Proposition \ref{prop:fixlocquiver}), at any fixed point $P$, we have $h^0(E_{Q_4}\udot|_{P}^{\TT}) = 0$, where $E_{Q_4}\udot|_P^{\TT}$ denotes the $\TT$-fixed part of the complex. Thus the differential
\[
d_Ps \colon T_{\cA}|_P^{\TT} \hookrightarrow E|_P^{\TT}
\]
is injective, and therefore $T_{\cA}|_P^{\TT}$ gets identified with an isotropic subspace of $E|_P^{\TT}$. There are (a priori) two cases:
\begin{itemize}
\item $T_{\cA}|_P^{\TT}$ is not maximal, in which case $\rk E_{Q_4}\udot |_P^{\TT} < 0$ and $P$ does not contribute to the invariant;
\item $T_{\cA}|_P^{\TT}$ is maximal, in which case $E_{Q_4}\udot |_P^{\TT}$ is acyclic and $[E_{Q_4}\udot|_P] = [E_{Q_4}\udot|_P^m] \in K_0^{\TT}(\cM)_{\loc}$, where $(-)^m$ denotes the $\TT$-moving part. 
\end{itemize}
In the second case\footnote{From Theorem \ref{thm:main}, proved in the next section, it is easy to see that $\rk E_{Q_4}^\mdot|_P^{\TT} = 0$ for all fixed points $P$, so the first case never occurs.}, the contribution to the invariant is 
\[
\pm \frac{1}{\sqrt{(-1)^{\vd}e(E_{Q_4}\udot|_P^\vee)}}, \quad \pm \frac{1}{\sqrt{(-1)^{\vd}\widehat{\Lambda}_{-1}(E_{Q_4}\udot|_P)}}.
\]
Here $\pm \sqrt{-}$ indicates a choice of square root determined by the global choice of orientation $o = (-1)^{(r-1)n} o_{\Lambda}$. The way to take square roots is described in general terms in \cite[Sect.~7]{OT} and is worked out explicitly for moduli (globally) cut out by isotropic sections in \cite{KR}. In K-theory, we have a splitting
\begin{equation} \label{eqn:locmodelhalf}
[E_{Q_4}\udot|_P] = F_P + F^*_P, \quad F_P := [T_{\cA}|_P - \Lambda|_P]. 
\end{equation}
Therefore, the above expressions can be slightly simplified to
\[
\pm \frac{1}{e(T_{\cA}|_P - \Lambda|_P)}, \quad \pm \frac{1}{\widehat{\Lambda}_{-1}(\Omega_{\cA}|_P - \Lambda^*|_P)}.
\]
The ``correct'' choice of sign comes from a comparison of the orientations induced by $T_{\cA}|_P^{\TT}$ and $\Lambda|_P^{\TT}$. The result is the following \cite{KR}.
\begin{proposition}[Kool--Rennemo] \label{prop:KR}
In the above setup, we have
\[
\pm = (-1)^{(r-1)n} \cdot (-1)^{\dim \mathrm{cok} ( \pi_{\Lambda} \circ d_Ps)^{\TT}},
\]
where we consider the composition 
\[
T_{\cA}|_P \stackrel{d_Ps}{\to} E|_P = \Lambda|_P \oplus \Lambda^*|_P \stackrel{\pi_\Lambda}{\to} \Lambda|_P
\]
and $\pi_\Lambda$ denotes the projection onto $\Lambda|_P$.
\end{proposition}

The goal of the next section is to determine
\begin{align*}
\int_{[\cM]^{\vir}} 1 &= \sum_{P \in \cM^{\TT}} \pm \frac{1}{e(T_{\cA}|_P - \Lambda|_P)} \\
\chi(\cM,\widehat{\O}_{\cM}^{\vir}) &= \sum_{P \in \cM^{\TT}} \pm \frac{1}{\widehat{\Lambda}_{-1}(\Omega_{\cA}|_P - \Lambda^*|_P)}.
\end{align*}
The main result is the determination of the signs $\pm$.

\subsection{Proof of Theorem \ref{thm:main}} \label{sec:proofmainthm}

We use the notation of the previous sections. Let
\[
P = [(B_a^P,I_A^P,J_A^P)] \in \mathcal{M}^{\TT}  = M_{Q_4}(\vec{r},n)^{\TT}.
\]
We start with the following crucial proposition.
\begin{proposition} \label{prop:cokcalc}Let 
\begin{align*}
\Xi_P : \bigoplus_{a=1}^{3} \mathrm{End}(V) \cdot t_a^{-1}  &\to \bigoplus_{A \in \bthree} \mathrm{End}(V) \cdot t_A^{-1}, \\ 
(B_1,B_2,B_3) &\mapsto \Big( [B_a^P, B_b] + [B_a,B_b^P] \Big)_{A = \{a<b\} \in \bthree}.
\end{align*}
Then $\mathrm{cok}(\pi_{\Lambda} \circ d_P s)^{\TT} = \mathrm{cok}(\Xi_P)^{\TT}.$ 

\end{proposition}

\begin{proof}
Recall from Section \ref{sec:quivermoduli} that $\cA = [\mathcal{U} / G]$. By a slight abuse of notation, we also write $P := (B_a^P,I_A^P,J_A^P) \in \mathcal{U}$ for a lift of $P$. Then there is a surjection $T_{\mathcal{U}}|_{P} \twoheadrightarrow T_{\cA}|_P$ and it suffices to consider the cokernel of the composition
\[
T_{\mathcal{U}}|_{P} \twoheadrightarrow T_{\cA}|_P \stackrel{d_Ps}{\to} E|_P = \Lambda|_P \oplus \Lambda^*|_P \stackrel{\pi_\Lambda}{\to} \Lambda|_P.
\]
Explicitly, using \eqref{eqn:defsec}, this composition is given by the map 
\begin{align*}
&\bigoplus_{a=1}^{4} \mathrm{End}(V) \cdot t_a^{-1} \oplus \bigoplus_{A \in \bsix} \Big( \Hom(W_A,V) \oplus \Hom(V,W_A) \cdot t_A^{-1 }\Big) \to \\
&\bigoplus_{A \in \bthree} \Big( \mathrm{End}(V) \cdot t_A^{-1} \oplus \Hom(W_A,W_{\overline{A}}) \cdot t_{\overline{A}}^{-1} \Big) \oplus \bigoplus_{A \in \bsix, \overline{a} \in \overline{A}} \Hom(V,W_A) \cdot t_{\overline{a}}^{-1}t_A^{-1}  \\
&(B_a,I_A,J_A) \mapsto \big( ( d_P\mu_A )_{A \in \bthree}, (d_P \nu_{A})_{A \in \bthree}, (d_P \nu_{A,\overline{a}})_{A \in \bsix, \overline{a} \in \overline{A}}\big)
\end{align*}
where the differentials are
\begin{align*}
    &d_P \mu_{A=\{a<b\}} = [B_a^P, B_b] + [B_a,B_b^P] + I_A^P J_A + I_A J_A^P,\\
    &d_P \nu_A = J_{\overline{A}}^P I_A + J_{\overline{A}} I_A^P, \quad d_P \nu_{A,\overline{a}} = J_A^P B_{\overline{a}} + J_A B_{\overline{a}}^P.
\end{align*}
As $P \in \cM^{\TT}$, Proposition \ref{prop:fixlocquiver} implies that $J_A^P = 0$. This simplifies the composition to the map 
$$(B_a,I_A,J_A) \mapsto \big( ([B_a^P, B_b] + [B_a,B_b^P] + I_A^P J_A)_{A\in \bthree},  (J_{\overline{A}} I_A^P)_{A\in \bthree},(J_A B_{\overline{a}}^P)_{A\in \bsix, \overline{a}\in\overline{A}} \big).
$$
As we are concerned with the $\mathbb{T}$-fixed part of the cokernel, it suffices to show that 
$$I_A^PJ_A\in \mathrm{End}(V)\cdot t_A^{-1},\quad 
J_{\overline{A}}I^P_A\in \Hom(W_A,W_{\overline{A}})\cdot t_{\overline{A}}^{-1},\quad 
J_AB^P_{\overline{a}}\in \Hom(V,W_A) \cdot t_{\overline{a}}^{-1}t_A^{-1}$$
have no $\mathbb{T}$-fixed part. Recall that $\TT$-characters of these vector spaces can be read off from the decompositions 
$$V=\bigoplus_{A\in \bsix}\bigoplus_{1\leq \alpha\leq r_A}\bigoplus_{(i,j)\in \lambda_{A,\alpha}}t_a^it_b^j w_{A,\alpha},\quad W_A=\bigoplus_{1\leq\alpha\leq r_A} w_{A,\alpha}.
$$
We deduce that the $\mathbb{T}$-fixed parts of the following vector spaces vanish
$$\left(\Hom(V,W_A)\cdot t_A^{-1}\right)^{\mathbb{T}}=\left(\Hom(W_A,W_{\overline{A}})\cdot t_{\overline{A}}^{-1}\right)^{\mathbb{T}} = \left(\Hom(V,W_A) \cdot t_{\overline{a}}^{-1}t_A^{-1}\right)^{\mathbb{T}}=0.
$$
In particular, $J_{\overline{A}}I^P_A$ and $J_AB^P_{\overline{a}}$ have no $\mathbb{T}$-fixed parts. Similarly, $I^P_A J_A$ has no $\mathbb{T}$-fixed part because it lies in the image of the $\mathbb{T}$-equivariant map
\begin{equation}
\Hom(V,W_A)\cdot t_A^{-1}\xrightarrow{I^P_A(-)}\mathrm{End}(V)\cdot t_A^{-1}. \qedhere
\end{equation}
\end{proof}

By Proposition \ref{prop:fixlocquiver}, the fixed point $P \in \cM^{\TT}$ corresponds to a collection of integer partitions
\[
P \leftrightarrow \blambda.
\]
The parity of $\dim \mathrm{cok}(\Xi|_P)^{\TT}$ was already determined in the context of the sign calculation of the fixed points for $\Hilb^n(\C^4)$ in \cite[Sect.~4.2]{KR}. By \cite[Prop.~4.5, 4.2]{KR}, the result applied to our setup is 
\begin{equation} \label{eqn:cokXi}
\dim \mathrm{cok}(\Xi|_P)^{\TT} \equiv \sum_{A \in \bsix \setminus \bthree} \sum_{\alpha} \Big( |\lambda_{A,\alpha}| - h_{\lambda_{A,\alpha}}(0,0) \Big) \mod 2,
\end{equation}
where $\bsix \setminus \bthree = \{\{1,4\}, \{2,4\}, \{3,4\}\}$ and $h_{\lambda_{A,\alpha}}(0,0)$ denotes the hook length of the square $(0,0) \in \lambda_{A,\alpha}$.\footnote{When we apply \cite[Prop.~4.5, 4.2]{KR} to $\blambda=(\lambda_{A,\alpha})$, their formula greatly simplifies because two of the coordinates of $(i,j,k,\ell)\in \lambda_{A,\alpha} \subset \Z_{\geq 0}^4$ are zero when $\lambda_{A,\alpha}$ is viewed as a solid partition.}

\begin{remark} \label{rem:n=0}
In the case $n=0$, the space $\cM = \cA$ consists of a single closed point $P$. Then the character of $\Lambda|_P$ is given by 
\[
\sum_{A \in \bthree} N_A^* N_{\overline{A}} t_{\overline{A}}^{-1},
\]
where $N_A$ is defined in \eqref{eqn:Neknot}, and
\[
\chi(\cM,\widehat{\O}_{\cM}^{\vir}) = \prod_{A \in \bthree} [N_A^* N_{\overline{A}} t_{\overline{A}}^{-1}],
\]
where $[-]$ is defined in \eqref{eqn:bracket}. We denote this invariant by $\mathsf{C}^K_{\vec{r}}$ and the sign is $+$ by \eqref{eqn:cokXi} applied to $\blambda = \varnothing$. Similarly, we define $\mathsf{C}_{\vec{r}} := \int_{[\cM]^{\vir}} 1 = \prod_{A \in \bthree} e(N_A^* N_{\overline{A}} t_{\overline{A}}^{-1})$.
\end{remark}

We now compare two choices of halves of $T_{\mathcal{M}}^{\vir}|_P := E_{Q_4}\udot|_P^\vee$: the choice obtained in \eqref{eqn:locmodelhalf} and the choice \eqref{eqn:Neknot} provided by Nekrasov.
\begin{proposition} \label{prop:compare}
We have the following identity in K-theory
\[
\mathsf{v}_{\blambda} - \Big(T_{\mathcal{A}}|_P - \Lambda|_P + \sum_{A \in \bthree} N_A^* N_{\overline{A}} t_{\overline{A}}^{-1}  \Big) = G_{\blambda} - G^*_{\blambda},
\]
for some class $G_{\blambda} \in K_0^{\TT}(\pt)$ satisfying
\[
\rk G^m_{\blambda} \equiv (r-1)n + \sum_{A \in \bsix \setminus \bthree} \sum_{\alpha} \Big( |\lambda_{A,\alpha}| - h_{\lambda_{A,\alpha}}(0,0) \Big) \mod 2,
\]
where $(-)^m$ denotes the $\TT$-moving part.
\end{proposition}
\begin{proof}
As before, we consider $\cA = [\mathcal{U} / G]$ and, by slight abuse of notation, we denote by $P \in \mathcal{U}$ a lift of $P$ so that we have a surjection $T_{\mathcal{U}}|_{P} \twoheadrightarrow T_{\cA}|_P$. In fact, $T_{\cA}|_P = T_{\cU}|_P - \mathrm{End}(V)$. By \eqref{eqn:charV} and \eqref{eqn:Neknot}, the characters of $V$ and $W_A$ are given by
\[
\chi_{V} = \sum_A K_A, \quad \chi_{W_A} = N_A,
\]
where $N_A$ was defined in \eqref{eqn:Neknot} and $\blambda$ is the collection of integer partitions corresponding to $P$. As $\TT$ acts as defined in \eqref{eqn:Taction}, the character of $T_{\mathcal{U}}|_{P}$ is given by
\begin{equation} \label{eqn:charTU}
(t_1^{-1}+t_2^{-1}+t_3^{-1}+t_4^{-1}) \cdot \sum_{A,B} K_A^* K_B + \sum_{A,B} N_A^* K_B + \sum_{A,B} K_A^* N_B \cdot t_B^{-1}. 
\end{equation}
On the other hand, by \eqref{eqn: Taction on bundle}, the character of $\Lambda|_P$ is given by
\begin{align}
\begin{split} \label{eqn:charLambda}
&((t_1t_2)^{-1}+(t_1t_3)^{-1}+(t_2t_3)^{-1}) \cdot \sum_{A,B} K_A^* K_B + \sum_{A \in \bthree} N_A^* N_{\overline{A}} t_{\overline{A}}^{-1} + \sum_{A,B, \overline{a} \in \overline{A}} K_B^* N_A \cdot t_{\overline{a}}^{-1}t_A^{-1}.
\end{split}
\end{align}

Recall that in the introduction, for $A \in \bsix$ we defined $\phi(A) := \min \overline{A}$. We also define $\psi(A) := \max \overline{A}$. A tedious yet straight-forward calculation shows that the first equality holds for 
\begin{align*}
G_{\blambda}:= 
&\sum_{A \in \bthree} (t_1+t_2+t_3 - t_1t_2 - t_1t_3 - t_2t_3) K_A^* K_A + \\
&\sum_{A \in \bsix \setminus \bthree} \Big( \sum_{a \neq \psi(A)} t_a - \sum_{B \in \bthree \mathrm{ \, s.t. \, } \psi(A) \notin B} t_B \Big) K_A^* K_A - \\
&\sum_A (1-t_A+t_{\phi(A)}t_A) N_A^* K_A - \sum_{A \neq B} N_A^* K_B + \\
&\sum_{A<B}(1-t_1^{-1}-t_2^{-1}-t_3^{-1}-t_4^{-1}+t_1t_4+t_2t_4+t_3t_4) K_A^* K_B, 
\end{align*}
where we always work modulo the relation $t_1t_2t_3t_4=1$.

For the second part of the proposition, we must determine $\rk G^m_{\blambda} = \rk G_{\blambda} - \rk G^{\TT}_{\blambda}$. Note that $\rk G_{\blambda} \equiv rn \mod 2$. The result follows from Lemma \ref{lem:comb} below, the proof of which is a calculation  left to the reader.
\end{proof}

\begin{lemma} \label{lem:comb}
Let $\lambda$ be an integer partition and define $Z_\lambda := \sum_{(i,j) \in \lambda} t_1^i t_2^j$. Then
\begin{align*}
\dim (t_2 Z_\lambda Z_\lambda^*)^{T} &= |\lambda| - \ell(\lambda), \\
\dim (t_3 Z_\lambda Z_\lambda^*)^{T} &= 0, \\
\dim (t_1 t_2 Z_\lambda Z_\lambda^*)^{T} &= |\lambda| - h_{\lambda}(0,0), \\
\dim (t_1 t_3 Z_\lambda Z_\lambda^*)^{T} &=  0,
\end{align*}
where $(-)^T$ is the $T$-fixed part with respect to $T:=Z(t_1t_2t_3t_4-1)$, $\ell(\lambda)$ denotes the length and $h_{\lambda}(0,0)$ the hook length at $(0,0)$ of $\lambda$.
\end{lemma}

We now put everything together.
\begin{proof}[Proof of Theorem \ref{thm:main}]
By Propositions \ref{prop:KR}, \ref{prop:cokcalc}, \ref{prop:compare}, Remark \ref{rem:n=0}, and \eqref{eqn:cokXi} we have
\begin{align*}
\frac{\chi(\cM,\widehat{\O}_{\cM}^{\vir})}{\mathsf{C}^K_{\vec{r}}}  &= (-1)^{(r-1)n} \cdot \sum_{P \in \cM^{\TT}} (-1)^{\dim \mathrm{cok} ( \pi_{\Lambda} \circ d_Ps)^{\TT}} \frac{\widehat{\Lambda}_{-1}(-\sum_{A \in \bthree} N_A N_{\overline{A}}^{*} t_{\overline{A}}  ))}{\widehat{\Lambda}_{-1}(\Omega_{\cA}|_P - \Lambda^*|_P)} \\
&=(-1)^{(r-1)n} \cdot \sum_{P \in \cM^{\TT}} (-1)^{\dim \mathrm{cok} ( \pi_{\Lambda} \circ d_Ps)^{\TT}} (-1)^{\rk(G^m)} [-\mathsf{v}_{\blambda}] \\
&= \sum_{\blambda \mathrm{ \, s.t. \, } |\blambda| = n} [-\mathsf{v}_{\blambda}],
\end{align*}
where in the second equality $\blambda$ denotes the partition corresponding to $P$. This finishes the proof.

\end{proof}

\subsection{Invariants} 

In this section, we investigate several consequences of Theorem \ref{thm:main}. We first discuss some generalities regarding the cohomological and K-theoretic origami partition functions of Definition \ref{def:origampartfun}. We then investigate two special cases (both of which require the fact that all the signs in Theorem \ref{thm:main} are ``$+$''):
\begin{itemize}
\item (2D ADHM) $r_{12} = r$ and $r_A = 0$ otherwise (continuation of Example ~\ref{ex:2dreduction});
\item (3D ADHM) $r_{a4} = 0$ for all $a$  (continuation of Example \ref{ex:3dreduction}).
\end{itemize}
We end the section with a ``genuinely 4D'' example: $r_{12} = r_{34} = 1$ and $r_A = 0$ otherwise (Theorem \ref{thm:crossinst}).

\subsubsection{Structure of the invariants}

We first observe the following standard property, that the cohomological partition function is obtained from the K-theoretic one by taking a certain limit (see \cite{CKM, Nek7} for similar results).
\begin{lemma} \label{lem:lim 1}
We have
\[
\lim_{b \to 0} \sfZ_{\vec{r}}^K(q) \Big|_{t_i = e^{b \cdot s_i}, w_{A,\alpha} = e^{b \cdot v_{A,\alpha}}} = \sfZ_{\vec{r}}(q).
\]
\end{lemma}
\begin{proof}
By Theorem \ref{thm:main}, we have 
\[
\sfZ_{\vec{r}}^K(q)  = \sum_{\blambda} q^{|\blambda|} [-\mathsf{v}_{\blambda}], \quad \sfZ_{\vec{r}}(q)  = \sum_{\blambda} q^{|\blambda|} e(-\mathsf{v}_{\blambda}),
\]
where $e(-)$ denotes the $\TT$-equivariant Euler class. The main point is that, for any $\blambda$ corresponding to a $\TT$-fixed point $P \in \cM^{\TT}$, the virtual character $\mathsf{v}_{\blambda}$ has rank 0 due to the normalization in Definition \ref{def:origampartfun}. Specifically, we can write
\[
\mathsf{v}_{\blambda} = \sum_{i} \tau^i - \sum_{j} \tau^j, 
\]
where we use multi-index notation $\tau = (t_1,t_2, t_3, (w_{A,\alpha}))$, $i = (i_1,i_2,i_3,(i_{A,\alpha}))$, $j = (j_1,j_2,j_3,(j_{A,\alpha}))$, we have eliminated $t_4 = (t_1t_2t_3)^{-1}$ and
\begin{equation} \label{eqn:rk0}
\sum_{i} 1 - \sum_{j} 1 = 0.
\end{equation}
Recall from the introduction that $\mathsf{v}_{\blambda}$ has no $\TT$-fixed term. 
Clearly
\[
[-\mathsf{v}_{\blambda}] \big|_{t_i = e^{b \cdot s_i}, w_{A,\alpha} = e^{b \cdot v_{A,\alpha}}} = \frac{\prod_j  \big( (j_1 s_1 + j_2 s_2 + j_3 s_3 + \sum_{A,\alpha} j_{A,\alpha} v_{A,\alpha})b + O(b^2) \big) }{\prod_i  \big( (i_1 s_1 + i_2 s_2 + i_3 s_3 + \sum_{A,\alpha} i_{A,\alpha} v_{A,\alpha})b + O(b^2) \big)}.
\]
By \eqref{eqn:rk0}, the number of factors in numerator and denominator is the same. So 
\[
\lim_{b \to 0} [-\mathsf{v}_{\blambda}] \big|_{t_i = e^{b \cdot s_i}, w_{A,\alpha} = e^{b \cdot v_{A,\alpha}}} = \frac{\prod_j   (j_1 s_1 + j_2 s_2 + j_3 s_3 + \sum_{A,\alpha} j_{A,\alpha} v_{A,\alpha}) }{\prod_i   (i_1 s_1 + i_2 s_2 + i_3 s_3 + \sum_{A,\alpha} i_{A,\alpha} v_{A,\alpha})} = e(-\mathsf{v}_{\blambda}).
\]
\end{proof}

As explained in Section \ref{sec:quivervirtualstruc}, the virtual structure sheaf is a priori a class in localized K-theory, i.e., $\widehat{\O}^{\vir} \in K_0^{\TT}(M_{Q_4}(\vec{r},n))_{\loc}$ and thus
\[
(\mathsf{C}_{\vec{r}}^K)^{-1} \cdot \chi(M_{Q_4}(\vec{r},n), \widehat{\O}^{\vir}) \in K_0^{\TT}(\pt)_{\loc} = \Q(t_1^{\frac{1}{2}},t_2^{\frac{1}{2}},t_3^{\frac{1}{2}},t_4^{\frac{1}{2}},(w_{A,\alpha}^{\frac{1}{2}} )) \quad \textrm{with \ } t_1t_2t_3t_4=1.
\]
We will now show that there exists a cohomologically $\Z/2$-graded $\TT$-equivariant coherent \emph{sheaf} $\mathcal{N}_{\vec{r},n}$ on $M_{Q_4}(\vec{r},n)$ representing the $\TT$-equivariant K-theory class $\widehat{\O}^{\vir}$ (up to an overall constant factor in the equivariant parameters). We use this to show the following ``K-theoretic integrality result''
\begin{align*}
(\mathsf{C}_{\vec{r}}^K)^{-1} \cdot \chi(M_{Q_4}(\vec{r},n), \widehat{\O}^{\vir}) \in   \Q(t_1,t_2,t_3,t_4,(w_{A,\alpha})), \quad \textrm{with \ } t_1t_2t_3t_4=1.
\end{align*}
We will use this result in the proof of Theorem \ref{thm:crossinst} at the end of this section.

The existence of $\mathcal{N}_{\vec{r},n}$ is based on a similar fact used in the context of the proof of Nekrasov's Magnificent Four conjecture \cite{Nek7} in \cite{KR}. In loc.~cit.~the virtual structure sheaf $\widehat{\O}^{\vir}$ does \emph{not} lift to a sheaf, it only lifts to a sheaf after multiplication by a suitable tautological insertion. Remarkably, in our setting $\widehat{\O}^{\vir}$ itself lifts to a sheaf (up to an overall constant factor).

We now recall some facts about the virtual structure sheaf in the general setup of Section \ref{sec:quivervirtualstruc}, where we have $E,q,\cA,s,\cM, \Lambda$ and an action by a torus $\TT$. For the discussion below, we do not need to assume $\cM^{\TT}$ is 0-dimensional and reduced. Recall that we also assume there exists a splitting\footnote{This assumption can be dropped, in which case one should work with spin modules over the sheaf of Clifford algebras associated to $(E,q)$ as in \cite{KR}. In the origami setting, however, we only encounter the split case, so we stick to 2-periodic complexes.} $E = \Lambda \oplus \Lambda^*$, where $\Lambda \subset E$ is a $\TT$-equivariant maximal isotropic subbundle. Hence, $s = (s_\Lambda, s_{\Lambda^*})$ splits accordingly. We then have a 2-periodic complex known as a Koszul factorization \cite{KO, OS}
\begin{align*}
&\mathcal{K}^{-1} = \bigoplus_{i \in \Z} \Lambda^{2i+1} \Lambda^*, \quad \mathcal{K}^{0} = \bigoplus_{i \in \Z} \Lambda^{2i} \Lambda^*,
&d = s_{\Lambda^*} \wedge + s_{\Lambda} \intprod.
\end{align*}
Then $d^2=0$ by the isotropic condition. Its cohomology sheaves are supported on $\cM$. The key fact is that, for the orientation $o = (-1)^{(r-1)n}\cdot o_{\Lambda}$, we have
\begin{align} 
\begin{split} \label{eqn:OS}
&\,(-1)^{(r-1)n}\cdot\cK \otimes \det(\Lambda^*)^{-\frac{1}{2}}  \otimes \det(T_{\cA}|_{\cM})^{-\frac{1}{2}} \\
=&\,(-1)^{(r-1)n}\cdot(H^{-1}(\cK,d)[1] \oplus H^0(\cK,d)) \otimes \det(\Lambda^*)^{-\frac{1}{2}}  \otimes \det(T_{\cA}|_{\cM})^{-\frac{1}{2}} \\
=&\,\widehat{\mathcal{O}}^{\vir}_{\cM} \in K_0^{\TT}(\cM)_{\loc},
\end{split}
\end{align}
where $H^i(-)$ are the cohomology sheaves of the 2-periodic complex $(-)$. The second equality follows from work of Oh--Sreedhar \cite{OS}. Thus, in order to lift $\widehat{\mathcal{O}}^{\vir}_{\cM}$ to a
    cohomologically $\Z/2$-graded $\TT$-equivariant coherent sheaf, we only require that $\det(\Lambda^*) \otimes \det(T_{\cA}|_{\cM})$ has a square root as a $\TT$-equivariant line bundle.\footnote{For a $\TT$-equivariant perfect obstruction theory associated to the section $s$ of a vector bundle $E$ on a smooth variety (with $\TT$-actions), $\O^{\vir}_{\cM}$ is represented by the cohomology of the Koszul complex $\bigoplus_i H^i(\Lambda^\mdot E^*,s\intprod)[i]$ \cite{Oko}. If $s = df$, then $\widehat{\O}^{\vir}_{\cM}$ also lifts to a sheaf (up to a cover $\widetilde{\TT} \to \TT$) by \eqref{eqn:NOtwist}.}We show that this is the case up to an overall constant factor in the equivariant parameters.
\begin{proposition} \label{prop:sheaflift}
There exists a cohomologically $\Z/2$-graded $\TT$-equivariant coherent sheaf $\mathcal{N}_{\vec{r},n}$ on $M_{Q_4}(\vec{r},n)$ such that, with respect to the orientation \eqref{eqn:choiceofori}, we have 
\[
[\mathcal{N}_{\vec{r},n}] = \Bigg( \prod_{A \in \bthree} t_A^{r_A r_{\overline{A}}} \prod_{\alpha=1}^{r_A} \prod_{\beta=1}^{r_{\overline{A}}} w_{A,\alpha}^{-1} w_{\overline{A},\beta} \Bigg)^{-\frac{1}{2}} \cdot \widehat{\O}^{\vir} \in K_0^{\TT}(M_{Q_4}(\vec{r},n)). 
\]
Moreover, we have
\begin{align*}
(\mathsf{C}_{\vec{r}}^K)^{-1} \cdot \chi(M_{Q_4}(\vec{r},n), \widehat{\O}^{\vir}) 
\in  \Q(t_1,t_2,t_3,t_4,(w_{A,\alpha})), \quad \mathrm{with \ } t_1t_2t_3t_4=1.
\end{align*}
\end{proposition}
\begin{proof}
Let $\cA := A_{\widehat{Q}_4}(\vec{r},n)$ and $\cM := M_{Q_4}(\vec{r},n)$. Recall that we denote by $\cV$ the tautological rank $n$ vector bundle on $\cA$ introduced at the end of Section \ref{sec:quivermoduli}. Then
\begin{align*}
T_{\cA} = &(t_1^{-1}+ t_2^{-1}+t_3^{-1}+t_4^{-1}-1) \cdot \cE{\it{nd}}(\cV) \oplus \\
&\bigoplus_{A} \cH{\it{om}}(W_A \otimes \O, \cV) \oplus \bigoplus_{A} \cH{\it{om}}(\cV,W_A \otimes \O) \cdot t_A^{-1},
\end{align*}
where $W_A = \bigoplus_{\alpha=1}^{r_A} \C \cdot w_{A,\alpha}$. We already encountered the fibre analogue of this equation over fixed points $P \in \cM^{\TT}$ in \eqref{eqn:charTU}.  Similarly, we have
\begin{align*}
\Lambda \cong &((t_1t_2)^{-1}+(t_1t_3)^{-1}+(t_2t_3)^{-1}) \cdot \cE{\it{nd}}(\cV) \\
&\oplus \bigoplus_{A \in \bthree} \Hom(W_A, W_{\overline{A}}) \otimes \O \cdot t_A \oplus \bigoplus_{A \in \bsix, \overline{a} \in \overline{A}} \hom(\cV,W_A \otimes \O) \cdot t_{\overline{a}},
\end{align*}
which over fixed points $P \in \cM^{\TT}$ reduces to \eqref{eqn:charLambda}. We deduce 
\[
\det( T_{\cA}|_{\cM} ) \otimes \det(\Lambda^*|_{\cM}) \cong \Bigg( \prod_{A \in \bthree}  t_{A}^{-r_A r_{\overline{A}}} \cdot \prod_{\alpha=1}^{r_A} \prod_{\beta=1}^{r_{\overline{A}}} w_{A,\alpha} w_{\overline{A},\beta}^{-1} \Bigg) \cdot \Bigg( \det(\cV)^{r} \cdot t_4^{-n^2} \prod_{A,\alpha} w_{A,\alpha}^{-n} \Bigg)^2,
\]
where $r := \sum_A r_A$. Then, for $(\cK,d)$ as above, we define
\begin{align*}
\mathcal{N}_{\vec{r},n} := t_4^{n^2} \cdot  \Big( \prod_{A,\alpha} w_{A,\alpha}^{n} \Big) \cdot \det(\cV)^{-r} \cdot (H^{-1}(\mathcal{K},d)[1] \oplus H^0(\mathcal{K},d))[(r-1)n].
\end{align*}
The first part of the proposition follows from \eqref{eqn:OS} and the second part from the fact that $[\mathcal{N}_{\vec{r},n}] \in K_0^{\TT}(\cM)$ (and the normalization in Definition \ref{def:origampartfun}). 
\end{proof}

We remark that the one-parameter subgroup in $\TT$ given by $t_i=1$ and $w_{A,\alpha}=w_{A',\alpha'}$ for all $i, A,\alpha, A', \alpha'$ acts trivially on $\mathcal{M}$ and on all terms of the complex $\mathcal{K}$. It follows that the equivariant parameters $w_{A,\alpha}$ enter $\chi(\cM,\widehat{\O}_{\cM}^{\vir})$ only through their pairwise quotients $w_{A,\alpha}/w_{A',\alpha'}$. Similarly, the equivariant parameters $v_{A,\alpha}$ enter $\int_{[\cM]^{\vir}} 1$ only through their pairwise differences $v_{A,\alpha}-v_{A',\alpha'}$.

\subsubsection{Invariants of the 2D and 3D ADHM quivers}

We now turn our attention to the continuation of Examples \ref{ex:2dreduction} and \ref{ex:3dreduction}.
\begin{example} \label{ex:2dreductioncont}
Consider $Q_4$ with $r_{12} = r$ and $r_A = 0$ otherwise. Then $M_{Q_4}(\vec{r},n) \cong M_{Q_2}(r,n) =: \mathcal{M}$. By Theorem \ref{thm:main} and \eqref{eqn:holsymp}, at a fixed point $P$ corresponding to $\blambda$, we have
\[
\mathsf{v}_{\blambda} = (1-t_3)t_1t_2 T_{\cM}|_P = (1-t_3) \Omega_{\cM}|_P. 
\] Therefore
\begin{equation} \label{eqn:vlambda2D}
[-\mathsf{v}_{\blambda}] = t_3^{-rn} \frac{\Lambda_{-1} (\Omega_{\cM}|_P \otimes t_3)}{\Lambda_{-1} \Omega_{\cM}|_P},
\end{equation}
where we used the identity $\widehat{\Lambda}_{-1} E = (-1)^{\rk(E)} \widehat{\Lambda}_{-1} E^*$. The symmetrized Hirzebruch $\chi_{-y}$-genus of a smooth projective variety $M$ is defined by 
\[
\widehat{\chi}_{-y}(M) := y^{-\dim(M)/2} \chi(M, \Lambda_{-1}( \Omega_M \otimes y)).
\]
For a smooth quasi-projective variety $M$ with torus action and proper fixed locus $M^T$ with connected components $M_i$, this expression is given equivariantly by
\[
\widehat{\chi}_{-y}(M) = y^{-\dim(M)/2} \sum_i \chi\Big(M_i, \frac{\Lambda_{-1}( \Omega_M \otimes y)|_{M_i}}{\Lambda_{-1}(N_{M_i / M}^*)} \Big),
\]
where $N_{M_i / M}$ denotes the normal bundle of $M_i \subset M$. By \eqref{eqn:vlambda2D}, we have
\[
\sfZ_{\vec{r}}^K(q) = \sum_ n q^n \widehat{\chi}_{-t_3}(M_{Q_2}(r,n)).
\]
This can be regarded as the rank $r$ K-theoretic Vafa--Witten generating function of $\C^2$; see for example \cite{AKL} for a precise relationship. The $r=1$ case is computed in \cite[Equation (48)]{CNO} and \cite[Theorem~5.3.13]{Oko}:
\begin{equation} \label{eqn:rk1formula}
\sfZ_{\vec{r}}^K(q) = \mathrm{Exp} \Bigg( \frac{[t_1t_3][t_2t_3]}{[t_1][t_2]} \frac{q}{1-q} \Bigg).
\end{equation}
Here, for $f(t_1, \ldots, t_r, q) \in \Q(t_1, \ldots, t_r)[\![q]\!]$ satisfying $f(t_1, \ldots, t_r,0)=0$, its plethystic exponential $\Exp(f)$ is defined by
\begin{align}\label{def:pleth} 
\Exp(f(t_1, \ldots, t_r,q)) &:= \exp\Big( \sum_{n=1}^{\infty} \frac{1}{n} f(t_1^n, \ldots, t_r^n,q^n) \Big)
\end{align}
viewed as an element of $\Q(t_1, \ldots, t_r)[\![q]\!]$. The plethystic exponential in \eqref{eqn:rk1formula} is obtained by expanding the exponent in ascending powers of $q$.
\end{example}

\begin{example} \label{ex:3dreductioncont}
Consider $Q_4$ with $r_{a4} = 0$ for all $a$. Then $M_{Q_4}(\vec{r},n) \cong M_{Q_3}(\vec{r},n) =: \cM$ and the latter is the critical locus of a regular function $f \colon \cA := A_{\widehat{Q}_3}(\vec{r},n) \to \C$ as described in Section \ref{sec:quivervirtualstruc}. This description induces a $\TT$-equivariant 2-term perfect obstruction theory. Thus, we have two virtual cycles and virtual structure sheaves
\begin{align*}
[\cM]_{\mathrm{OT}}^{\vir}, \quad [\cM]_{\mathrm{BF}}^{\vir} &\in A_0^{\TT}(\cM,\Z[\tfrac{1}{2}]), \\
\widehat{\O}^{\vir}_{\cM, \mathrm{OT}}, \quad \widehat{\O}^{\vir}_{\cM, \mathrm{BF}} &\in K_0^{\TT}(\cM)_{\loc},
\end{align*}
where the subscript $(-)_{\mathrm{OT}}$ indicates a class induced from the 3-term symmetric obstruction theory associated to the isotropic section $s$ (and orientation \eqref{eqn:choiceofori}), and $(-)_{\mathrm{BF}}$ a class induced from the 2-term obstruction theory given by the regular function $f$. Note that  $\mathsf{C}_{\vec{r}}^K = \mathsf{C}_{\vec{r}} = 1$ (Remark \ref{rem:n=0}). We claim that 
\begin{align*}
\int_{[\cM]^{\vir}_{\mathrm{OT}}} 1 = \int_{[\cM]^{\vir}_{\mathrm{BF}}} 1, \quad  \chi(\cM, \widehat{\O}^{\vir}_{\cM, \mathrm{OT}}) = \chi(\cM, \widehat{\O}^{\vir}_{\cM, \mathrm{BF}}).
\end{align*}
This result follows from the following lemma. 
\end{example}

\begin{lemma} \label{lem:3Dcompare}
For a fixed point $P \in \cM^{\TT}$ corresponding to integer partitions $\blambda$, we have
\[
\mathsf{v}_{\blambda} - \Big(T_{\mathcal{A}}|_P - \Omega_{\cA}|_P \otimes (t_1t_2t_3)^{-1} \Big) = G_{\blambda} - G^*_{\blambda},
\]
for some class $G_{\blambda} \in K_0^{\TT}(\pt)$ satisfying $\rk G^m_{\blambda} \equiv 0 \mod 2$, where $(-)^m$ denotes the $\TT$-moving part. 
\end{lemma}
\begin{proof}
Along the lines of \eqref{eqn:charTU}, the $\TT$-character of $T_{\cA}|_P$ is given by
\[
T_{\cA}|_P = (t_1^{-1}+t_2^{-1}+t_3^{-1} -1) \cdot \sum_{A,B \in \bthree} K_A^* K_B + \sum_{A,B \in \bthree} N_A^* K_B + \sum_{A,B \in \bthree} K_A^* N_B \cdot t_B^{-1}. 
\]
Using \eqref{eqn:Neknot}, the proof is analogous to that of Proposition \ref{prop:compare}. For instance, the following choice of $G_{\blambda}$ works:
\begin{align*}
G_{\blambda} = &(t_1+t_2+t_3+t_1^{-1}t_2^{-1}+t_1^{-1}t_3^{-1}+t_2^{-1}t_3^{-1}+t_1t_2t_3) \sum_{A \in \bthree} K_A K_A^* \\
&+ \sum_{A \in \bthree} (t_A - t_1t_2t_3 - 1 + t_c^{-1}) N_A^* K_A \\
&+ (-1+t_1+t_2+t_3-t_1t_2-t_1t_3-t_2t_3+t_1t_2t_3)\sum_{A < B \in \bthree} K_A K_B^* \\
&+ \sum_{A \neq B \in \bthree} (1-t_c) N_A K_B^*,
\end{align*}
where $c$ always denotes the unique element of $\{1,2,3\} \setminus A$. 
\end{proof}

It is now easy to determine the invariants of Example \ref{ex:3dreductioncont} under the Calabi--Yau-3 specialization $s_1+s_2+s_3=0$. This specialization corresponds to the restriction to  
$$
T_0 := Z(t_1t_2t_3-1) \leq T, \quad \TT_0 := T_0 \times (\C^*)^r \leq \TT.
$$
\begin{proposition} \label{prop:CY3spec}
Let $r_{a4} = 0$ for all $a$. Then the specialization $\mathsf{Z}_{\vec{r}}(q)|_{s_1+s_2+s_3 = 0}$ is well-defined and we have\footnote{Since $\vd = 0$, we have $\mathsf{Z}^K_{\vec{r}}(q)|_{t_1t_2t_3=1} = \mathsf{Z}_{\vec{r}}(q)|_{s_1+s_2+s_3 = 0}$ so we only need to consider the latter.}
\[
\mathsf{Z}_{\vec{r}}(q)|_{s_1+s_2+s_3 = 0} = \overline{\eta}(q)^{-r},
\]
where $\overline{\eta}(q) = \prod_{n>0}(1-q^n)$ denotes the normalized Dedekind eta function.
\end{proposition}
\begin{proof} 
Note that $M_{Q_3}(\vec{r},n)^{\TT_0}=M_{Q_3}(\vec{r},n)^{\TT}$ by Proposition \ref{prop:fixlocquiver}. 
For any $P \in \cM := M_{Q_3}(\vec{r},n)^{\TT}$, consider $T_{\cA}|_P - \Omega_{\cA}|_P$. Evidently, it has no $\TT_0$-fixed part. So, the desired specialization is well-defined. The contribution of $P$ to $\mathsf{Z}_{\vec{r}}(q)|_{s_1+s_2+s_3 = 0}$ is given by
\[
(-1)^{\dim T_{\cA}|_P^m} q^n,
\]
where $(-)^m$ denotes the moving part with respect to $\TT_0$. Let $\blambda$ be the integer partitions corresponding to $P$. An easy calculation using Lemma \ref{lem:comb} (but working modulo the relation $t_1t_2t_3=1$) yields 
\begin{align}
\begin{split} \label{eqn:3Dparitycalc}
\rk T_{\cA}|_P^m &= \rk T_{\cA}|_P - \rk T_{\cA}|_P^f \\
&= 2n^2 + 2rn - 2\sum_{A,\alpha} \Big( |\lambda_{A,\alpha}| - h_{\lambda_{A,\alpha}}(0,0) \Big) \\
&\equiv 0 \mod 2.
\end{split}
\end{align}
Therefore, the desired generating function is given by $\sum_{\blambda} q^{|\blambda|} = \overline{\eta}(q)^{-r}$.
\end{proof}

\begin{remark}
Consider $\cM := M_{Q_3}(\vec{r},n)$. Instead of $\int_{[\cM]^{\vir}_{\mathrm{BF}}} 1 \in \Z$ (defined with respect to the torus $\TT_0$), one can also define invariants by integrating Behrend's constructible function $\nu_{\cM} : \cM(\C) \to \Z$ \cite{Beh}
\[
\int_{\cM} \nu_{\cM} \, de := \sum_k k \cdot e(\nu_{\cM}^{-1}(\{k\})) \in \Z,
\]
where $e(-)$ denotes the topological Euler characteristic. As $\cM$ is non-compact, this could (a priori) give a different result from equivariant integration over the virtual cycle. By Proposition \ref{prop:fixlocquiver}, we have $\cM^{\TT} = \cM^{\TT_0}$, which is 0-dimensional and reduced. In particular
\[
\int_{\cM} \nu_{\cM} \, de = \sum_{P \in \cM^{\TT_0}} \nu_{\cM}(P).
\]
We have $T_{\cA}|_P - \Omega_{\cA}|_P = T_{\cM}|_P - \Omega_{\cM}|_P$, where $T_{\cM}|_P:=\Omega_{\cM}|_P^*$ denotes the Zariski tangent space of $\cM$ at $P$. As $\cM^{\TT_0}$ is 0-dimensional and reduced, the restriction $T_{\cM}|_P$ has no $\TT_0$-fixed part. Thus we determine the Behrend value
\[
1 = (-1)^{\rk T_{\cA}|_{P}^m} = (-1)^{\rk T_{\cM}|_{P}^m} = (-1)^{\rk T_{\cM}|_{P}}  = \nu_{\cM}(P),
\]
where the first equality follows from \eqref{eqn:3Dparitycalc} and last equality is given by \cite[Thm.~3.4]{BF2}. Therefore, the Behrend value at the fixed points is $+1$, so the virtual and Behrend invariants agree
\[
\int_{\cM} \nu_{\cM} \, de = \int_{[\cM]^{\vir}_{\mathrm{BF}}} 1.
\]
This situation is similar to the case of $\Hilb^n(\C^3)$ studied in \cite{MNOP, BF2}, except that in our setting the Behrend value at the fixed points is always $+1$.\footnote{We caution the reader that this is not enough evidence to conjecture that $\nu_{\cM}$ is constant \cite{JKS}.} 
\end{remark}

We observe that the invariants of Example \ref{ex:3dreductioncont} also appear to have interesting modular properties under other specializations.

\begin{example}
Let $r_{12}=r_{13} = 1$ and $r_A = 0$ otherwise. Consider the Jacobi theta function $\theta_2(q) = \sum_{n \in \Z + \frac{1}{2}} q^{n^2}$, Dedekind eta function $\eta(q)$  and its normalized version $\overline{\eta}(q)$. For the diagonal specialization $s_1=s_2=s_3$ and $e_{12,1} = e_{13,1}$, the equivariant parameters drop out and we find (by explicit calculation) 
\[
\frac{\sfZ_{\vec{r}}(q)}{\overline{\eta}(q)^{-8}} = \frac{\eta(q^4)^2}{\eta(q^2)\eta(q)^6} \mod q^5.
\]
The right hand side can also be written as $\theta_2(q) / 2\eta(q)^6$. We expect that this equality holds to all orders in $q$. 
\end{example}

\begin{remark}
    Let $\mathbf{M}_{Q_2}(r,n)$, $\mathbf{M}_{Q_3}(\vec{r},n)$ and $\mathbf{M}_{Q_4}(\vec{r},n)$ be the derived moduli spaces of stable representations of the 2D, 3D and 4D ADHM quivers, where the derived structures are induced from the smooth structure, critical locus description, and isotropic zero locus description, respectively (Section \ref{sec:quivervirtualstruc}). They are $0$-shifted symplectic, $(-1)$-shifted symplectic, and $(-2)$-shifted symplectic, respectively, in the sense of \cite{PTVV}.  The fact that the invariants of the 2D and 3D ADHM quivers are obtained as specializations of the 4D ADHM quiver (Examples \ref{ex:2dreductioncont}, \ref{ex:3dreductioncont}) suggests the following geometric relationship between the derived moduli spaces. If $r_{12}=r$ and $r_{A}=0$ otherwise, then we expect that
    $$\mathbf{M}_{Q_3}(\vec{r},n) \simeq \mathrm{Tot}\Big(\mathbb{L}_{\mathbf{M}_{Q_2}(r,n)}[-1]\Big),
    $$
    where $\mathbb{L}[n]$ denotes the $n$-shifted cotangent complex and $\mathrm{Tot}(-)$ denotes the total space. Similarly, if $r_{a4}=0$ for all $a$, then we expect that
    $$\mathbf{M}_{Q_4}(\vec{r},n) \simeq \mathrm{Tot}\Big(\mathbb{L}_{\mathbf{M}_{Q_3}(\vec{r},n)}[-2]\Big).
    $$
\end{remark}

\subsubsection{Crossed instantons}

In the rest of this section, we study $Q_4$ with dimension vector $r_{12} = r_{34} = 1$ and $r_A = 0$ otherwise. We prove Theorem \ref{thm:crossinst}. We start with the following proposition.
\begin{proposition} \label{prop:Jiszero}
Let $P = [(B_a,I_A,J_A)] \in M_{Q_4}(\vec{r},n)$ with $r_{12} = r_{34} = 1$ and $r_A = 0$ otherwise. Then $J_A = 0$ for all $A \in \bsix$. 
\end{proposition}
\begin{proof}
Without loss of generality, we focus on the case $A = \{3,4\}$. By using the stability condition \eqref{eqn:ADHMstab}, we can write any $v \in V = \C^n$ as a linear combination of vectors of the form $F I_{12} w$ and $F' I_{34} w'$ where $w \in W_{12}$, $w' \in W_{34}$, and $F, F' \in \C \langle B_1, \ldots, B_4 \rangle$. By \eqref{eqn:4dADHM}, 
\[
J_{34} F I_{12} w = 0.
\]
Furthermore, all monomials $M$ appearing in $F'$ which contain at least one $B_1$ or $B_2$ have the property that $J_{34} M I_{34} w' = 0$. It remains to consider $F'$ that depend only on $B_3, B_4$. Note that so far, we only used that $r_{A} = 0$ for $A \neq \{1,2\}, \{3,4\}$. For the next part, we require $r_{12} = r_{34} = 1$.

Now, for any monomial $M$ in $B_3, B_4$, we have 
\[
J_{34} M I_{34} = \tr(J_{34} M I_{34}) = 0.
\]
We claim that the second equality follows from repeated use of the equation $[B_3,B_4] + I_{34} J_{34} = 0$ and cyclicity of trace as in \cite[Prop.~2.8]{Nak}. This can be shown by induction on the length $\ell$ of $M$. Note that the case $\ell = 0$ follows immediately from $\tr(J_{34} I_{34})  = \tr(I_{34} J_{34})= - \tr([B_3,B_4]) = 0$. 
\end{proof}

As a direct consequence of Propositions \ref{prop: Quot scheme} and \ref{prop:Jiszero}, we have the following result.
\begin{corollary}
Suppose $r_{12} = r_{34} = 1$ and $r_A = 0$ otherwise. Then $$M_{Q_4}(\vec{r},n) \cong \Quot_{\C^4}\Big(\bigoplus_{A}\O_{\C^2_A}^{\oplus r_A}, n\Big).$$
\end{corollary}

We require one more lemma before we can prove Theorem \ref{thm:crossinst}. We define $w := w_{12,1}$, $w' := w_{34,1}$ and let $\blambda = (\lambda_{12}, \lambda_{34}) = (\lambda,\lambda')$ be a pair of integer partitions.
\begin{lemma} \label{lem:lim 2}
Suppose $r_{12} = r_{34} = 1$ and $r_A = 0$ otherwise. Then
\[
\lim_{L \to \infty} [-\mathsf{v}_{\blambda}] \Big|_{w = 1, w' = L} = [-P_{3} T_{\lambda}][-P_{1} T_{\lambda'}] = \lim_{L \to 0} [-\mathsf{v}_{\blambda}] \Big|_{w = 1, w' = L},
\] 
where $T_{\lambda} := T_{12}$ and $T_{\lambda'} := T_{34}$ were defined in \eqref{eqn:Neknot}.
\end{lemma}
\begin{proof}
This follows from a straightforward calculation. The claim is that, for $w=1$, $w' = L$, the expressions 
\[
[-P_{12} N_{34} K_{12}^*], \quad [-P_{34} N_{12} K_{34}^*], \quad [P_{1234} K_{12} K_{34}^*]
\]
all become 1 for $L ^{\pm 1} \to \infty$. Indeed, for the first term
\[
[-P_{12} N_{34} K_{12}^*] \Big|_{w = 1, w' = L} = \prod_{(i,j) \in \lambda} \frac{(t_1^{-\frac{i-1}{2}} t_2^{-\frac{j}{2}} L^{\frac{1}{2}} - t_1^{\frac{i-1}{2}} t_2^{\frac{j}{2}} L^{-\frac{1}{2}})  (t_1^{-\frac{i}{2}} t_2^{-\frac{j-1}{2}} L^{\frac{1}{2}} - t_1^{\frac{i}{2}} t_2^{\frac{j-1}{2}} L^{-\frac{1}{2}})}{(t_1^{-\frac{i}{2}} t_2^{-\frac{j}{2}} L^{\frac{1}{2}} - t_1^{\frac{i}{2}} t_2^{\frac{j}{2}} L^{-\frac{1}{2}}) (t_1^{-\frac{i-1}{2}} t_2^{-\frac{j-1}{2}} L^{\frac{1}{2}} - t_1^{\frac{i-1}{2}} t_2^{\frac{j-1}{2}} L^{-\frac{1}{2}})}
\]
which approaches 1 as $L^{\pm 1} \to \infty$. The limits of other terms are computed similarly.
\end{proof}

\begin{proof}[Proof of Theorem \ref{thm:crossinst}]
Recall that $w := w_{12,1}$, $w' := w_{34,1}$.
    As the diagonal of the framing torus $(\C^*)^r$ acts trivially on the origami moduli spaces, we may set $w=1$. By Proposition \ref{prop:sheaflift}, we have 
\[
( \mathsf{C}_{\vec{r}}^K)^{-1} \cdot \chi(M_{Q_4}(\vec{r},n), \widehat{\O}^{\vir}) = \frac{\chi(M_{Q_4}(\vec{r},n), \mathcal{N}_{\vec{r},n})}{1-(w' t_1t_2)^{-1}}. 
\]
We use the Quot-to-Chow morphism of \cite[Cor.~7.15]{Ryd} 
\[
M_{Q_4}(\vec{r},n) \cong \Quot_{\C^4}\Big(\bigoplus_{A}\O_{\C^2_A}^{\oplus r_A}, n\Big) \to \Sym^n(\C^4).
\]
Since this map is $\TT$-equivariant and proper, we have 
\[
R \nu_* \mathcal{N}_{\vec{r},n} \in K_0^{\TT}(\Sym^n(\C^4)).
\]
Moreover, the framing torus $(\C^*)^r$ acts trivially on $\Sym^n(\C^4)$, so we can write
\[
R \nu_* \mathcal{N}_{\vec{r},n} = \sum_i V_i \cdot w^{\prime i}, \quad V_i \in K_0^T(\Sym^n(\C^4)),
\]
where the sum is finite and $T := Z(t_1t_2t_3t_4-1)$. Hence, we have
\[
\frac{\chi(M_{Q_4}(\vec{r},n), \mathcal{N}_{\vec{r},n})}{1-(w' t_1t_2)^{-1}} = \frac{f(w')}{1-(w' t_1t_2)^{-1}}, \quad f(w') \in \Q(t_1,t_2,t_3,t_4)[w^{\prime \pm 1}], \quad t_1t_2t_3t_4=1.
\]
By Theorem \ref{thm:main} and Lemma \ref{lem:lim 2}, the limits $\lim_{L^{\pm 1} \to \infty} f(L) / (1-(L t_1t_2)^{-1})$ exist so we have $f(w') = \alpha + \beta (w')^{-1}$ for some $\alpha, \beta \in \Q(t_1,t_2,t_3,t_4)$ where $t_1t_2t_3t_4=1$. In fact, these limits are equal so we have $\alpha  = - t_1 t_2 \beta$. Consequently $f(w') / (1-(w' t_1t_2)^{-1}) = \alpha$ is independent of $w'$.\footnote{We stress that the contribution of individual $\TT$-fixed points depends on $w'$ --- the dependence on $w'$ disappears only after summation.} Therefore, again by Theorem \ref{thm:main} and Lemma \ref{lem:lim 2}, we have
\begin{align*}
( \mathsf{C}_{\vec{r}}^K)^{-1} \cdot \chi(M_{Q_4}(\vec{r},n), \widehat{\O}^{\vir}) &= \lim_{L \to \infty} ( \mathsf{C}_{\vec{r}}^K)^{-1} \cdot \chi(M_{Q_4}(\vec{r},n), \widehat{\O}^{\vir}) \\
&= \sum_{\substack{\lambda, \lambda' \\ |\lambda| + |\lambda'| = n}}  [-P_{3} T_{\lambda}][-P_{1} T_{\lambda'}],
\end{align*}
and the theorem follows from Example \ref{ex:2dreductioncont}.
\end{proof}

\subsubsection{Stable/co-stable wall-crossing} \label{sec:stabcostab}

In the definition of $M_{Q_4}(\vec{r},n)$, one can replace the stability condition by the following condition referred to as co-stability: 
\begin{center}\vspace{3pt}
    there is no nonzero subspace $V'\subseteq \bigcap\limits_{A} \ker(J_A)$ that is invariant under $B_1,\dots, B_4$.\vspace{3pt}
\end{center}
The co-stability condition is equivalent to King stability with respect to $(\Theta_{\infty} = -n, \Theta_0 = 1)$ after applying the Crawley-Boevey trick as in the proof of Proposition \ref{prop:King}. Then Section \ref{sec:quivermoduli} holds analogously in this setup. We denote the resulting moduli space by $M_{Q_4}^c(\vec{r},n)$. In particular, there are (projective) semisimplification morphisms
\begin{displaymath}
\xymatrix
{
M_{Q_4}(\vec{r},n) \ar[dr] & & \ar[dl] M_{Q_4}^c(\vec{r},n) \\
& M_{Q_4}^{\mathrm{ssimp}}(\vec{r},n). &
}
\end{displaymath}
Moreover, we can use the same $\TT$-action on $M_{Q_4}^c(\vec{r},n)$ as before, namely \eqref{eqn:Taction}. As we will see below, the $\TT$-fixed locus of $M_{Q_4}^c(\vec{r},n)$ is 0-dimensional and reduced. Then Section \ref{sec:quivervirtualstruc} holds analogously in this setting, so we have an Oh--Thomas virtual cycle and virtual structure sheaf for $M_{Q_4}^c(\vec{r},n)$ (using the analogue of the orientation defined in \eqref{eqn:choiceofori}). We conjecture that the wall-crossing between the invariants of these two moduli spaces is trivial.
\begin{conjecture} \label{conj:stab-costab}
For all $\vec{r}$, $n$, we have
$$
\chi(M_{Q_4}(\vec{r},n),\widehat{\O}^{\vir}) = \chi(M_{Q_4}^c(\vec{r},n),\widehat{\O}^{\vir}).
$$
\end{conjecture}

This conjecture implies an interesting symmetry of the origami partition function.
\begin{corollary} \label{cor:symm}
Suppose Conjecture \ref{conj:stab-costab} holds. Then we have
$$
\mathsf{Z}_{\vec{r}}^K(q) \Big|_{(w_{A,\alpha}) \mapsto (t_A w_{A,\alpha}^{-1})} = \mathsf{Z}_{\vec{r}}^K(q),
$$
that is to say, $\mathsf{Z}_{\vec{r}}^K(q)$ is invariant under replacing all $w_{A,\alpha}$ by $t_A w_{A,\alpha}^{-1}$.
\end{corollary}

\begin{proof}
The map
\[
\sigma \colon M_{Q_4}(\vec{r},n) \rightarrow M_{Q_4}^c(\vec{r},n), \quad [(B_a,I_A,J_A)] \mapsto [(B_a^t, -J_A^t, I_A^t)]
\]
defines an isomorphism. This map only intertwines the $\TT$-actions up to an automorphism of $\TT$, i.e.,
\[
\sigma((t_a,w_A) \cdot [(B_a,I_A,J_A)]) = (t_a, t_A w_A^{-1}) \cdot \sigma([(B_a,I_A,J_A)]),
\]
for all $(t_a,w_A) \in \TT$, where $w_A$ denotes the diagonal matrix with entries $w_{A,\alpha}$ as before. In particular, it follows that $M_{Q_4}^c(\vec{r},n)$ has a 0-dimensional reduced $\TT$-fixed locus. The map $\sigma$ is compatible with the virtual structures up to the automorphism $(w_{A,\alpha}) \mapsto (t_A w_{A,\alpha}^{-1})$ of $\TT$, and we conclude
$$
\chi(M_{Q_4}(\vec{r},n),\widehat{\O}^{\vir}) = \chi(M_{Q_4}^c(\vec{r},n),\widehat{\O}^{\vir})\Big|_{(w_{A,\alpha}) \mapsto (t_A w_{A,\alpha}^{-1})}.
$$
The result follows from Conjecture \ref{conj:stab-costab}.
\end{proof}

We stress that, for arbitrary $\vec{r}$, the operation $(w_{A,\alpha}) \mapsto (t_A w_{A,\alpha}^{-1})$ does not simply give rise to a permutation of the contributions of the $\TT$-fixed points. The invariance in the corollary only becomes manifest after summing all the $\TT$-fixed points. In particular, we are not aware of a simple combinatorial argument for this invariance. We have the following verifications of Conjecture \ref{conj:stab-costab} and its corollary.
\begin{enumerate}
\item By Example \ref{ex:2dreductioncont}, for $r_{12} = r$ and $r_A = 0$ otherwise, the statement in the conjecture reduces to 
\begin{equation} \label{eqn:AKL}
\widehat{\chi}_{-t_3}(M_{Q_2}(\vec{r},n)) = \widehat{\chi}_{-t_3}(M_{Q_2}^c(\vec{r},n)),
\end{equation}
which was recently proved in \cite{AKL} using mixed Hodge modules and results of \cite{DHSM} on the BPS sheaf. This proof suggests that the derived push-forward of $\widehat{\O}^{\vir}$ to $M_{Q_4}^{\mathrm{ssimp}}(\vec{r},n)$ behaves like a kind of BPS sheaf and is in particular independent of the choice of stability. The cohomological analogue of \eqref{eqn:AKL} was also proved using Mochizuki style wall-crossing in \cite{Ohk}.
\item If $r_{12}=r_{34}=1$ and $r_{A}=0$ otherwise, the origami partition function is independent of the framing equivariant parameters by Theorem \ref{thm:crossinst}, so Corollary \ref{cor:symm} holds trivially.
\item Using SAGE, we verified Corollary \ref{cor:symm} in the following cases:
\begin{itemize}
    \item $\sum_A r_A = 2$ and $n\leq 4$ (i.e..~for the first 5 terms of the generating series),
    \item $\sum_A r_A = 3$ and $n\leq 2$ (i.e.,~for the first 3 terms of the generating series),
    \item $\sum_A r_A = 4$ and $n\leq 1$ (i.e.,~for the first 2 terms of the generating series).
\end{itemize}
\end{enumerate}

\subsection{Non-perturbative Dyson--Schwinger equations}\label{sec:DS}

Nekrasov's motivation for the construction of crossed instantons is to provide a geometric explanation for so-called non-perturbative Dyson--Schwinger equations, a class of certain recurrence relations satisfied by generating series of integrals over Nakajima quiver varieties, a class of spaces including the moduli spaces $M_{Q_2}(r,n)$. The resulting identities verify predictions arising from a physical correspondence between 4-dimensional gauge theories and 2-dimensional conformal field theories, see e.g. \cite{Nek2}. In the physics literature, a concise summary of Nekrasov's approach can be found in \cite[Sect.~6]{Tam}.

In the mathematics literature, examples of non-perturbative Dyson--Schwinger equations have fruitfully been applied in probability (see, for one of many examples, \cite[Sect.~4]{BBG}), and in conformal field theory \cite{NT}, to match Nekrasov partition functions with certain conformal blocks.  The identities are proved on an as-needed basis in these applications; see \cite[App.~A]{NT}.

Nekrasov's origami moduli spaces provide a uniform, geometric approach to proving many classes of non-perturbative Dyson--Schwinger equations. The algebraic virtual cycle construction and associated sign computation of the previous sections are needed to make Nekrasov's approach mathematically rigorous.

The key geometric idea is that the compactness properties of gauge origami moduli spaces translate to the regularity in the equivariant parameters of the terms of the series $\sfZ_{\vec{r}}(q)$. We recall the following result, which can be regarded as a cohomological version of \cite[Prop.~3.2]{Arb}; see also \cite[Sect.~5.5--6]{Liu2}.

Let $M$ be a scheme of finite type over $\C$ equipped the action of an algebraic torus $T$ such that the fixed locus $M^T$ is proper. We consider Chow groups with $\Q$-coefficients. By \cite[Sect.~2.3, Cor.~2]{Bri}, the map $A^T_*(M^T)\to A^T_*(M)$ induced by the inclusion $M^T\hookrightarrow M$ becomes an isomorphism after inverting all nontrivial $T$-characters (in fact, it suffices to invert finitely many $T$-characters). Composing its inverse with push-forward from $M^T$ to a point, one obtains a well-defined map $A^T_*(M)\to \mathrm{Frac}(A^T_*(\pt))$ given by integration of $1\in A_T^*(M)$ against the cycle.

Given a nontrivial $T$-weight $e^{x}=w \colon  T \to \C^*$, and let $T_w$ denote the maximal torus contained in $\ker(w).$  

\begin{proposition}\label{prop:cohpole}
If the fixed locus $M^{T_w}$ with respect to $T_w\subset T$ is proper, then for any $\alpha\in A^T_*(M)$,  the rational function $$\int_\alpha 1\in \mathrm{Frac}(A^T_*(\mathrm{pt}))$$ has no pole at $x=0$.
\end{proposition}
\begin{proof}
The inclusion $i \colon  M^{T_w}\to M$ induces a morphism $i_* \colon A^{T}_*(M^{T_w})\to A^T_*(M)$. We use the strengthening \cite[Prop.~2.15]{Gon} of \cite[Sect.~2.3, Cor.~2]{Bri}, which asserts that the map $i_*$ becomes an isomorphism after inverting finitely many characters $u_i \colon T\to \C^*,$ all of whose restrictions to $T_w$ are nontrivial. Let $(i_*)_{\mathrm{loc}}$ denote this isomorphism and $\pi \colon M^{T_w}\to \mathrm{pt}$ be the projection. By properness of $\pi$ and linearity of $\pi_*$ over $A^T_*(\mathrm{pt})$, we have  
\begin{equation*}
\int_{\alpha}1=\pi_*((i_*)_{\mathrm{loc}}^{-1}\alpha)\in A^T_*(\mathrm{pt})[c^{T}_1(u_i)^{-1}]\subset \mathrm{Frac}(A^{T}_*(\mathrm{pt})). \qedhere
\end{equation*}
\end{proof}

We specialize the above discussion to moduli spaces of ``crossed instantons.'' Consider $M_{Q_4}(\vec{r},n)$ for $\vec{r}$ such that $r_{12}, r_{34}>0$ and $r_A=0$ otherwise and the torus $\TT \cong  (\C^*)^3\times (\C^*)^{r_{12}}\times (\C^*)^{r_{34}}$ defined in Section \ref{subsec:torus}. 
Let $u \colon \TT\to \C^*$ be the character given by $(t_i,w_{12,\alpha},w_{34,\beta})\mapsto \prod_{\alpha}w_{12,\alpha},$ and set $x=c^{\TT}_1(u)$; in other words $u=e^{x}.$

In \cite[Sect.~8]{Nek3}, an argument is presented for the compactness of the fixed loci of $M_{Q_4}(\vec{r},n)^{\TT'}$ for classes of subtori $\TT'\subset \TT$. One consequence is that the fixed locus  $M_{Q_4}(\vec{r},n)^{\TT_u}$ for $u$ as defined in the previous paragraph is proper. Applying Proposition \ref{prop:cohpole}, one concludes the following statement of  \cite[Sect.~3.4]{Nek4}: expanding $\mathsf{C}_{\vec{r}} \cdot \mathsf{Z}_{\vec{r}}(q)$ in descending powers of $x$, one has 
\begin{equation} \label{eqn:DSinthm}
[x^{-k}]\left(\mathsf{C}_{\vec{r}} \cdot \mathsf{Z}_{\vec{r}}(q)\right)=0\quad \textnormal{for all}\ k>0,
\end{equation}
where $[x^{-k}](-)$ denotes the $x^{-k}$-coefficient.

\begin{remark}
The fact that the non-normalized partition function $\mathsf{C}_{\vec{r}} \cdot \mathsf{Z}_{\vec{r}}(q)$ is defined using \emph{globally defined} classes $[M_{Q_4}(\vec{r},n)]^{\vir}$ is crucial to deduce equation \eqref{eqn:DSinthm} geometrically. 
\end{remark}

Moreover, performing localization with respect to the one-parameter subgroup $\sigma \colon \C^*\to(\C^*)^3\times (\C^*)^{r_{12}}\times(\C^*)^{r_{34}}$ given by $t\mapsto (1,(t,\ldots,t),1),$ one can express integrals against the virtual class of $M_{Q_4}(\vec{r},n)$ in terms of integrals over $$M_{Q_4}(\vec{r},n)^{\sigma}\cong \bigsqcup_{n_1+n_2=n} M_{Q_2}(r_{12},n_1)\times M_{Q_2}(r_{34},n_2).$$

By Theorem \ref{thm:main} and equivariant localization, the virtual
invariant $\int_{[M_{Q_4}(\vec{r},n)]^{\vir}} 1$ can then be written as
\begin{align*}
\sum_{n_1+n_2=n} \int_{M_{Q_2}(r_{12},n_1)\times M_{Q_2}(r_{34},n_2)} &e(T_{M_{Q_2}(r_{12},n_1)}\otimes e^{-s_3})e(T_{M_{Q_2}(r_{34},n_2)}\otimes e^{-s_1}) \cdot \\
&e(\mathcal{W}\otimes e^{-x+s_1+s_2}), \end{align*}
for an explicit equivariant $K$-class $\mathcal{\cW}$ expressible in terms of the tautological bundles on $M_{Q_2}(r_{12},n_1)$ and $M_{Q_2}(r_{34},n_2)$; see \cite[Eqn.~(6.8)--(6.11)]{Tam}. For each $k>0$, extracting coefficients of powers of $x^{-k}$, the equations \eqref{eqn:DSinthm} therefore yield relations of the form \begin{align}\label{eqn: DSform} 0=\sum_{n_1\leq n} \int_{M_{Q_2}(r,n_1)} e(T_{M_{Q_2}(r_{12},n_1)}\otimes e^{-s_3})f_k(\mathcal{V})\in \Q(s_1,s_2,s_3, w_{12,\alpha},w_{34,\beta}),\end{align} 
where $\mathcal{V}$ is a tautological bundle on $M_{Q_2}(r,n_1)$ and $f_k$ is an explicit (but often complicated) function on a localization of $K^{\TT}_0(M_{Q_2}(r,n))$, expressed in terms of Chern roots. In \cite{Nek2}, Nekrasov explains how to interpret the vanishings \eqref{eqn: DSform} as regularity relations for $qq$-characters, one-parameter deformations of $q$-characters studied in the theory of quantum affine algebras.

As explained in \cite{Nek2}, in the case $r_{12}=r_{34}=1$, the relations \eqref{eqn:DSinthm} for small values of $k$ enjoy a particularly nice form. Let $(\C^*)^2$ with equivariant parameters $e^{s_1},e^{s_2}$ act on $M_{Q_2}(1,n)\cong \Hilb^n(\C^2)$ as in Section \ref{subsec:torus}, and set $$\mathsf{G}(q;s_1,s_2,s_3):=\sum_{n=0}^{\infty} q^n \int_{M_{Q_2}(1,n)} e(T_{M_{Q_2}(1,n)} \otimes e^{-s_3}).$$ 

\begin{proposition}\cite[Eqn.~(265)]{Nek2}\label{prop:DS}
Let $r_{12}=r_{34}=1$ and $r_A = 0$ otherwise. Then the Dyson--Schwinger equation \eqref{eqn:DSinthm} for $k=1$ is the differential equation
\[
s_3 s_4 \mathsf{G}(q;s_1,s_2,s_3) \cdot \frac{d}{dq} \mathsf{G}(q;s_3,s_4,s_1) + s_1 s_2 \mathsf{G}(q;s_3,s_4,s_1) \cdot \frac{d}{dq} \mathsf{G}(q;s_1,s_2,s_3) = 0.
\]
\end{proposition}
\begin{proof} 
As $r_{12} = r_{34} = 1$, we have 
$x=v_{12,1}$ and $\mathsf{C}_{\vec{r}} = -x + v_{34,1} - s_3 - s_4$. We will show $[x^{-1}](\mathsf{Z}_{\vec{r}}(q)) = 0$. It then follows that the Dyson--Schwinger equation \eqref{eqn:DSinthm} for $k = 1$ equals
\[
[x^{-2}](\mathsf{Z}_{\vec{r}}(q)) = 0.
\]

Using the notation \eqref{eqn:Neknot}, we have
\[
\mathsf{v}_{\blambda}^{\sigma-\mathrm{fix}} = P_3 T_{12} + P_1 T_{34}, \ \ \mathsf{v}_{\blambda}^{\sigma-\mathrm{mov}} = P_{34} {N}_{12} K_{34}^* + P_{12} N_{34} {K}_{12}^* - P_{1234} {K}_{12} K_{34}^*.
\]
Note that a character is $\sigma$-fixed if and only if it exhibits no dependence on $u$. Moreover,  $\mathsf{v}_{\blambda}$ is a rank zero complex and each of the three summands of $\mathsf{v}_{\blambda}^{\sigma-\mathrm{mov}}$ is of rank zero. Thus, we can expand $e(-\mathsf{v}_{\blambda})$ in negative powers of $x$. Using the identities
\begin{align*}
&[x^{-1}]\frac{(x+w+w_1)(x+w+w_2)}{(x+w)(x+w+w_1+w_2)} = 0, \\
&[x^{-2}]\frac{(x+w+w_1)(x+w+w_2)}{(x+w)(x+w+w_1+w_2)} = w_1w_2, \\ 
&[x^{-1}]\frac{(x+w) \Big( \prod_{i<j} (x+w+w_i+w_j) \Big) (x+w+w_1+w_2+w_3+w_4)}{\Big(\prod_i (x+w+w_i) \Big) \prod_{i<j<k}(x+w+w_i+w_j+w_k)} = 0, \\
&[x^{-2}]\frac{(x+w) \Big( \prod_{i<j} (x+w+w_i+w_j) \Big) (x+w+w_1+w_2+w_3+w_4)}{\Big(\prod_i (x+w+w_i) \Big) \prod_{i<j<k}(x+w+w_i+w_j+w_k)} = 0,
\end{align*}
where $i,j,k \in \{1,2,3,4\}$, it is not hard to see that 
\begin{align*}
&[x^{-1}](e(-\mathsf{v}_{\blambda})) = 0, \\
&[x^{-2}](e(-\mathsf{v}_{\blambda})) = e(-P_3 T_{12}) e(-P_1 T_{34}) \Big( |\lambda_{34}| s_3s_4 + |\lambda_{12}| s_1s_2 \Big).
\end{align*}

Denoting $N:=M_{Q_2}(r_{12},n_{12})$ where $n_{12} := |\lambda_{12}|$, we note that
\[
e(-P_3 T_{12}) = \frac{e(t_1t_2t_3 T_N|_P)}{e(t_1t_2 T_N|_P)} = \frac{e(t_3 \Omega_N|_P)}{e(\Omega_N|_P)} =  \frac{e(t_3^{-1} T_N|_P)}{e(T_N|_P)}, 
\]
where $P \in N$ is the fixed point corresponding to $\lambda_{12}$ and we used \eqref{eqn:holsymp}. The result follows by summing over all $\blambda$. 
\end{proof}

In \cite[Sect.~10.3.1]{Nek2}, it is explained how to use the equation of Proposition \ref{prop:DS} to deduce the explicit formula 
\begin{align*}\label{eqn: Hilbadjoint} \mathsf{G}(q;s_1,s_2, s_3)
=\Bigg( \prod_{n=1}^{\infty} \frac{1}{1-q^n} \Bigg)^{\frac{(s_1+s_3)(s_2+s_3)}{s_1 s_2}}\end{align*}
previously proved, for example in \cite[Cor.~1]{CO}.

For arbitrary $r_{12}, r_{34}$, the non-perturbative Dyson--Schwinger equations are much more complicated to state explicitly; see \cite[Sect.~7.1.2, Thm.~6.1]{Nek2}. When $r>1$, there is no known concise closed form for the higher-rank analog $$\sum_{n=0}^{\infty} q^n \int_{M_{Q_2}(r,n)} e(T_{M_{Q_2}(r,n)} \otimes e^{-s_3})$$ of $\mathsf{G}(q;s_1,s_2, s_3)$.  It would be interesting to understand if relations of the form \eqref{eqn:DSinthm} uniquely characterize such series.

\begin{remark}
    In \cite{Nek3}, Nekrasov gives an equivalent definition of the origami moduli spaces as spaces of linear maps between Hermitian vector spaces obeying real-analytic equations quotiented by a $U(n)$-action. The proof of the compactness theorem in loc.~cit.~appears to rely on the real-analytic description of the origami moduli space. It would be interesting to have a purely algebraic proof of his compactness theorem. 
\end{remark}

\begin{remark} In principle,  non-perturbative Dyson--Schwinger equations of the form \eqref{eqn:DSinthm} should be extractable for arbitrary choices of $\vec{r}\in\Z_{\geq 0}^{6}.$
    However, in the notation of \cite{Nek3}, the compactness theorem is proven for the spaces $\mathfrak{M}_n^\infty(\vec{r})$. The space  $\mathfrak{M}_n^\infty(\vec{r})$ is strictly smaller than $\mathfrak{M}_n^0(\vec{r})=M_{Q_4}(\vec{r},n)$ unless the rank vector is of the form $r_{A}, r_{\overline{A}}\geq 0$ for some $A\in \bsix$ and $r_{B}=0$ otherwise. Thus, we only apply the compactness theorem of \cite[Sect.~8]{Nek3} to our setting for such rank vectors. 
    
    For example, if $r_{12}=r_{23}=1$ and $r_A=0$ otherwise, after simplification, the denominator of $\int_{[M_{Q_4}(\vec{r},1)]^{\vir}} 1$ becomes
    \begin{equation} \label{eqn:denom}
    s_1 s_2 s_3 (s_1 + s_2 - v_{12,1} + v_{23,1}) (s_2 + s_3 - v_{23,1} + v_{12,1} ).
\end{equation}
This can be geometrically interpreted as follows. Consider the following codimension one subtori
\begin{align*}
    \mathbb{T}'&:=\{(t_a,w_{A,\alpha})\in \mathbb{T}\,|\, t_1t_2 w_{12,1}^{-1} w_{23,1}=1\},\\
    \mathbb{T}''&:=\{(t_a,w_{A,\alpha})\in \mathbb{T} \,|\, t_2t_3 w_{12,1} w_{23,1}^{-1}=1\}
\end{align*}
corresponding to the last two factors of \eqref{eqn:denom}. Using Example \ref{ex:book n=1}, we can check that the $\mathbb{T}'$-fixed locus of $M_{Q_4}(\vec{r},1)$ is the $B_2=0$ part of the solutions of type (v), and the $\mathbb{T}''$-fixed locus is the $B_2=0$ part of the solutions of type (iv). In particular, $M_{Q_4}(\vec{r},1)^{\mathbb{T}'}$ and $M_{Q_4}(\vec{r},1)^{\mathbb{T}''}$ are isomorphic to $\mathbb{A}^1$, which is not compact. Observe that components (iv) and (v) are exactly the missing components from the smaller moduli space $\mathfrak{M}_n^\infty(\vec{r})$ for which the  compactness theorem is proved in \cite[Sect.~8]{Nek3}.
\end{remark}

\section{Framed sheaves description}

\subsection{Framed sheaves moduli}

In this section, we introduce moduli space of framed sheaves in arbitrary dimensions. We discuss two types of framed sheaves; one depends on a choice of an effective divisor and the other does not. We call the former  framed sheaves and the latter HL-framed sheaves after Huybrechts--Lehn \cite{HL}. 

\begin{definition}
    Let $(X,D)$ be a pair of a smooth projective variety $X$ and an effective divisor $i\colon D\hookrightarrow X$. A framed sheaf on $(X,D)$ is a pair $(E,\phi)$ of a pure coherent sheaf $E$ on $X$ and an isomorphism $\phi\colon i^*E\xrightarrow{\sim}F$ where $F$ is a fixed coherent sheaf on $D$, called a framing sheaf. 
\end{definition}

\begin{definition}
    Let $X$ be a smooth projective variety. An HL-framed sheaf on $X$ is a pair $(E,\phi)$ of a coherent sheaf $E$ on $X$ and a morphism $\phi\colon E\rightarrow F$ where $F$ is a fixed coherent sheaf on $X$, called an HL-framing sheaf. 
\end{definition}

\begin{remark}
    To each framed sheaf, we can canonically associate an HL-framed sheaf via adjunction as follows. Consider a pair $(X,D)$ and a framing sheaf $F$ on $D$. Given any framed sheaf $(E,\phi)$, an isomorphism $\phi\colon i^*E\xrightarrow{\sim} F$ induces a morphism $\phi\colon E\rightarrow i_*F$, for which we use the same notation. Then $(E,\phi\colon E\rightarrow i_*F)$ is an HL-framed sheaf with respect to an HL-framing sheaf $i_*F$. 
\end{remark}

While framed sheaves are geometrically very natural, little is known about their moduli spaces in general. Conversely, HL-framed sheaves are less geometric as  arbitrary morphisms are allowed instead of isomorphisms, but more is known about their moduli spaces \cite{HL}. Specifically, Huybrechts--Lehn construct projective moduli spaces of HL-framed sheaves with respect to certain stability conditions. The main result of this section is that framed sheaves are stable HL-framed sheaves for a certain choice of stability condition under a fairly general assumption. As a corollary, the moduli space of framed sheaves is shown to be quasi-projective under this assumption. Our use of stability conditions for HL-framed sheaves to characterize framed sheaves is motivated by work of Bruzzo--Markushevich \cite[Thm.~3.1]{BM}. 

\medskip

Let $(X,\O_X(1))$ be a polarized smooth projective variety. Recall that the Hilbert polynomial of a coherent sheaf $E$ on $X$ is the unique polynomial $P_E(x)$ such that $P_E(n)=\chi(X,E(n))$ for any integer $n$. If $E$ is at most $d$-dimensional, we have
$$P_E(x)=a_d(E)\frac{x^d}{d!}+a_{d-1}(E)\frac{x^{d-1}}{(d-1)!}+\cdots+a_0(E),\quad a_d(E)\geq0,
$$
with $a_d(E)=0$ if and only if $\dim(E)<d$. 

Fix an HL-framing sheaf $F$. We discuss stability conditions for HL-framed sheaves $(E,\phi\colon E\rightarrow F)$ of $\dim(E)=d$. The stability parameter is given by a polynomial $\delta(x)\in \Q[x]$ of degree strictly smaller than $d$ with a positive leading coefficient. Define the Hilbert polynomial of $(E,\phi)$ as 
$$P_{(E,\phi)}(x)=P_E(x)-\epsilon_{\phi}\cdot \delta(x)
$$
where $\epsilon_\phi=1$ if $\phi\neq 0$ and $\epsilon_\phi=0$ if $\phi= 0$. If $(E,\phi)$ is an HL-framed sheaf, then any subsheaf $E'\subseteq E$ is naturally HL-framed by the composition $\phi'\colon E'\subseteq E\xrightarrow{\phi}F$. For polynomials $f(x)$ and $g(x)$, we say that $f(x)\leq g(x)$ if $f(x)\leq g(x)$ holds for $x\gg 0$. 
\begin{definition}
    A $d$-dimensional HL-framed sheaf $(E,\phi)$ is $\delta$-(semi)stable if $\phi\neq 0$ and for all $0\subsetneq E'\subsetneq E$ we have
    $$a_d(E)\cdot P_{(E',\phi')}(x)\,(\leq)\,a_d(E')\cdot P_{(E,\phi)}(x).
$$
Note that it follows that for a $\delta$-semistable sheaf $(E,\phi)$, either $\ker(\phi) = 0$ or $\ker(\phi)$ is pure $d$-dimensional. 
\end{definition}

\begin{lemma}\label{lem: Hom vanishing}
    Let $(E,\phi)$ be a $d$-dimensional $\delta$-stable \textnormal{HL}-framed sheaf. Then $$\Hom(E,[E\xrightarrow{\phi}F])=0,$$ where $[E\xrightarrow{\phi}F]$ is a complex in degrees $0,1$. 
\end{lemma}
\begin{proof}
    Using the exact triangle $[E\xrightarrow{\phi}F]\rightarrow E\rightarrow F\xrightarrow{\,[1]\,}$, we have 
    $$\Hom(E,[E\xrightarrow{\phi}F])
    =\ker\Big(
    \Hom_X(E,E)\xrightarrow{\phi\circ -}\Hom_X(E,F)
    \Big).
    $$
    Suppose that there exists a nonzero $f \colon E\rightarrow E$ such that $\phi\circ f=0$. Note that $f$ is not an isomorphism, or else we would have $\phi=0$. Therefore $E'\coloneq \im(f)$ is a nontrivial proper subsheaf of $E$ whose induced HL-framing is $\phi'=0$. Since $E'$ is a subsheaf of a pure $d$-dimensional sheaf $\ker(\phi)$, $E'$ is also $d$-dimensional. As $(E,\phi)$ is $\delta$-stable, we get
    $$a_d(E)\cdot P_{E'}(x)=a_d(E)\cdot P_{(E',\phi')}(x)\,<\,a_d(E')\cdot P_{(E,\phi)}(x)
$$
with $a_d(E), a_d(E')>0$. On the other hand, consider $E''\coloneq \ker(f)$ together with the induced HL-framing $\phi''$. Again, the $\delta$-stability of $(E,\phi)$ implies that
$$a_d(E)\cdot (P_{E''}(x)-\delta(x))\leq a_d(E)\cdot P_{(E'',\phi'')}(x)\,<\,a_d(E'')\cdot P_{(E,\phi)}(x).
$$
As $0\rightarrow E''\rightarrow E\rightarrow E'\rightarrow 0$ is a short exact sequence, we have
$$P_{E}(x)=P_{E'}(x)+P_{E''}(x),\quad a_d(E)=a_d(E')+a_d(E''). 
$$
Therefore, the last inequality can be written as 
$$a_d(E)\cdot (P_{(E,\phi)}(x)-P_{(E',\phi')}(x))<(a_d(E)-a_d(E'))\cdot P_{(E,\phi)}(x),
$$
contradicting the first inequality.
\end{proof}

\begin{definition}
Let $S$ be a $\C$-scheme of finite type. An $S$-family of HL-framed sheaves is a pair $(\cE,\Phi)$ of an $S$-flat coherent sheaf $\cE$ on $X\times S$ together with a morphism $\Phi\colon \cE\rightarrow \pi_1^*F$, where $\pi_1 \colon X \times S \to X$ denotes projection. An isomorphism between two such families $(\cE,\Phi)$ and $(\cE',\Phi')$ is an isomorphism $f\colon \cE\xrightarrow{\sim}\cE'$ such that $\Phi=\Phi'\circ f$. We say that an $S$-family of $d$-dimensional HL-framed sheaves $(\cE,\Phi)$ is $\delta$-(semi)stable if the base change $(\cE_s, \Phi_s)$ along any closed points $s\rightarrow S$ is $\delta$-(semi)stable. 
\end{definition}

Let $v\in H^*(X,\Q)$ be a Chern character of a $d$-dimensional sheaf. The above definition gives rise to a moduli functor $\underline{M}^{\delta\textrm{-}ss}_{(X,F)}(v)$ (resp.~$\underline{M}^{\delta\textrm{-}st}_{(X,F)}(v)$) of isomorphism classes of $\delta$-semistable (resp.~$\delta$-stable) HL-framed sheaves $(E,\phi\colon E\rightarrow F)$ on $X$ such that $\ch(E)=v$ in the usual way. 

\begin{theorem}[Huybrechts--Lehn]\label{thm: HL moduli}
    The moduli functor $\underline{M}^{\delta\textrm{-}ss}_{(X,F)}(v)$ is corepresented by a projective scheme $M^{\delta\textrm{-}ss}_{(X,F)}(v)$. Furthermore, there exists an open subscheme $M^{\delta\textrm{-}st}_{(X,F)}(v)\subseteq M^{\delta\textrm{-}ss}_{(X,F)}(v)$ which represents the functor $\underline{M}^{\delta\textrm{-}st}_{(X,F)}(v)$.
\end{theorem}

\begin{remark}
    In \cite{HL}, the authors only consider the case when the dimension of the framed sheaf is maximal, i.e., $d=\dim(X)$. Nevertheless, all the results they use to compare $\delta$-(semi)stability with GIT-(semi)stability were already available for arbitrary dimensional cases, most notably the construction of the moduli space of pure sheaves \cite{S}. See also \cite{Sa} for related discussion on how to modify various concepts of HL-framed sheaves for  arbitrary dimension $d$. 
\end{remark}

We now apply Theorem \ref{thm: HL moduli} to $d$-dimensional framed sheaves on $(X,D)$ with respect to a framing sheaf $F\in \Coh(D)$. The following assumptions will be used at various places in this section:
\begin{enumerate}
    \item [(A1)] $\dim(F)=d-1$; 
    \item [(A2)] $D$ is ample; 
    \item [(A3)] $F$ is $\mu_D$-semistable and $d\geq 2$. 
\end{enumerate}
When $D$ is ample, we will consider the polarization $\O_X(1)=\O_X(D)$. Then Hilbert polynomials, $\delta$-(semi)stability and $\mu_D$-semistability (also known as slope semistability) in (A3) will be with respect to this choice.

The following lemma shows that topological transversality in terms of dimension of the intersection and cohomological transversality in terms of vanishing of higher cohomology groups are equivalent for pure $d$-dimensional sheaves. 

\begin{lemma}\label{lem: transverse}
Let $i\colon D\hookrightarrow X$ be an effective divisor and let $E$ be a pure $d$-dimensional sheaf on $X$. Then the following statements are equivalent:
\begin{enumerate}
    \item[$\mathrm{(1)}$] $\dim(\Supp(E)\cap D)=d-1$ or $\Supp(E)\cap D=\emptyset$.
    \item[$\mathrm{(2)}$] $E\otimes^L \O_D\cong E\otimes \O_D$. 
    \item[$\mathrm{(3)}$] $Li^*E\cong i^*E$.
\end{enumerate}
When these conditions are satisfied, we have a short exact sequence
\begin{equation}\label{eq: SES}
    0\rightarrow E(-D)\rightarrow E\rightarrow E\otimes \O_D\rightarrow 0.
\end{equation}
\end{lemma}
\begin{proof}
    The equivalence between (2) and (3) follows from the projection formula
    $$E\otimes^L i_*\O_D \cong i_*(Li^*E\otimes \O_D) \cong i_*(Li^*E). 
$$
As $\O_D$ has homological dimension 1, we have  $\mathcal{T}{\!\it{or}}_i(-,\O_D)=0$ for $i\geq 2$. Therefore (2) is equivalent to $\mathcal{T}{\!\it{or}}_1(E,\O_D)=0$.

Assume (1). Tensoring $E$ with the ideal exact sequence of $D$, we obtain
$$0\rightarrow \mathcal{T}{\!\it{or}}_1(E,\O_D)\rightarrow E(-D)\rightarrow E\rightarrow E\otimes \O_D\rightarrow 0. 
$$
As the support of $\mathcal{T}{\!\it{or}}_1(E,\O_D)$ is contained in $\Supp(E)\cap D$, it is at most $(d-1)$ dimensional by the assumption. On the other hand, $E$ is pure $d$-dimensional. Hence, so is $E(-D)$. This implies that $\mathcal{T}{\!\it{or}}_1(E,\O_D)=0$ so (2) holds. 

Conversely, assume (2). Then $\mathcal{T}{\!\it{or}}_1(E,\O_D)=0$ and we have a short exact sequence
$$0\rightarrow E(-D)\rightarrow E\rightarrow E\otimes \O_D\rightarrow 0. 
$$
So $\ch(E\otimes \O_D)=\ch(E)-\ch(E(-D))$ has no $d$-dimensional component. Therefore, $\dim(\Supp(E)\cap D)=\dim (E\otimes\O_D)\leq d-1$. As $D$ is a divisor, the fundamental theorem of dimension theory implies that either $\dim(\Supp(E)\cap D)=d-1$ or $\Supp(E)\cap D=\emptyset$.
\end{proof}
If a pure $d$-dimensional sheaf $E$ satisfies one of the equivalent statements in Lemma \ref{lem: transverse}, then we say that $E$ is transverse to $D$.

\begin{corollary}\label{cor: framed sheaves are transverse}
    If $(E,\phi)$ is a $d$-dimensional framed sheaf satisfying (A1), then $E$ is transverse to $D$. 
\end{corollary}
\begin{proof}
The framing defines an isomorphism $i^*E\simeq F$ which is of dimension $d-1$ by the assumption (A1). Therefore, it satisfies statement (1) of Lemma \ref{lem: transverse}. 
\end{proof}

The Hilbert polynomials of $E$ and $i_*F$ can be related using the short exact sequence \eqref{eq: SES}.
\begin{proposition}\label{prop: Hilbert polynomial relationship}
    Let $(E,\phi)$ be a $d$-dimensional framed sheaf satisfying (A1) and (A2). Then $P_E(x)$ is of the form
    $$P_E(x)=\frac{a_{d-1}(i_*F)}{d!}x^d+\frac{a_{d-2}(i_*F)+\frac{1}{2}a_{d-1}(i_*F)}{(d-1)!}x^{d-1}+\cdots+\chi(E)
    $$
    and is determined by $P_{i_*F}(x)$ and $\chi(E)$.
\end{proposition}
\begin{proof}
    We take the Chern character of the short exact sequence \eqref{eq: SES} to obtain
$$D\cdot\ch(E)=\frac{D}{1-e^{-D}}\cdot\ch(i_*F).$$
By the Hirzebruch--Riemann--Roch theorem, we have 
\begin{align*}
    P_E(x)
    &=\int_X e^{xD}\cdot \ch(E)\cdot\td(X)\\
    &=\chi(E)+\int_X \frac{e^{xD}-1}{D}\cdot\frac{D}{1-e^{-D}} \cdot\ch(i_*F)\cdot\td(X)\\
    &=\chi(E)+\sum_{k\geq 0}\frac{x^{k+1}}{(k+1)!}\int_X D^k\cdot \left(1+\frac{D}{2}+\cdots\right)\cdot\ch(i_*F)\cdot\td(X)\\
    &=\frac{a_{d-1}(i_*F)}{d!}x^d+\frac{a_{d-2}(i_*F)+\frac{1}{2}a_{d-1}(i_*F)}{(d-1)!}x^{d-1}+\cdots+\chi(E).
\end{align*}
The third equality shows that $P_E(x)$ is determined by $P_{i_*F}(x)$ and $\chi(E)$. 
\end{proof}
Taking $\O_X(1) = \O_X(D)$, we write the stability condition $\delta(x)$ of $d$-dimensional HL-framed sheaves as
$$\delta(x)=\sum_{n=0}^{d-1}\delta_n\frac{x^n}{n!}.
$$
We show that if $\delta_{d-1}$ is a small enough positive number with respect to the topological type of the framing sheaf $F$, then any framed sheaf $(E,\phi)$ is automatically $\delta$-stable as an HL-framed sheaf. 

\begin{proposition}\label{prop: framed sheaves are delta-stable}
    Assume (A1)--(A3) hold and $0<\delta_{d-1}<a_{d-1}(F)$. Then any $d$-dimensional framed sheaf $(E,\phi)$ is $\delta$-stable as an HL-framed sheaf. 
\end{proposition}
\begin{proof}
    
Let $(E,\phi)$ be a $d$-dimensional framed sheaf on $(X,D)$. Consider the associated HL-framed sheaf $\phi\colon E\rightarrow i_*F$. We show that it is $\delta$-stable. Clearly, $\phi\neq 0$ because it is the adjoint of an isomorphism. Let $0\subsetneq E'\subsetneq E$ be a nonzero proper subsheaf. We need to show that 
$$a_d(E)\cdot P_{(E',\phi')}(x)< a_d(E')\cdot P_{(E,\phi)}(x).
$$
We separately analyze the cases where $\phi'\colon E'\subset E\rightarrow i_*F$ is zero and nonzero. 

We begin with the case where $\phi'\colon E'\rightarrow i_*F$ is nonzero. We need to show that
\begin{equation}\label{eq: case 1 need to show}
    a_d(E)\cdot(P_{E'}(x)-\delta(x))
    <a_d(E')\cdot(P_E(x)-\delta(x)).
\end{equation}
If we replace $E'$ with its saturation inside $E$, then we still have $\phi'\neq 0$, $a_d(E')$ is unchanged and $P_{E'}(x)$ cannot decrease. So, assume that $E'$ is saturated inside $E$. Then $E'':=E/E'$ is zero or pure $d$-dimensional. We leave the (much easier) case $E''=0$ to the reader and assume that $E''$ is pure $d$-dimensional so that $0<a_d(E')<a_d(E)$. On the other hand, both sides of \eqref{eq: case 1 need to show} are polynomials of degree $d$ with the same leading coefficients. So the inequality would follow from the inequality between the $x^{d-1}$-coefficients 
\begin{equation}\label{eq: case 1 d-1 coeff}
    \frac{a_{d-1}(E')-\delta_{d-1}}{a_{d}(E')}
<\frac{a_{d-1}(E)-\delta_{d-1}}{a_{d}(E)}.
\end{equation}

We prove \eqref{eq: case 1 d-1 coeff} by replacing the information of Hilbert polynomials of $E$ and $E'$ by those of sheaves supported on $D$ using Proposition \ref{prop: Hilbert polynomial relationship}. Consider the exact sequence $0\rightarrow E'\rightarrow E\rightarrow E''\rightarrow 0$ of pure $d$-dimensional sheaves. As $E$ is transverse to $D$ by Corollary \ref{cor: framed sheaves are transverse}, so are $E'$ and $E''$ by the topological characterization of the transversality in Lemma \ref{lem: transverse}. Therefore, we obtain a short exact sequence
$$0\rightarrow E'\otimes \O_D\rightarrow E\otimes\O_D\rightarrow E''\otimes\O_D\rightarrow 0,
$$
and rewrite it as
$$0\rightarrow i_*F'\rightarrow i_*F\rightarrow i_*F''\rightarrow 0,
$$
for some $F', F''\in \Coh(D)$. Note that this yields a framed sheaf $E'$ with $i^*E'\simeq F'$ where $\dim(F')=d-1$. By applying Proposition \ref{prop: Hilbert polynomial relationship} to this framed sheaf and $(E,\phi)$, the inequality \eqref{eq: case 1 d-1 coeff} becomes
$$\frac{a_{d-2}(F')+\frac{1}{2}a_{d-1}(F')-\delta_{d-1}}{a_{d-1}(F')}
<\frac{a_{d-2}(F)+\frac{1}{2}a_{d-1}(F)-\delta_{d-1}}{a_{d-1}(F)}.
$$
Here we crucially used $d\geq 2$ from the assumption (A3). This inequality follows from
$$0<\delta_{d-1},
\quad 0<a_{d-1}(F')<a_{d-1}(F),\quad\textnormal{and}\quad
\frac{a_{d-2}(F')}{a_{d-1}(F')}
\leq\frac{a_{d-2}(F)}{a_{d-1}(F)}.$$
The last inequality uses the $\mu_D$-semistability of $F$ from the assumption (A3). 

We now consider the case where $\phi'\colon E'\rightarrow i_*F$ is zero. We need to show that
\begin{equation}\label{eq: case 2 need to show}
    a_d(E)\cdot P_{E'}(x) <a_d(E')\cdot(P_E(x)-\delta(x)).
\end{equation}
Note that $\phi'=0$ implies that $E'$ is a subsheaf of $E(-D)=\ker(E\rightarrow i_*F)$. We claim that there exists some $n\geq 1$ such that $E' \subseteq E(-nD)$ and $E' \not\subseteq E(-(n+1)D)$. Suppose otherwise. Then we have 
$$0\subsetneq E'\subseteq \bigcap_{n\geq 1}E(-nD)=:E_\infty.
$$
By Krull's theorem, $\Supp(E_{\infty}) \cap D = \varnothing$. However, since $E_\infty$ is 2-dimensional, ampleness of $D$ implies that $\Supp(E_\infty)\cap D\neq \emptyset$, yielding a contradiction. Therefore, there exists an $n\geq 1$ such that $E'\subseteq E(-nD)$ but $E' \not\subseteq E(-(n+1)D)$. 

Set $\widetilde{E}:=E'(nD)\subseteq E$. Then $\widetilde{\phi}\colon \widetilde{E}\subseteq E\xrightarrow{\phi}i_*F$ is nonzero by the choice of $n$. This allows us to use arguments from the argument for the case where $\phi'\neq 0$. In particular, we may assume that $\widetilde{E}$ is saturated inside $E$ and so $i^*\widetilde{E}\simeq \widetilde{F}$ for some $\widetilde{F}\subseteq F$. Again we leave the (easier) case $\widetilde{E} = E$ to the reader and assume $\widetilde{E} \subsetneq E$. Applying Proposition \ref{prop: Hilbert polynomial relationship} to $(\widetilde{E},i^*\widetilde{E}\simeq \widetilde{F})$, we have 
\begin{align*}
    P_{E'}(x)
    &=P_{E'(nD)}(x-n)=P_{\widetilde{E}}(x-n)\\
    &=a_{d-1}(\widetilde{F})\frac{(x-n)^d}{d!}+\left(a_{d-2}(\widetilde{F})+\frac{a_{d-1}(\widetilde{F})}{2}\right)\frac{(x-n)^{d-1}}{(d-1)!}+\cdots\\
    &=a_{d-1}(\widetilde{F})\frac{x^d}{d!}+\left(a_{d-2}(\widetilde{F})+\frac{a_{d-1}(\widetilde{F})}{2}-n\cdot a_{d-1}(\widetilde{F})\right)\frac{x^{d-1}}{(d-1)!}+\cdots.
\end{align*}
Therefore, the inequality \eqref{eq: case 2 need to show} follows from the corresponding inequality of the $x^{d-1}$-coefficients, which becomes 
$$\frac{a_{d-2}(\widetilde{F})+\frac{1}{2}{a_{d-1}(\widetilde{F})}-n\cdot a_{d-1}(\widetilde{F})}{a_{d-1}(\widetilde{F})}
<\frac{a_{d-2}(F)+\frac{1}{2}{a_{d-1}(F)}-\delta_{d-1}}{a_{d-1}(F)}
$$
or, equivalently,
$$\frac{a_{d-2}(\widetilde{F})}{a_{d-1}(\widetilde{F})}-n
<\frac{a_{d-2}(F)-\delta_{d-1}}{a_{d-1}(F)}.
$$
This follows from the inequalities
$$n\geq 1,\quad \delta_{d-1}<a_{d-1}(F),\quad\textnormal{and}\quad
\frac{\alpha_{d-2}(\widetilde{F})}{\alpha_{d-1}(\widetilde{F})}
\leq\frac{\alpha_{d-2}(F)}{\alpha_{d-1}(F)}.$$
The last inequality uses the $\mu_D$-semistability of $F$ from the assumption (A3). 
\end{proof}
\begin{corollary}\label{cor: Hom vanishing for framed sheaves}
If $(E,\phi)$ is a $d$-dimensional framed sheaf satisfying the assumptions (A1)--(A3), then $\Hom_X(E,E(-D))=0$.    
\end{corollary}
\begin{proof}
    Choose a stability condition $\delta$ as in Proposition \ref{prop: framed sheaves are delta-stable}. Then the corresponding HL-framed sheaf $\phi\colon E\rightarrow i_*F$ is $\delta$-stable. As $[E\xrightarrow{\phi}i_*F]\simeq E(-D)$, the statement follows from Lemma \ref{lem: Hom vanishing}. 
\end{proof}

We now apply the above proposition to moduli theory of framed sheaves. We first define the moduli functor for framed sheaves. 

\begin{definition} \label{def:functfr}
    Let $S$ be a $\C$-scheme of finite type. An $S$-family of $d$-dimensional framed sheaves is a pair $(\cE,\Phi)$ of an $S$-flat coherent sheaf $\cE$ on $X \times S$, for which the fibre $\cE_s$ is pure $d$-dimensional for all closed points $s \in S$,  together with an isomorphism $\Phi \colon i^*\cE\xrightarrow{\sim}\pi_1^*F$, where $i \colon D \times S \hookrightarrow X \times S$ denotes the inclusion and $\pi_1 \colon D \times S \to D$ denotes projection. An isomorphism between two such families $(\cE,\Phi)$ and $(\cE',\Phi')$ is an isomorphism $f\colon \cE\xrightarrow{\sim}\cE'$ such that $\Phi=\Phi'\circ i^*(f)$.
\end{definition}

Under the assumption (A1), framed sheaves are transverse to the divisor $D$. So, their restriction to the divisor behaves nicely in families. 

\begin{proposition}\label{prop: no derived pull back}
    Assume (A1) holds. If $(\cE,\Phi)$ is an $S$-flat family of $d$-dimensional framed sheaves, then $Li^*\cE \cong i^*\cE$ is flat over $S$.
\end{proposition}
\begin{proof}

By Corollary \ref{cor: framed sheaves are transverse}, $\cE_s$ is transverse to the divisor $D$ for every closed point $s\in S$. As the statement is local on $X$, we may assume that $X$ is affine. Let $f \colon X\rightarrow\mathbb{A}^1$ be a regular function for which the Cartier divisor $D=f^{-1}(0)$. In the proof of \cite[Cor.~3.4]{LW}, it is shown that there exists an open subset $U\subseteq \mathbb{A}^1\times S$ containing $\{0\}\times S$ such that $\cE|_U$ is flat over $U$. It follows that $Li^*\cE \cong i^*\cE$ is flat over $S$. 
\end{proof}

Let $v\in H^*(X,\Q)$ be a Chern character of a $d$-dimensional sheaf. This defines a moduli functor $\underline{M}_{(X,D,F)}(v)$ of isomorphism classes of $d$-dimensional framed sheaves $(E,\phi\colon i^*E\simeq F)$ on $(X,D)$ such that $\ch(E)=v$. By adjunction and Proposition \ref{prop: framed sheaves are delta-stable}, we have a natural map
$$\underline{M}_{(X,D,F)}(v)\rightarrow \underline{M}_{(X,i_*F)}^{\delta\textrm{-}st}(v)
$$
under the assumptions (A1)--(A3). We prove that this map is an open embedding. 
\begin{lemma}\label{lem: framed sheaves form an open subset}
Assume that (A1)--(A3) hold and $0<\delta_{d-1}<a_{d-1}(F)$. Then the morphism $\underline{M}_{(X,D,F)}(v)\rightarrow \underline{M}_{(X,i_*F)}^{\delta\textrm{-}st}(v)$ is an open embedding. 
\end{lemma}
\begin{proof}
Suppose we are given an $S$-flat family of $d$-dimensional $\delta$-stable HL-framed sheaves $(\cE,\Phi\colon \cE\rightarrow \pi_1^*(i_*F))$ on $X\times S$. We prove that the locus of $s\in S$ coming from framed sheaves forms an open subset. Recall that the locus where $\cE_s$ is pure defines an open subset $U_1\subseteq S$. Also, framed sheaves are required to be transverse to $D$. By \cite[Cor.~3.4]{LW}, the locus where this occurs forms an open subset $U_2\subseteq U_1$ and $i^*(\cE|_{U_2})$ is flat over $U_2$. Now, consider the adjunction morphism $\Phi|_{U_2}\colon i^*(\cE|_{U_2})\rightarrow \pi_1^* F$ between $U_2$-flat sheaves. The locus where the restriction $\Phi_s \colon i^*\cE_s\rightarrow F$ is surjective defines an open subset $U_3\subseteq U_2$. On this locus, the morphism $\Phi|_{U_3} \colon i^*(\cE|_{U_3})\rightarrow \pi_1^* F$ is surjective by Nakayama's lemma. The kernel is again a $U_3$-flat sheaf, so the locus where $\Phi_s \colon i^*\cE_s\rightarrow F$ is an isomorphism forms a union of connected components $U_4\subseteq U_3$. Here, we use that $Li^* \cE_s \cong i^*\cE_s$ and, as we fix $v$, either $\ch(i^* \cE_s) = \ch(Li^* \cE_s)$ and $\ch(F)$ coincide or not. In other words, $U_4\subseteq S$ is precisely the locus of framed sheaves, which is shown to be open.
\end{proof}

\begin{corollary}\label{cor: quasi-projectivity of moduli of framed sheaves}
    Assume that (A1)--(A3) hold. Then the moduli functor $\underline{M}_{(X,D,F)}(v)$ is represented by a quasi-projective scheme $M_{(X,D,F)}(v)$. 
\end{corollary}
\begin{proof}
    This result follows from Lemma \ref{lem: framed sheaves form an open subset} and Theorem \ref{thm: HL moduli}. 
\end{proof}

\subsection{Derived structure} \label{sec:derstrucsheaves}

In this section, we discuss a derived enhancement of the moduli space of framed sheaves. Let $(X,D)$ be a smooth projective variety and an effective divisor. Denote the derived stack of perfect complexes on $X$ (resp.~$D$) by $$\overline{\mathbf{M}}_X:=\Map(X,\mathbf{Perf})\quad \big(\textrm{resp.  } \overline{\mathbf{M}}_D:=\Map(D,\mathbf{Perf})\big),$$
which were constructed in \cite{TV}. Pullback along the embedding $i\colon D\rightarrow X$ defines a morphism $Li^*\colon\overline{\mathbf{M}}_X\rightarrow \overline{\mathbf{M}}_D$. A perfect framing sheaf $F\in \Coh(D)$ gives a point in $\overline{\mathbf{M}}_D$. Define $\overline{\mathbf{M}}_{(X,D,F)}$ via the homotopy fibre diagram
\begin{center}
\begin{tikzcd}
 \overline{\mathbf{M}}_{(X,D,F)}\arrow[r] \arrow[d] & \overline{\mathbf{M}}_X \arrow[d,"Li^*"] \\
\{F\} \arrow[r]           & \overline{\mathbf{M}}_D. 
\end{tikzcd}
\end{center}

Let $f=f_E\colon X\rightarrow \mathbf{Perf}$ be a point in $\overline{\mathbf{M}}_X$ given by a perfect complex $E$ on $X$. Consider the exact triangle of tangent complexes 
$$\mathbb{T}_{Li^*, f}\rightarrow \mathbb{T}_{\overline{\mathbf{M}}_X,f}\rightarrow \mathbb{T}_{\overline{\mathbf{M}}_D,f\circ i}\xrightarrow{[1]}
$$
with respect to the morphism $Li^*\colon \overline{\mathbf{M}}_X\rightarrow \overline{\mathbf{M}}_D$. This can be rewritten as
$$R\Gamma(X,Lf^*\mathbb{T}_{\mathbf{Perf}} (-D))\rightarrow R\Gamma(X,Lf^*\mathbb{T}_{\mathbf{Perf}})\rightarrow R\Gamma(X,i_*Li^*Lf^*\mathbb{T}_{\mathbf{Perf}})\xrightarrow{[1]}
$$
obtained by tensoring $Lf^*\mathbb{T}_{\mathbf{Perf}}$ with the ideal exact sequence of $D$, see \cite[Sect.~3]{Spa}. As $\mathbb{T}_{\mathbf{Perf}}=R\mathcal{H}om(\mathcal{U},\mathcal{U})[1]$ for the universal perfect complex $\mathcal{U}$, we can further rewrite the exact triangle in terms of the perfect complex $E$
$$R\Hom_X(E,E(-D))[1]\rightarrow R\Hom_X(E,E)[1]\rightarrow R\Hom_D(Li^*E,Li^*E)[1]\xrightarrow{[1]}.
$$
When $2D\in |-K_X|$, the relative tangent complex $R\Hom_X(E,E(-D))[1]$ satisfies
$$
R\Hom_X(E,E(-D))[1] \cong (R\Hom_X(E,E(-D))[1])^\vee[2-\dim(X)]
$$
by Serre duality. This duality suggests that $\overline{\mathbf{M}}_{(X,D,F)}$ may admit a shifted symplectic structure. The existence of such a structure has been proved in \cite{Spa}.\footnote{In the notation of \cite[Thm.~0.3]{Spa}, we take $Y=\mathbf{Perf}$, $E=0$ with the base map $f\colon D\rightarrow Y$ corresponding to the perfect coherent sheaf $F$ on $D$. The assumption of loc.~cit.~on the mapping stacks being derived Artin stacks locally of finite presentation is satisfied by \cite[Cor.~3.3]{T}.}

\begin{theorem}[Spaide]\label{thm: Spaide}
    If $2D\in |-K_X|$, then $\overline{\mathbf{M}}_{(X,D,F)}$ is equipped with a $(2-\dim(X))$-shifted symplectic structure.
\end{theorem}

Let $\mathbf{M}_{(X,D,F)}\subseteq \overline{\mathbf{M}}_{(X,D,F)}$ be the pullback of $\overline{\mathbf{M}}_{(X,D,F)}$ under the open derived substack $\mathbf{M}_X\subseteq \overline{\mathbf{M}}_X$ corresponding to the locus of pure coherent sheaves on $X$. Similarly, denote by $\mathbf{M}_D \subseteq \overline{\mathbf{M}}_D$ the open derived substack of coherent sheaves on $D$. By definition, $\mathbf{M}_{(X,D,F)}$ is the derived moduli stack parametrizing pure coherent sheaves $E$ on $X$ together with an isomorphism $\phi\colon Li^*E\xrightarrow{\sim} F$ on $D$. 

Assume (A1) holds and let $\gamma \in H^{2\dim(X)-4}(X,\Q)$ be Poincar\'e dual to the class of a 2-dimensional closed subscheme of $X$. We denote by $\mathbf{M}_{(X,D,F)}(v) \subseteq \mathbf{M}_{(X,D,F)}$ the open and closed substack parametrizing objects of fixed topological type $v$, i.e., $\ch(E)=v=(0,\ldots,0,\gamma,*,*) \in H^{\mathrm{even}}(X,\Q)$. Then $\mathbf{M}_{(X,D,F)}(v)$ gives a derived enhancement of the moduli stack $\underline{M}_{(X,D,F)}(v)$, that is, \footnote{
    By slight abuse of notation, $\underline{M}_{(X,D,F)}(v)$ denotes the stack whose corresponding classical moduli functor was defined in Definition \ref{def:functfr}.}
$$t_0(\mathbf{M}_{(X,D,F)}(v))=\underline{M}_{(X,D,F)}(v).$$
Here Proposition \ref{prop: no derived pull back} (and in particular (A1)) is used. This construction in particular defines an obstruction theory (not necessarily $2$-term) on the moduli stack $\underline{M}_{(X,D,F)}(v)$ with a tangent complex equal to 
\begin{equation}\label{eqn: tangent complex of framed moduli}
    R\Hom_X(E,E(-D))[1]
\end{equation}
at a point $(E,\phi\colon i^*E\xrightarrow{\sim}F)$. We note that in the case where $\dim(X) = 4$, the rank of this tangent complex is $-\gamma^2$.

\begin{corollary}\label{cor: derived structure}
    Assume that (A1) holds, $F\in \Coh(D)$ is perfect, and $2D\in |-K_X|$. Then the derived enhancement $\mathbf{M}_{(X,D,F)}(v)$ of $\underline{M}_{(X,D,F)}(v)$ admits a $(2-\dim(X))$-shifted symplectic structure. 
\end{corollary}
\begin{proof}
    Restrict the $(2-\dim(X))$-shifted symplectic structure of $\overline{\mathbf{M}}_{(X,D,F)}$ to the open derived substack $\mathbf{M}_{(X,D,F)}(v)$.
\end{proof}

For convenience, we record the 4-fold case as a self-contained statement. 

\begin{theorem}\label{cor: CY4 obstruction theory}
    Let $(X,D)$ be a smooth projective 4-fold together with an ample effective divisor $D$ satisfying $2D\in |-K_X|$. Let $F\in \Coh(D)$ be a $\mu_D$-semistable 1-dimensional sheaf which is perfect. Then the moduli space $M_{(X,D,F)}(v)$ of 2-dimensional framed sheaves is a quasi-projective scheme and admits a $(-2)$-shifted symplectic derived enhancement. In particular, $M_{(X,D,F)}(v)$ admits a 3-term symmetric obstruction theory, with self-dual 3-term locally free resolution,  which is isotropic and has virtual dimension $-\tfrac{1}{2} \gamma^2$.
\end{theorem}
\begin{proof}
Combine Corollaries \ref{cor: quasi-projectivity of moduli of framed sheaves} and \ref{cor: derived structure}. In particular, the existence of the self-dual 3-term locally free resolution follows from quasi-projectivity of $M_{(X,D,F)}(v)$ (Corollary \ref{cor: quasi-projectivity of moduli of framed sheaves}) and \cite[Prop.~4.1]{OT}. The existence of the $(-2)$-shifted symplectic derived enhancement implies the isotropic property for the intrinsic normal cone by \cite[Prop.~4.3, Thm.~4.6]{OT}. 
\end{proof}

\begin{remark}
There is a ``classical'' way, not involving derived geometry, to construct the obstruction theory on $M_{(X,D,F)}(v)$. Consider the restriction morphism $r = Li^* \colon M_{X} \to M_D$ on the classical truncations $M_X = t_0(\overline{\mathbf{M}}_X)$, $M_D = t_0(\overline{\mathbf{M}}_D)$. Let $\cU$, $\cU'$ be the universal complexes on $M_X \times X$, $M_D \times D$ respectively, and denote the projections to the first components by $\pi$. Letting $i$ also denote the base change of $i$ to $M_X \times X$ , we have  $L i^* \cU \cong  L(r \times \mathrm{id}_D)^* \cU'$. Consider the natural exact triangle 
\[
R\hom_{\pi}(\cU,\cU(-D)) \longrightarrow R\hom_{\pi}(\cU,\cU) \stackrel{\alpha}{\longrightarrow} R\hom_{\pi}(\cU,i_* Li^* \cU) \stackrel{[1]}{\longrightarrow}.
\]
By adjunction, we have $R\hom_{\pi}(\cU,i_* Li^* \cU) \cong Lr^* R\hom_{\pi}(\cU',\cU')$. Defining $\EE := (R\hom_{\pi}(\cU,\cU)[1])^{\vee}$ and $\FF := (R\hom_{\pi}(\cU',\cU')[1])^{\vee}$, we obtain a diagram
\begin{displaymath}
\xymatrix
{
Lr^* \FF \ar^{(\alpha[1])^\vee}[r] \ar[d] & \EE \ar[d] & & \\
Lr^* \LL_{M_D} \ar[r] & \LL_{M_X} \ar[r] & \LL_{M_X / M_D} \ar^>>>>>{[1]}[r] &
}
\end{displaymath}
where the vertical arrows are induced by the Atiyah classes of $\cU$, $\cU'$ and the square commutes by functoriality of the Atiyah class. Completing the diagram, we obtain a morphism 
\[
\mathbb{G} \to \LL_{M_X/M_D}, \quad \mathbb{G} \cong (R\hom_{\pi}(\cU,\cU(-D))[1])^\vee,
\]
which is a relative obstruction theory by \cite[(3.13)]{Man}. It induces an obstruction theory on $M_{(X,D,F)}(v)$ for each perfect framing sheaf $F \in \Coh(D)$. See \cite[Prop.~3.3.2]{KLT} for a similar construction. This approach does not suffice to deduce the isotropic property in Theorem \ref{cor: CY4 obstruction theory}.
\end{remark}

\subsection{Origami partition function via framed sheaves} \label{sec:origamiviasheaves}

In this section, we consider a specific moduli space of framed sheaves which is closely related to the moduli space of stable representations of the 4D ADHM quiver discussed in Section \ref{sec:quiver}. 

Let $(X,D)$ be the following specific pair satisfying $2D\in |-K_X|$
$$X=\PP^1\times\PP^1\times\PP^1\times\PP^1,\quad D=\{(z_1,z_2,z_3,z_4)\,|\,z_1z_2z_3z_4=\infty\}. 
$$
For each $A\in \bsix$, define a coordinate surface 
$$S_A:=\{(z_1,z_2,z_3,z_4)\in X\,|\,z_{\bar{a}}=z_{\bar{b}}=0\},\quad \overline{A}=\{\bar{a}<\bar{b}\}
$$
and denote its intersection with the divisor by $\ell_A:=S_A\cap D$. Then $S_A\simeq \PP^1\times \PP^1$ and $\ell_A$ is a chain of two $\PP^1$'s meeting in a point.

Recall that the moduli space $M_{Q_4}(\vec{r},n)$ depends on the rank vector $\vec{r}=(r_A)_{A\in \bsix}$ and the instanton number $n\in\Z_{\geq 0}$. To compare $M_{Q_4}(\vec{r},n)$ with a moduli space of framed sheaves, we need to choose the appropriate framing sheaf and topological type. We define a framing sheaf with respect to $\vec{r}$ by
$$F_{\vec{r}}:=\bigoplus_{A}\O_{\ell_A}^{\oplus r_A}\in \Coh(D).
$$
As $\ell_A=S_A\cap D$, the simplest example of a framed sheaf with respect to $F_{\vec{r}}$ is 
$$E=\bigoplus_{A} \O_{S_A}^{\oplus r_A},\quad \phi\colon i^*E\xrightarrow{\sim} F_{\vec{r}},
$$
where $\phi$ corresponds to the identity map. We consider this as a framed sheaf of ``instanton number" zero. For general ``instanton number'' $n\geq 0$, we choose a Chern character $v_{\vec{r},n}\in H^*(X,\Q)$ such that
$$v_{\vec{r},n}=\ch\Big(\bigoplus_{A}\O_{S_A}^{\oplus r_A}\Big)-n[\pt] = \big(0,0,\sum_{A} r_A[S_A],0,-n \big),
$$
where the second equality follows from Grothendieck--Riemann--Roch and the fact that $N_{S_A/(\PP^1)^4} \cong \O_{S_A} \oplus \O_{S_A}$.

\begin{remark}
    If $(E,\phi\colon i^*E\xrightarrow{\sim}F)$ is a framed sheaf, then Proposition \ref{prop: Hilbert polynomial relationship} implies that the Hilbert polynomial $P_E(x)$ is determined by the Hilbert polynomial $P_{i_* F}(x)$ up to an integer fixed by $\chi(E)$. Here, we are essentially doing the same for the Chern character $\ch(E)$ by fixing $i_*\ch(F)$ and the instanton number $n\in \Z_{\geq 0}$. 
\end{remark}
\begin{definition}
    The moduli space of 2-dimensional framed sheaves on $(X,D)$ with respect to the rank vector $\vec{r}$ and the instanton number $n$ is defined as 
    $$M_{(\PP^1)^4}(\vec{r},n):=M_{(X,D,F_{\vec{r}})}(v_{\vec{r},n}). 
$$
\end{definition}

Next, we define an action of $(\C^*)^4 \times (\C^*)^r$ and $\TT := T \times (\C^*)^r$ on $M_{(\PP^1)^4}(\vec{r},n)$, where $T = Z(t_1t_2t_3t_4-1)$ and $r := \sum_{A} r_{A}$ as before. The action on the points of $(\PP^1)^4 \setminus D$ is given by
\[
(t_1,t_2,t_3,t_4) \cdot (p_1,p_2,p_3,p_4) = (t_1p_1,t_2p_2,t_3p_3,t_4p_4),
\]
which extends to an action on $(\PP^1)^4$ leaving $D$ invariant. Moreover, $(\C^*)^4$ leaves the surfaces $S_A$ and the curves $\ell_A$ invariant. For each $t = (t_1,t_2,t_3,t_4) \in (\C^*)^4$, we denote the induced automorphism by $t\colon (\PP^1)^4 \to (\PP^1)^4$. Let $w = (w_{A,\alpha}) \in (\C^*)^r$. Then $(t,w)$ sends a framed sheaf $[(E,\phi)] \in M_{(\PP^1)^4}(\vec{r},n)$ to
\[
(t^* E, \quad t^* E|_D \stackrel{t^* \phi}{\longrightarrow} t^* F_{\vec{r}} \cong F_{\vec{r}} \stackrel{w}{\longrightarrow} F_{\vec{r}}),
\]
where $t^* F_{\vec{r}} \cong F_{\vec{r}}$ is the direct sum of the equivariant structures $t^* \O_{\ell_A} \cong \O_{\ell_A}$ and $w$ acts on the direct summand $\O_{\ell_A}^{\oplus r_A}$of $F_{\vec{r}}$ by the diagonal matrix $\mathrm{diag}(w_{A,1}, \ldots, w_{A,r_A})$.

\begin{theorem}\label{thm: moduli of framed sheaves}
    The moduli space $M_{(\PP^1)^4}(\vec{r},n)$ is a quasi-projective scheme and admits a $(-2)$-shifted symplectic derived enhancement. In particular, $M_{(\PP^1)^4}(\vec{r},n)$ admits a 3-term symmetric obstruction theory 
    $$
    \phi_{\fr} \colon E\udot_{\fr} \to \tau^{\geq -1} \LL_{\cM}, \quad (E\udot_{\fr})^\vee[2] \stackrel{\theta_{\fr}}{\cong} E\udot_{\fr}, \quad \vd := \frac{1}{2} \rk(E_{\fr}\udot) =  - \sum_{A \in \bthree} r_A r_{\overline{A}}
    $$ 
    that is isotropic. Furthermore, $M_{(\PP^1)^4}(\vec{r},n)$ carries a $\TT$-action with respect to which the 3-term symmetric obstruction theory is $\TT$-equivariant. 
\end{theorem}
\begin{proof}
We check that $M_{(\PP^1)^4}(\vec{r},n)$ satisfies the assumptions of Theorem \ref{cor: CY4 obstruction theory}. The divisor $D$ is clearly ample and satisfies $2D\in |-K_X|$. As $\ell_A\hookrightarrow D$ is lci, $\O_{\ell_A}$ is perfect as a coherent sheaf on $D$, hence so is $F_{\vec{r}}$. On the other hand, $\O_{\ell_A}$ is $\mu_D$-stable and the sheaves $\O_{\ell_A}$, for each $A\in \bsix$, have the same $\mu_D$-slope. Therefore, $F_{\vec{r}}$ is $\mu_D$-semistable. We also  calculate
\[
\Big( \sum_{A} r_A [S_A] \Big)^2 = 2 \sum_{A \in \bthree} r_A r_{\overline{A}}.
\]

Finally, we observe that the obstruction theory is $\TT$-equivariant (in fact, $(\C^*)^4 \times (\C^*)^r$-equivariant). Furthermore, there is a $(\C^*)^4$-equivariant isomorphism
\[
K_{(\PP^1)^4} \cong \O(-2D) \otimes (t_1t_2t_3t_4).
\]
By Grothendieck--Verdier duality, we therefore have an isomorphism
\[
\theta_{\fr} \colon (E\udot_{\fr})^\vee[2] \cong E\udot_{\fr} \otimes (t_1t_2t_3t_4)^{-1}, \quad \theta_{\fr}^\vee[2] = \theta_{\fr} \otimes (t_1t_2t_3t_4).
\]
The obstruction theory is therefore equivariantly symmetric upon restriction to $\TT.$
\end{proof}

Let $\ell_{\infty}$ denote $\{(z_1,z_2) \, | \, z_1z_2 = \infty\} \subset \mathbb{P}^1 \times \mathbb{P}^1$ and consider the moduli space $M_{\mathbb{P}^1 \times \mathbb{P}^1}(r,n)$ parametrizing pairs $(E,\phi)$ of a rank $r$ torsion free sheaf $E$ on $\mathbb{P}^1 \times \mathbb{P}^1$ satisfying $c_2(E) = n$ and an isomorphism $\phi \colon E|_{\ell_{\infty}} \to \O_{\ell_{\infty}}^{\oplus r}$ .\footnote{This provides an equivalent description of the traditional framed sheaves moduli space, i.e., there is an equivariant isomorphism $M_{\mathbb{P}^1 \times \mathbb{P}^1}(r,n) \cong M_{\PP^2}(r,n)$.}

\begin{proposition} \label{prop:supportprop}
Suppose $r_{12} = r$ and $r_A = 0$ otherwise. Then there exists a $\TT$-equivariant isomorphism
$$
M_{(\PP^1)^4}(\vec{r},n) \cong M_{\PP^1 \times \PP^1}(r,n).
$$
\end{proposition}
\begin{proof}
Identifying $S_{12} = \PP^1 \times \PP^1$, there is an obvious morphism
\[
f\colon M_{\PP^1 \times \PP^1}(r,n) \to M_{(\PP^1)^4}(\vec{r},n),
\]
where the Chern classes match due to the isomorphism $N_{S_{12} / (\PP^1)^4} \cong \O \oplus \O$ and Grothendieck--Riemann--Roch. We first prove that $f$ is a bijection. Let $[(E,\phi)] \in M_{(\PP^1)^4}(\vec{r},n)$. It suffices to show that the scheme theoretic support $S:=\mathrm{Supp}(E)$ equals $S_{12}$, so that $[(E,\phi)]$ is in the image of $f$. We first show that $S$ is set theoretically equal to $S_{12}$ by showing that $S\cap Z(z_4-t)=S\cap Z(z_3-t)=\emptyset$ for any $t\in \PP^1\backslash\{0\}$. Suppose otherwise. Without loss of generality, we may assume that $S\cap Z(z_4-t)\neq \emptyset$. Then the dimension of $S\cap Z(z_4-t)$ is at least 1 so it intersects the ample divisor $D=\{z_1z_2z_3z_4=\infty\}$, i.e., 
$$\emptyset\neq S\cap Z(z_4-t)\cap D=\ell_{12}\cap Z(z_4-t),
$$
which is clearly a contradiction. Next, we note that $E|_{D}\simeq \O_{\ell_{12}}^{\oplus r}$, which implies that $S\cap D=\ell_{12}$ scheme theoretically. So, $S$ is generically reduced. Then purity of $E$  implies that $S=S_{12}$ as claimed.

Now we prove that $f$ is an isomorphism. Recall that $M_{\PP^1\times \PP^1}(r,n)$ is smooth of dimension $2rn$ with tangent space at $[(E,\phi)]\in M_{\PP^1\times \PP^1}(r,n)$ given by $\Ext^1_{S_{12}}(E,E(-\ell_{12}))$. On the other hand, using \eqref{eqn: tangent complex of framed moduli}, the tangent space of $M_{(\PP^1)^4}(\vec{r},n)$ at $f([E,\phi)])$ is given by 
$$\Ext^1_{(\PP^1)^4}(i_*E, i_*E(-D))
$$
where $i \colon S_{12}\hookrightarrow (\PP^1)^4$ is the natural inclusion. As $N_{S_{12} / (\PP^1)^4} \cong \O \oplus \O$, $L_q i^* i_* E \cong E \otimes \Lambda^q N^*_{S_{12}/(\PP^1)^4}$, and $\Hom_{S_{12}}(E,E(-\ell_{\infty})) \cong 0 \cong \Ext^2_{S_{12}}(E,E(-\ell_{\infty}))$ (Corollary \ref{cor: Hom vanishing for framed sheaves}), we  deduce that
\begin{align*}
    \Ext^1_{(\PP^1)^4}(i_*E, i_*E(-D))
    &\cong \Ext^1_{S_{12}}(E, E(-\ell_{12})).
\end{align*}
This argument also proves that the Zariski tangent spaces at points of of $M_{(\PP^1)^4}(\vec{r},n)$ are all of dimension $2rn$. As $f$ is a bijection, this implies that $M_{(\PP^1)^4}(\vec{r},n)$ is reduced and hence smooth. Thus $f$ is a bijective morphism between smooth varieties over $\C$, therefore it is an isomorphism. 
\end{proof}

\begin{proposition} \label{prop:fixlocframed}
There exists a $\TT$-equivariant isomorphism of schemes
\[
M_{(\PP^1)^4}(\vec{r},n)^{(\C^*)^r} \cong  \bigsqcup_{\sum_{A,\alpha} n_{A,\alpha} = n} \prod_{A,\alpha} M_{\PP^1 \times \PP^1}(1,n_{A,\alpha}),
\]
where the disjoint union is over all decompositions $\sum_{A,\alpha} n_{A,\alpha} = n$. 
\end{proposition}
\begin{proof}
There is an obvious map from right to left by taking direct sums. We now describe the map from left to right. Let $[(E,\phi)] \in M_{(\PP^1)^4}(\vec{r},n)$ be $(\C^*)^r$-fixed. Then for every $w \in (\C^*)^r$, there exists an isomorphism $g_w \colon E \to E$ such that the diagram
\begin{displaymath}
\xymatrix
{
E \ar[r] \ar_{g_w}[d] & F_{\vec{r}} \ar^w[d] \\
E \ar[r] & F_{\vec{r}}
}
\end{displaymath}
commutes, where the horizontal arrows are the restriction $E \to E|_D$ composed with $\phi \colon E|_D \cong F_{\vec{r}}$. By Corollary \ref{cor: Hom vanishing for framed sheaves}, the natural map
\[
\Hom_X(E,E) \hookrightarrow \Hom_X(E,E|_D)
\]
is an injection, which implies that $g_w$ is uniquely determined by the diagram. In particular, for all $v,w \in (\C^*)^r$ one has $g_{vw} = g_v \circ g_w$. In other words, $E$ is $(\C^*)^r$-equivariant. In terms of $(\C^*)^r$-weights, there is a decomposition 
\[
F_{\vec{r}} = \bigoplus_{A,\alpha} \O_{\ell_A} \otimes \chi_{A,\alpha}, 
\]
where $\chi_{A,\alpha} \colon  (\C^*)^r \to \C^*$ is the character defined by $\chi_{A,\alpha}(w) = w_{A,\alpha}$. Consider the eigensheaf $E_{\chi_{A,\alpha}} \subset E$ of weight $\chi_{A,\alpha}$.  As the map $E \to E|_D \cong F_{\vec{r}}$ is equivariant, the pair $(E_{\chi_{A,\alpha}}, \phi|_{E_{\chi_{A,\alpha}}})$ is a framed sheaf and by Proposition \ref{prop:supportprop} its support is $S_A$. As $\ch_2(E) = \sum_{A} r_{A} [S_A]$, the sheaf $E_{\chi_{A,\alpha}}$ has rank 1 on its support and $E_{\chi} = 0$ for all $\chi \notin (\chi_{A,\alpha})$. Denoting $n_{A,\alpha} = -\ch_4(E_{\chi_{A,\alpha}})$ (which is $c_2(E_{\chi_{A,\alpha}})$ for $E_{\chi_{A,\alpha}}$ as a sheaf on $S_A$), we have specified an element
\[
[(E_{\chi_{A,\alpha}}, \phi|_{E_{\chi_{A,\alpha}}})] \in M_{\PP^1 \times \PP^1}(1,n_{A,\alpha}).
\]
It can be verified that the two maps described above are inverse to each other.

As in the proof of Proposition \ref{prop:fixlocquiver}, the arguments work in families over a base $\C$-scheme $S$ of finite type, which then extends the above bijection to an isomorphism of schemes by \cite[Thm.~2.3]{Fog}. This argument uses that for a family of framed sheaves $(\mathcal{E},\phi)$ as in Definition \ref{def:functfr} we still have an injection $\Hom_{X \times S}(\cE,\cE) \hookrightarrow \Hom_{X \times S}(\cE,\cE|_{D \times S})$. 
\end{proof}

From the equivariant isomorphism
\begin{equation} \label{eqn:MP1xP1toHilb}
M_{\PP^1 \times \PP^1}(1,n) \cong \Hilb^n(\C^2),
\end{equation}
we obtain the following two corollaries.
\begin{corollary}\label{cor: T-fixed loci framed sheaves}
    The $\mathbb{T}$-fixed locus $M_{(\PP^1)^4}(\vec{r},n)^{\mathbb{T}}$ consists of reduced points labeled by 
\[
\blambda = (\lambda_{A,\alpha})_{A\in\bsix,\,1\leq\alpha\leq r_A}, \quad |\blambda| = \sum_{A,\alpha} |\lambda_{A,\alpha}| = n,
\]
where each $\lambda_{A,\alpha}$ is an integer partition. 
\end{corollary}

\begin{corollary} \label{cor:isofixloc}
There exists an isomorphism of schemes 
\[
M_{Q_4}(\vec{r},n)^{(\C^*)^r} \cong M_{(\PP^1)^4}(\vec{r},n)^{(\C^*)^r}.
\]
\end{corollary}
\begin{proof}
Apply Propositions \ref{prop:fixlocquiver}, \ref{prop:fixlocframed}, and \eqref{eqn:MP1xP1toHilb}.
\end{proof}

Using the notation of Section \ref{sec:quivervirtualstruc}, we recall that for $P \in M_{Q_4}(\vec{r},n)^{\TT}$, we have
\begin{align*}
(E_{Q_4}^{\mdot})^{\vee}|_P &=  (T_{\cA}|_P - \Lambda|_P) + (T_{\cA}|_P - \Lambda|_P)^* \\
&=- \sum_{A} N_A N_{\overline{A}}^* t_{\overline{A}} + \mathsf{v}_{\blambda} + \mathsf{v}_{\blambda}^* \in K_0^{\TT}(\pt) 
\end{align*}
where $\cA := A_{\widehat{Q}_4}(\vec{r},n)$. 
\begin{proposition} \label{prop:Tvirmatch}
For any $P \in M_{Q_4}(\vec{r},n)^{\TT} \cong M_{(\PP^1)^4}(\vec{r},n)^{\TT}$, we have
\[
(E_{Q_4}^{\mdot})^{\vee}|_P = (E_{\mathrm{fr}}^{\mdot})^{\vee}|_P \in K_0^{\TT}(\pt).
\]
\end{proposition}
\begin{proof}
Using the relation $t_1t_2t_3t_4=1$, we have 
\begin{equation} \label{eqn:Pidentities}
P_{\phi(A)}+P_{\phi(A)}^*t_A^{-1}=P_{\overline{A}},\quad P_{\phi(A)}P_A+P_{\phi(A)}^*P_A^*=P_{1234},
\end{equation}
for all $A \in \bsix$. 

Let $\blambda$ be the collection of integer partitions corresponding to a fixed point $P$. From the definition of $\mathsf{v}_{\blambda}$ in \eqref{eqn:Neknot} and \eqref{eqn:Pidentities}, one has
\[
- \sum_{A} N_A N_{\overline{A}}^* t_{\overline{A}} + \mathsf{v}_{\blambda} + \mathsf{v}_{\blambda}^* = - \sum_{A} N_A N_{\overline{A}}^* t_{\overline{A}} + \sum_{A,B} \big(P_{\overline{A}} N_A K_B^* + P_{\overline{A}}^* N_A^* K_B - P_{1234} K_A K_B^*\big).
\]
Denote $X:=(\PP^1)^4$. Let $Z_{A,\alpha} \subset \C^2 \subset S_A \subset X$ be the 0-dimensional subscheme corresponding to $\lambda_{A,\alpha}$ and denote by $I_{Z_{A,\alpha}/S_A} \subset \O_{S_A}$ its ideal sheaf. 
It remains to show that $R\Hom_X(E_{\blambda},E_{\blambda}(-D))[1]$ is given by the same expression, where
\[
E_{\blambda} := \bigoplus_{A,\alpha} I_{Z_{A,\alpha} / S_A} \otimes w_{A,\alpha}
\]
and $I_{Z_{A,\alpha}/S_A}$ is viewed as a 2-dimensional sheaf on $X$.

Using the identity $I_{Z_{A,\alpha}} = \O_{S_A} - \O_{Z_{A,\alpha}}$ and Serre duality on $X$, this reduces to the calculation of
\[
R\Hom_X(\O_{S_A},\O_{S_B}(-D)), \ R\Hom_X(\O_{S_A}, \O_{Z_{B,\beta}}(-D)), \ R\Hom_X(\O_{Z_{A,\alpha}},\O_{Z_{B,\beta}}(-D)).
\]
The calculation of the last term follows from a rather standard argument involving the local-to-global spectral sequence and \v{C}ech cohomology as in \cite{MNOP}. The result is
\[
R\Hom_X(\O_{Z_{A,\alpha}},\O_{Z_{B,\beta}}(-D)) = \frac{P_{1234}}{t_1t_2t_3t_4} K_{A,\alpha}^* K_{B,\beta}.
\]
For the second term, define $\TT$-invariant divisors $\widetilde{D}_a = Z(z_a)$ and $D_a = Z(1/z_a)$. Then 
\begin{align}
\begin{split} \label{eqn:restrOD}
&D = D_1 + D_2 + D_3 + D_4, \\
&\O(-\widetilde{D}_a)  \cong \O(-D_a) \otimes t_a, \\
&\O(-D_a)|_{X \setminus D} \cong \O, \\
&D_{\overline{a}}|_{S_A} = D_{\overline{b}}|_{S_A} = 0,
\end{split}
\end{align}
for any $A = \{a<b\}$, where $\overline{A} = \{\overline{a} < \overline{b}\}$. Next, we use the $\TT$-equivariant resolution
\[
0 \to \O(-\widetilde{D}_{\overline{a}} - \widetilde{D}_{\overline{b}}) \to \O(-\widetilde{D}_{\overline{a}}) \oplus \O(-\widetilde{D}_{\overline{b}}) \to \O_X \to \O_{S_A} \to 0, 
\]
which, together with \eqref{eqn:restrOD}, yields
\begin{align*}
R\Hom_X(\O_{S_A}, \O_{Z_{B,\beta}}(-D)) = &R\Gamma(X,\O_{Z_{B,\beta}}(-D)) - R\Gamma(X,\O_{Z_{B,\beta}}(\widetilde{D}_{\overline{a}}-D)) - \\
&R\Gamma(X,\O_{Z_{B,\beta}}(\widetilde{D}_{\overline{b}}-D)) + R\Gamma(X,\O_{Z_{B,\beta}}(\widetilde{D}_{\overline{a}} + \widetilde{D}_{\overline{b}}-D)) \\
=&P_{\overline{A}}^* \cdot K_{B,\beta}.
\end{align*}
For the first term, recall that $$\chi(\O_{\PP^1 \times \PP^1}(-1,-1)) = \chi(\O_{\PP^1 \times \PP^1}(-1,0))  = \chi(\O_{\PP^1 \times \PP^1}(0,-1)) = 0,\quad \chi(\O_{\PP^1 \times \PP^1}) = 1,$$ all of which hold $(\C^*)^2$-equivariantly. Together with \eqref{eqn:restrOD}, we conclude that
\[
R\Hom_X(\O_{S_A},\O_{S_B}(-D)) = \delta_{A \pitchfork B} \cdot t_A,
\]
where we define
\[
\delta_{A \pitchfork B} := \left\{ \begin{array}{cc} 1 & \mathrm{if \, } B = \overline{A} \\ 0 & \mathrm{otherwise.} \end{array} \right.
\]
The desired formula for $R\Hom_X(E_{\blambda},E_{\blambda}(-D))[1]$ follows by summing all contributions.
\end{proof}

\begin{remark}
It is an interesting question whether the obstruction theory of Theorem \ref{thm: moduli of framed sheaves} on $\mathcal{M} := M_{(\mathbb{P}^1)^4}(\vec{r},n)$ admits an orientation, as required to construct a virtual cycle and structure sheaf  
$$
[\mathcal{M}]^{\mathrm{vir}} \in A_{\vd}^{\TT}(\mathcal{M}), \quad \vd = -\sum_{A \in \underline{\bfthree}} r_A r_{\overline{A}}, \quad \widehat{\O}_{\mathcal{M}}^{\vir} \in K_0^{\mathbb{T}}(\mathcal{M})_{\loc}.
$$

One could give a naive \emph{definition} of invariants by assuming the Oh--Thomas localization formula
$$
\int_{[\mathcal{M}]^{\vir}} 1 := \sum_{P \in \mathcal{M}^{\TT}} \pm \frac{1}{\sqrt{(-1)^{\vd}e(E_{\mathrm{fr}}\udot|_P^\vee)}}, \quad \chi(\mathcal{M}, \widehat{\mathcal{O}}_{\mathcal{M}}^{\vir}) := \sum_{P \in \mathcal{M}^{\TT}} \pm \frac{1}{\sqrt{(-1)^{\vd}\widehat{\Lambda}_{-1}(E_{\mathrm{fr}}\udot|_P)}}
$$
and picking some explicit choice of signs at each of the fixed points $P \in \mathcal{M}^{\TT}$ (the invariants depend on these sign choices). Of course, an arbitrary such choice of signs need not arise from a global orientation. Referring to Corollary \ref{cor:isofixloc}, one could pick the sign choice from Section \ref{sec:quivervirtualstruc} (as calculated in Section \ref{sec:proofmainthm}). Then obviously Proposition \ref{prop:Tvirmatch} would allow one to describe the origami partition function of Definition \ref{def:origampartfun} in terms of framed sheaves, i.e.
\begin{align*}
\sfZ_{\vec{r}}^K(q) = \frac{1}{\mathsf{C}^K_{\vec{r}}}\sum_{n=0}^\infty q^n \chi(M_{(\mathbb{P}^1)^4}(\vec{r},n), \widehat{\O}^{\vir}), \quad \mathsf{C}^K_{\vec{r}} = \chi(M_{(\mathbb{P}^1)^4}(\vec{r},0), \widehat{\O}^{\vir}),
\end{align*}
with a similar description for the cohomological generating series. This definition of the origami partition function could be properly justified by showing that there exists a $\mathbb{T}$-equivariant isomorphism 
$$
M_{Q_4}(\vec{r},n) \cong M_{(\mathbb{P}^1)^4}(\vec{r},n)
$$
compatible with the obstruction theories, and then using the orientation constructed on the left hand side as in Section \ref{sec:quivervirtualstruc}. The existence of such an equivariant isomorphism will be conjectured in the following section. 
\end{remark}

\subsection{Main conjecture} \label{sec:mainconj}

The following conjecture relates the moduli space of stable representations of the 4D ADHM quiver to the moduli space of 2-dimensional framed sheaves on $(\mathbb{P}^1)^4$. It can be viewed as a conjectural generalization of Barth's result \eqref{eqn:Barth} \cite{Bar}.
\begin{conjecture} \label{conj: main body main conjecture}
For any $(\vec{r},n)$, there exists a $\TT$-equivariant isomorphism
$$M_{Q_4}(\vec{r},n) \cong M_{(\PP^1)^4}(\vec{r},n).$$ 
Moreover, under this isomorphism there exists a $\TT$-equivariant commutative diagram
\begin{displaymath}
\xymatrix
{
E_{Q_4}\udot \ar^{\cong}[r] \ar[d] & E_{\fr}\udot \ar[dl] \\
\tau^{\geq -1} \LL_{M_{Q_4}(\vec{r},n)}
}
\end{displaymath}
where the top isomorphism intertwines the symmetric pairings $\theta_{Q_4}$, $\theta_{\fr}$.
\end{conjecture}

Supporting evidence for this conjecture includes the isomorphism on the level of $(\C^*)^r$-fixed loci in Corollary \ref{cor:isofixloc} and the comparison of the virtual tangent characters at $\mathbb{T}$-fixed points in Proposition \ref{prop:Tvirmatch}. In what follows, we establish a stronger version of Corollary \ref{cor:isofixloc}. 

Recall from Proposition \ref{prop: Quot scheme} that there exists a closed subscheme $\widetilde{M}_{Q_4}(\vec{r},n)\hookrightarrow M_{Q_4}(\vec{r},n)$, cut out by $J_A = 0$, which is isomorphic to $\Quot_{\C^4}\Big(\bigoplus_{A}\O_{\C^2_A}^{\oplus r_A}, n\Big)$. The conjecture suggests that this Quot scheme should also lie in the framed moduli space $M_{(\mathbb{P}^1)^4}(\vec{r},n)$. Indeed, there exists a morphism $$f\colon\Quot_{\C^4}\Big(\bigoplus_{A}\O_{\C^2_A}^{\oplus r_A}, n\Big)\rightarrow M_{(\mathbb{P}^1)^4}(\vec{r},n)$$ that sends a quotient\footnote{Let $j\colon \C^4\rightarrow (\mathbb{P}^1)^4$ be the open embedding of the complement of $D$. Then we may identify the quotient of $\bigoplus_A \O_{\C^2_A}^{\oplus r_A}=j^*(\bigoplus_A \O_{S_A}^{\oplus r_A})\rightarrow Q$ on $\C^4$ with the adjoint quotient $\bigoplus_A \O_{S_A}^{\oplus r_A}\rightarrow j_*Q$ on $(\mathbb{P}^1)^4$. }
$$0\rightarrow E\rightarrow \bigoplus_{A}\O_{S_A}^{\oplus r_A}\rightarrow Q\rightarrow 0
$$
with $\textnormal{Supp}(Q)\subset \C^4 = (\PP^1)^4 \setminus D$ to a framed sheaf $(E, \phi)$ where $\phi$ is the composition 
$$E\subseteq \bigoplus_{A}\O_{S_A}^{\oplus r_A}\rightarrow \bigoplus_{A}\O_{S_A}^{\oplus r_A}\Big|_D\simeq \bigoplus_{A}\O_{\ell_A}^{\oplus r_A}.
$$
Note that this indeed defines a framed sheaf because $E$ and $\bigoplus_{A}\O_{S_A}^{\oplus r_A}$ agree on the divisor $D$.

\begin{proposition}\label{prop:Quot in framed moduli}
The morphism $f$ is a closed embedding. Moreover $f$ fits into a commutative diagram 
\begin{center}
\begin{tikzcd}
M_{Q_4}(\vec{r},n)^{(\C^*)^r} \arrow[r, hook] \arrow[d, "\rotatebox{90}{$\sim$} \textnormal{ Corollary \ref{cor:isofixloc}}", no head] & \widetilde{M}_{Q_4}(\vec{r},n) \arrow[r, hook] \arrow[d, "\rotatebox{90}{$\sim$} \textnormal{ Proposition \ref{prop: Quot scheme}}", no head] & M_{Q_4}(\vec{r},n) \\
M_{(\PP^1)^4}(\vec{r},n)^{(\C^*)^r} \arrow[r, hook]        &\Quot_{\C^4}\Big(\bigoplus_{A}\O_{\C^2_A}^{\oplus r_A}, n\Big)  \arrow[r, hook]                              & M_{(\mathbb{P}^1)^4}(\vec{r},n).
\end{tikzcd}
\end{center}
\end{proposition}
\begin{proof}
    Once we show that $f$ is a closed embedding, it is straightforward to check that both $\widetilde{M}_{Q_4}(\vec{r},n)$ and $\Quot_{\C^4}\Big(\bigoplus_{A}\O_{\C^2_A}^{\oplus r_A}, n\Big)$ contain the $(\C^*)^r$-fixed locus and that the isomorphism from Proposition \ref{prop: Quot scheme} is compatible with the one from Corollary \ref{cor:isofixloc}. In order to prove that the morphism $f$ is a closed embedding, it suffices to show that it is a proper monomorphism \cite[\href{https://stacks.math.columbia.edu/tag/04XV}{Tag 04XV}]{Sta}.

    We first show that $f$ is a monomorphism, i.e., it defines an injective map between $T$-points, where $T$ is any base $\C$-scheme of finite type. For notational simplicity, set $E_{\vec{r}}:=\bigoplus_{A}\O_{S_A}^{\oplus r_A}$ and $F_{\vec{r}}:=E_{\vec{r}}\big|_D=\bigoplus_{A}\O_{\ell_A}^{\oplus r_A}$. Consider two $T$-valued points of the Quot scheme  
    \begin{equation} \label{eqn:Quotses}
    0\rightarrow \cE_i\xrightarrow{\iota_i} E_{\vec{r}}\boxtimes\O_T\rightarrow\cQ_i\rightarrow 0, \quad i=1,2.
    \end{equation}
    Assume that they define the same $T$-family of framed sheaves, that is, that there exists an isomorphism $\cE_1\xrightarrow{\phi}\cE_2$ such that the following diagram commutes
    \begin{equation}\label{eqn: random diagram}
        \begin{tikzcd}
\cE_1 \arrow[d, "\phi"] \arrow[r, "\iota_1"] &  E_{\vec{r}}\boxtimes\O_T\arrow[r, two heads] & F_{\vec{r}}\boxtimes\O_T \arrow[d, equal] \\
\cE_2 \arrow[r, "\iota_2"]                & E_{\vec{r}}\boxtimes\O_T \arrow[r, two heads] & F_{\vec{r}}\boxtimes\O_T.                     
\end{tikzcd}
    \end{equation}
    Let $X := (\PP^1)^4$ and denote by $p \colon X \times T \to T$ the projection. The short exact sequence \eqref{eqn:Quotses} induces a long exact sequence
    \begin{align*}
    0 &\to p_* \hom(\cQ_i, E_{\vec{r}} \boxtimes \O_T) \to \mathrm{End}_X(E_{\vec{r}}) \otimes \O_T \to p_* \hom(\cE_i,E_{\vec{r}} \boxtimes \O_T) \\
    &\to \ext_p^1(\cQ_i, E_{\vec{r}} \boxtimes \O_T) \to \cdots.
    \end{align*}
    For any closed point $t \in T$, we claim that
    \[
    \Hom_X(\cQ_{i,t},E_{\vec{r}}) = 0, \quad \Ext^1_X(\cQ_{i,t},E_{\vec{r}}) = 0.
    \]
    The first vanishing follows from the fact that $\cQ_{i,t}$ is 0-dimensional and $E_{\vec{r}}$ is pure 2-dimensional. The last term $\Ext^1_X(\cQ_{i,t},E_{\vec{r}})\simeq \Ext^3_X(E_{\vec{r}},\cQ_{i,t}\otimes K_X)^*$ also vanishes for the following reason. The terms of the spectral sequence $$H^p(X,\mathcal{E}xt^q(E_{\vec{r}},\cQ_{i,t}\otimes K_X)) \Rightarrow \Ext^{p+q}_X(E_{\vec{r}}, \mathcal{Q}_{i,t} \otimes K_X)$$  are zero for $p>0$ because the local Ext sheaves are 0-dimensional. Moreover, $E_{\vec{r}}$ has homological dimension $2$ which implies that $\mathcal{E}xt^3(E_{\vec{r}},\cQ_{i,t}\otimes K_X) = 0$. As $\cE_i$ and $E_{\vec{r}} \boxtimes \O_T$ are $T$-flat, we deduce that
    \[
    p_* \hom(\cQ_i, E_{\vec{r}} \boxtimes \O_T) = 0, \quad \ext_p^1(\cQ_i, E_{\vec{r}}\boxtimes \O_T) = 0,
    \]
    by \cite[Kor.~1]{BPS} and thus that $\mathrm{End}_X(E_{\vec{r}}) \otimes \O_T \cong p_* \hom(\cE_i,E_{\vec{r}} \boxtimes \O_T)$. Taking $i=1$ and global sections, we see that $\iota_2 \circ \phi = \widetilde{\phi} \circ \iota_1$ for some $\widetilde{\phi} \in \Gamma(T,\O_T) \otimes_{\C} \mathrm{End}_X(E_{\vec{r}})$. 
    
     Restricting \eqref{eqn: random diagram} to $D\times T\hookrightarrow X \times T$, we obtain 
    \begin{equation*}
        \begin{tikzcd}
\cE_1\big|_{D\times T} \arrow[d, "\phi|_{D\times T}"] \arrow[r] &  E_{\vec{r}}\big|_D\boxtimes\O_T\arrow[r, equal] \arrow[d, "\widetilde{\phi}|_{D\times T}"] & F_{\vec{r}}\boxtimes\O_T \arrow[d, equal] \\
\cE_2\big|_{D\times T} \arrow[r]                & E_{\vec{r}}\big|_D\boxtimes\O_T \arrow[r, equal] & F_{\vec{r}}\boxtimes\O_T,              
\end{tikzcd}
\end{equation*}
    where the left square and big square commute. As all horizontal arrows are now isomorphisms, the right square also commutes. In particular $\widetilde{\phi}|_{D\times T}$ is the identity map.
    Moreover, the restriction map
    \[
    \mathrm{End}_X(E_{\vec{r}}) \otimes_{\C} \Gamma(T,\O_T) \to \mathrm{End}_D(F_{\vec{r}}) \otimes_{\C} \Gamma(T,\O_T)
    \]
    is easily seen to be an isomorphism that sends the identity to the identity. It follows that $\widetilde{\phi} = \id$. We conclude that $\iota_2 \circ \phi = \iota_1$ and the two $T$-valued points of the Quot scheme are the same. 

    We now prove that $f$ is proper. It suffices to show that there is a Cartesian square
    \begin{equation*}
\begin{tikzcd}
\Quot_{\C^4}\Big(\bigoplus_{A}\O_{\C^2_A}^{\oplus r_A}, n\Big) \arrow[r, "f"] \arrow[d, phantom, "\subseteq" sloped] & M_{(\mathbb{P}^1)^4}(\vec{r},n) \arrow[d, phantom, "\subseteq" sloped] \\
\Quot_{(\mathbb{P}^1)^4}\Big(\bigoplus_{A}\O_{S_A}^{\oplus r_A}, n\Big) \arrow[r, "\bar{f}", dashed] & M_{((\mathbb{P}^1)^4,i_*F_{\vec{r}})}^{\delta(x)\textrm{-}st}(v_{\vec{r},n})
\end{tikzcd}
\end{equation*}
where $\delta(x)=\delta_1x+\delta_0$ is an HL-stability condition for some $0<\delta_1 < a_1(F)$ (Proposition \ref{prop: framed sheaves are delta-stable}). Once this diagram is constructed, properness of $f$ follows from projectivity of the lower left term and separatedness of the lower right term. Let 
$$0\rightarrow E\rightarrow E_{\vec{r}}\rightarrow Q\rightarrow 0$$
be a point in the Quot scheme of $(\mathbb{P}^1)^4$. We would like to show that $(E,\phi)$ is $\delta(x)$-stable where $\phi$ is defined as the composition
$$E\subseteq E_{\vec{r}}\twoheadrightarrow E_{\vec{r}}\big|_D=i_*F_{\vec{r}}. 
$$
In particular, we have $\epsilon_{\phi} = 1$. Once $\delta(x)$-stability is established, the morphism $\overline{f}$ is well-defined and it can be checked that the resulting diagram is Cartesian. 

Let $0\subsetneq E'\subsetneq E$ be any proper nonzero subsheaf. Set $\phi'$ to be the composition $E'\subset E\xrightarrow{\phi}i_*F_{\vec{r}}$. In order to prove $\delta(x)$-stability of $(E,\phi)$, we have to show that
\begin{equation}\label{eqn: necessary inequality}
    a_2(E)(a_1(E')-\epsilon_{\phi'}\cdot\delta_1) \leq a_2(E')(a_1(E)-\epsilon_{\phi}\cdot\delta_1)
\end{equation}
and, whenever equality occurs, we have 
\begin{equation} \label{eqn: ineq cst}
    a_2(E)(a_0(E')-\epsilon_{\phi'}\cdot\delta_0) < a_2(E')(a_0(E)-\epsilon_{\phi}\cdot\delta_0).
\end{equation}
We divide into two cases depending on the value of  $\epsilon_{\phi'}$. First, assume that $\epsilon_{\phi'}=1$, i.e., $\phi'\neq 0$. As $E_{\vec{r}}$ is Gieseker $D$-semistable and $E$ and $E'$ agree away from the 0-dimensional support of $Q$, we have 
$$a_2(E)a_1(E')\leq a_2(E')a_1(E).
$$
When $a_2(E') < a_2(E)$, we deduce that \eqref{eqn: necessary inequality} holds with strict inequality. When $a_2(E') = a_2(E)$, we only have to deal with the case $a_1(E') = a_1(E)$. But then $a_0(E') < a_0(E)$ and \eqref{eqn: ineq cst} holds.

Now, we assume that $\epsilon_{\phi'}=0$, i.e., $\phi'=0$. In this case, we have $E'\subseteq E_{\vec{r}}(-D)$. Gieseker $D$-semistability of $E_{\vec{r}}(-D)$ then implies that
$$a_2(E)a_1(E')=a_2(E(-D))a_1(E')\leq a_2(E')a_1(E(-D)) = a_2(E')(a_1(E) - D^2 \ch_2(E)).
$$
As $D$ is ample, we can (uniformly) choose $\delta_1$ sufficiently small so that \eqref{eqn: necessary inequality} holds with strict inequality.
\end{proof}

\begin{example}
Consider the case $r_{12} = r_{34} = 1$. Then $\widetilde{M}_{Q_4}(\vec{r},n) \cong M_{Q_4}(\vec{r},n)$ by Proposition \ref{prop:Jiszero}. Denote by $M_{(\PP^1)^4}(\vec{r},n)^{\mathrm{main}}$ the connected component containing the connected closed subscheme
\begin{equation} \label{n=1iso}
\widetilde{M}_{Q_4}(\vec{r},n) \cong \Quot_{\C^4}\Big(\bigoplus_{A}\O_{\C^2_A}^{\oplus r_A}, n\Big),
\end{equation}
where we use Propositions \ref{prop: Quot scheme}, \ref{prop: connectedness of Quot scheme}, and \ref{prop:Quot in framed moduli}. Thus, we have a closed embedding
$$
f \colon  M_{Q_4}(\vec{r},n) \hookrightarrow M_{(\PP^1)^4}(\vec{r},n)^\main.
$$
By the inverse function theorem for varieties \cite[Thm.~14.9]{Har}, it follows that if $d_P f$ is surjective at all closed points $P$, then $f$ induces an isomorphism of the underlying reduced varieties. 

Taking $n=1$, $M_{Q_4}(\vec{r},1)$ is reduced and is as described in Example \ref{ex:butterfly n=1}. The elements in the image of $f$ are characterized as follows. Recall that the Quot scheme locus parametrizes quotients of $\O_{S_{12}} \oplus \O_{S_{34}}$ of the form
\begin{equation}\label{eqn: butterfly n=1 quotient}
    0\rightarrow E\rightarrow \O_{S_{12}}\oplus \O_{S_{34}}\rightarrow Q\rightarrow 0,
\end{equation}
where $Q$ is a skyscraper sheaf supported at a point $p \in \C^2_{12}\cup \C^2_{34}$, where $\C_A^2 = S_A \setminus D$. If $p \in \C^2_{12}\backslash\{o\}$ (resp.~$p\in \C^2_{34}\backslash\{o\}$), then $E$ is isomorphic to $I_{p\in S_{12}}\oplus \O_{S_{34}}$ (resp.~$\O_{S_{12}}\oplus I_{p\in S_{34}}$), where $I_{p \in S_A} \subset \O_{S_A}$ denote the corresponding ideal sheaves viewed as sheaves on $(\PP^1)^4$. If $p=o$, then defining the quotient \eqref{eqn: butterfly n=1 quotient} amounts to choosing a 1-dimensional quotient of $\O_{S_{12}}|_o\oplus\O_{S_{34}}|_o$. Therefore, such points are parametrized by a pencil $\mathbb{P}^1_{[s:t]}$. Under the isomorphism \eqref{n=1iso}, these quotients correspond to points of the component (iii) in Example \ref{ex:butterfly n=1}. For $[1:0]\in \mathbb{P}^1$ (resp.~$[0:1]\in \mathbb{P}^1$), the subsheaf $E$ is isomorphic to $I_{o\in S_{12}}\oplus \O_{S_{34}}$ (resp.~$\O_{S_{12}}\oplus I_{o\in S_{34}}$). For $[s:t]$ with both $s$ and $t$ nonzero, the subsheaf $E$ is isomorphic to $\O_{S_{12}\cup S_{34}}$. The dimensions of the Zariski tangent spaces at points in the image of \eqref{n=1iso} can be computed using standard (but lengthy) Ext group calculations. Up to symmetry, they are given by
$$
\dim \Ext^1_{(\PP^1)^4}(E, E(-D))=
\begin{cases}
    2&\textnormal{if}\quad E= I_{p\in S_{12}}\oplus \O_{S_{34}}\ \textnormal{for } p\in \C^2_{12}\backslash\{o\},\\
    3&\textnormal{if}\quad E= I_{o\in S_{12}}\oplus \O_{S_{34}},\\
    1&\textnormal{if}\quad E= \O_{S_{12}\cup S_{34}}.\\
\end{cases}
$$
These calculations imply the following additional evidence for Conjecture \ref{conj: main body main conjecture}
$$
M_{Q_4}(\vec{r},1) \cong M_{(\PP^1)^4}(\vec{r},1)^\main_{\red}.
$$
\end{example}

\hfill

\end{document}